\documentclass[11pt]{amsart}

\usepackage{amsmath,latexsym,amssymb,amsthm,amsfonts,graphicx}
\usepackage{subfigure}
\usepackage{xfrac}
\usepackage{fix-cm}
\usepackage{hyperref,doi}
\usepackage{url}
\usepackage{stmaryrd}
\usepackage{multirow}
\usepackage{tabularx}
\usepackage{tikz-cd}
\usepackage[inline]{enumitem}
\usepackage{wrapfig}
\usepackage[title]{appendix}
\usepackage{color}
\usepackage{parskip}
\usepackage{dsfont}

\usepackage[margin=1.25in]{geometry}

\setenumerate[1]{label=(\thesection.\arabic*)}

\numberwithin{equation}{section}
\numberwithin{figure}{section}

\newtheorem{theorem}{Theorem}[section]

\newtheorem{lemma}[theorem]{Lemma}

\newtheorem{corollary}[theorem]{Corollary}

\theoremstyle{definition}
\newtheorem{remark}[theorem]{Remark}
\newtheorem{remarks}[theorem]{Remarks}
\newtheorem{example}[theorem]{Example}
\newtheorem{examples}[theorem]{Examples}

\newtheorem{notation}[theorem]{Notation}

\newtheorem{conjecture}{Conjecture}

\def\Z{\ensuremath{\mathbb{Z}}}

\def\R{\ensuremath{\mathbb{R}}}
\def\C{\ensuremath{\mathbb{C}}}
\def\H{\ensuremath{\mathbb{H}}}

\newcommand{\deffont}[1]{\textbf{#1}}
\newcommand{\pa}[1]{\left(#1\right)}
\newcommand{\cpa}[1]{\left\{#1\right\}}
\newcommand{\wt}[1]{\widetilde{#1}}
\newcommand{\tn}[1]{\textnormal{#1}}
\newcommand{\br}[1]{\left[#1\right]}
\newcommand{\fg}[1]{\left\langle #1\right\rangle}
\newcommand{\Int}[2]{\tn{Int}_{#2} #1}
\newcommand{\Cl}[2]{\tn{Cl}_{#2} #1}
\newcommand{\Fr}[2]{\tn{Fr}_{#2} #1}
\newcommand{\interior}[1]{#1^{\circ}}
\newcommand{\closure}[1]{\overline{#1}}
\newcommand{\im}[1]{\tn{Im} \, #1}

\newcommand{\cokernel}[1]{\tn{Coker} \, #1}
\newcommand{\myspan}[1]{\tn{span} \, #1}
\newcommand{\rank}[1]{\tn{rank} \, #1}
\newcommand{\rest}[1]{\left.#1\right|}
\newcommand{\card}[1]{\left| #1 \right|}

\newcommand{\eqc}[1]{\llbracket #1 \rrbracket}

\newcommand{\e}{\varepsilon}

\newcommand{\eds}{\oplus}
\newcommand{\bc}[1]{\mathcal{B}\pa{#1}}
\newcommand{\uc}[1]{\mathcal{U}\pa{#1}}

\newcommand*{\twoheadrightarrowtail}{\mathrel{\rightarrowtail\kern-1.9ex\twoheadrightarrow}}

\newcommand{\nnr}{[0,\infty)}

\newcommand{\es}{\mathbin{\natural}}

\newcommand{\dlim}{\varinjlim}
\newcommand{\ilim}{\varprojlim}

\newcommand{\zco}[2]{H^{#1}\pa{#2}}
\newcommand{\rzco}[2]{\wt{H}^{#1}\pa{#2}}
\newcommand{\co}[3]{H^{#1}\pa{#2;#3}}

\newcommand{\zho}[2]{H_{#1}\pa{#2}}

\newcommand{\zeco}[2]{H^{#1}_{\tn{e}}\pa{#2}}
\newcommand{\rzeco}[2]{\wt{H}^{#1}_{\tn{e}}\pa{#2}}
\newcommand{\eco}[3]{H^{#1}_{\tn{e}}\pa{#2;#3}}

\renewcommand{\hom}[3]{\tn{Hom}_{#1}\pa{#2,#3}}

\font\cuf=cmtt8
\newcommand{\curl}[1]{{\cuf #1}}

\makeatletter
\@namedef{subjclassname@2020}{%
  \textup{2020} Mathematics Subject Classification}
\makeatother

\title{Ends and end cohomology}

\author[W.~Bass]{William G. Bass}
\address{Department of Mathematics, Oberlin College, Oberlin, OH 44074}
\email{wbass@oberlin.edu}

\author[J.~Calcut]{Jack S. Calcut}
\address{Department of Mathematics, Oberlin College, Oberlin, OH 44074}
\email{jcalcut@oberlin.edu}

\date{\today}

\begin{document}

\begin{abstract}
Ends and end cohomology are powerful invariants for the study of noncompact spaces.
We present a self-contained exposition of the topological theory of ends
and prove novel extensions including the existence of an exhaustion of a proper map.
We define reduced end cohomology as the relative end cohomology of a ray-based space.
We use those results to prove a version of a theorem of King that computes the reduced end cohomology
of an end sum of two manifolds.
We include a complete proof of Freudenthal's fundamental theorem on the number of ends of a topological group,
and we use our results on dimension-zero end cohomology to prove---without using transfinite induction---a theorem
of N\"obeling on freeness of certain modules of continuous functions.
\end{abstract}

\maketitle

\tableofcontents

\section{Introduction}
\label{sec:introduction}

In the study of noncompact spaces, ends play a fundamental role.
The notion of an end is quite intuitive---the real line has two ends and the real plane has one end.
Freudenthal~\cite{freudenthal31} put the theory of ends on a firm foundation and proved
his groundbreaking result that each path connected topological group has at most two ends.
Since that time, Hopf~\cite{hopf}, Stallings~\cite{stallings}, and others
have significantly extended Freudenthal's ideas to prove substantial topological, algebraic, and geometric results.

Several topological invariants have useful adaptations to the noncompact setting including
homotopy and homology groups and cohomology rings.
Proper maps play a central role, and adaptations often involve a direct or an inverse limit---see
Hughes and Ranicki~\cite[Ch.~1--3]{hr}, Geoghegan~\cite[Ch.~11--16]{geoghegan}, and Guilbault~\cite{guilbault}.

The classical connected sum of compact surfaces also has an analogue for noncompact manifolds---\textit{end sum}---introduced
by Gompf~\cite{gompf83,gompf85}.
One glues together two noncompact manifolds guided by a proper ray in each manifold.
That operation is now a major tool for constructing interesting manifolds.
For recent examples, see Bennett~\cite{bennett}, Sparks~\cite{sparks}, Gompf and the second author~\cite{cg},
and Guilbault, Haggerty, and the second author~\cite{cgh}.
End sum may also be used to prove the hyperplane unknotting theorem:
each proper embedding of $\R^{n-1}$ in $\R^n$ is unknotted provided $n\ne3$~\cite{cks,cg}.

The dependence of end sum on ray choice is an immediate question.
Even in $\R^n$, the question of whether there is a \textit{choice} of ray is interesting.
Fox and Artin~\cite[p.~983]{foxartin} exhibited the first knotted ray in $\R^3$.
Gompf~\cite[Lemma~A.1]{gompf85} showed using finger moves that each ray in $\R^n$ is unknotted for $n\geq 4$.
King, Siebenmann, and the second author~\cite[Thm.~9.13]{cks} used embedded Morse theory to show that each ray unknots in $\R^2$.
Thus, a ray knots in $\R^n$ if and only if $n=3$, and the end sum of two copies of $\R^n$ is diffeomorphic to $\R^n$ provided $n\ne 3$.
Myers~\cite{myers} showed that end sums of two copies of $\R^3$ yield uncountably many topological types---see also~\cite[App.]{ch}.

In dimensions greater than $3$, Siebenmann conjectured~\cite[pp.~1804--1805]{cks} that ray choice
could alter the topological type of the end sum of two open, one-ended manifolds.
Haggerty and the second author~\cite{ch} constructed examples verifying Siebenmann's conjecture.
Guilbault, Haggerty, and the second author~\cite{cgh} later produced many more examples.
All of those examples were distinguished using their end cohomology algebras.
King (unpublished) suggested a theorem for computing the end cohomology algebra of an end sum in terms
of the end cohomology algebras of the summands~\cite[$\S$5]{cgh}.
That theorem relies on \textit{reduced} end cohomology---which lacks a precise definition in the literature.
The main goals of the present paper are to provide a precise definition of reduced end cohomology,
to use that definition to state and prove King's theorem, and to include all supporting topological background.
We have intentionally included complete proofs of that background material since some of it is hard to find in the literature.

Reduced cohomology is sometimes regarded merely as a tool for streamlining statements
such as the cohomology of a point, a sphere, a wedge sum, and a connected sum of manifolds.
More significantly, the suspension axiom holds in all dimensions for reduced cohomology, and
homotopical cohomology yields reduced theories---see
May~\cite[Ch.~19 \& 22]{maycc}.
Various constructions lead to reduced cohomology $\rzco{\ast}{X}$ of a space $X$:
augmentation, relative cohomology of based spaces, and homotopical cohomology.
We compare those approaches in Section~\ref{sec:reducedcohomology}.
They yield canonically isomorphic results when $X$ is path-connected---but not otherwise.
That difference is relevant to the theory of ends where disconnected spaces are ubiquitous.
All three approaches yield splittings
\begin{equation}\label{reducedsplitting}
\zco{\ast}{X} \cong \rzco{\ast}{X} \eds \zco{\ast}{\bullet}
\end{equation}
where $\cpa{\bullet}$ is a one-point space.
While augmentation has the apparent advantage of avoiding a choice of basepoint,
the resulting splitting~\eqref{reducedsplitting} is not natural in $X$ and cup products with dimension-zero classes are not well-defined.
If instead one chooses a basepoint and defines $\rzco{\ast}{X}=\zco{\ast}{X,\bullet}$,
then the splitting~\eqref{reducedsplitting} is natural in $X$ for based maps,
all cup products are well-defined, and $\rzco{\ast}{X}$ is a graded, associative algebra.
Therefore, the definition using a basepoint is preferred.

Given a noncompact space $X$ with an exhaustion by compacta, the ends of $X$ are defined using the complements of those compacta.
The end cohomology algebra $\zeco{\ast}{X}$ of $X$ is defined to be the direct limit of the cohomology algebras of those complements.
A \deffont{baseray} in $X$ is a proper embedding $r:\nnr\to X$.
By analogy with the preferred definition of reduced cohomology,
we define---see~\eqref{eq:recdef}---the \deffont{reduced end cohomology algebra} $\rzeco{\ast}{X}$ of $X$ to be the relative end cohomology algebra
of the ray-based space $(X,r)$.
In Theorem~\ref{mainrecotheorem}, we obtain a splitting
\begin{equation}\label{recosplitting}
\zeco{\ast}{X} \cong \rzeco{\ast}{X} \eds \zeco{\ast}{\nnr}
\end{equation}
that is natural in $X$ for ray-based proper maps.
To obtain that splitting, we prove in Theorem~\ref{retract} that $X$ properly retracts to any baseray.

We study dimension-zero end cohomology in detail.
Using coefficients in a PID $R$---equipped with the discrete topology---we prove in Lemma~\ref{dimzeroendcoho}
that $\zeco{0}{X} \cong C(E(X),R)$
where $C(E(X),R)$ denotes the $R$-module of continuous functions from the end space of $X$ to $R$.
In Theorem~\ref{recdirect}, we give a very direct proof that $\zeco{0}{X}$ is a free $R$-module
of countable rank equal to the number of ends of $X$ (where infinities are not distinguished).
To make the algebra in the proof of Theorem~\ref{recdirect} as simple as possible, we introduce in Section~\ref{sec:cem}
the notion of an \textit{exhaustion of a map} and prove its existence for any proper map.

Lemma~\ref{dimzeroendcoho} and Theorem~\ref{recdirect} combine to prove a theorem of N\"obeling:
if $A$ is a profinite space---the limit
of an inverse system of finite sets---then $C(A,\Z)$ is a free $\Z$-module.
Our proof of N\"obeling's theorem in Remarks~\ref{ecmstR} relies on the theory of ends only for
locally finite trees and does not use transfinite induction.

Our development of reduced end cohomology requires several topological results---some novel---including the existence of:
compact exhaustions of spaces and maps, end spaces,
baserays, and retracts to baserays.
Those results ultimately require some niceness conditions on spaces.
The main class of spaces with which we work---generalized continua---includes manifolds, locally finite simplicial complexes,
locally finite CW-complexes, and many ANRs.
We have included complete proofs of those results for generalized continua---sometimes assumed metrizable---for the benefit of the reader,
to make this paper as self-contained as possible, since we are unaware of complete published proofs of some background material,
and to advertize certain notions such as Guilbault's \textit{efficient} compact exhaustions.

This paper is organized as follows.
Section~\ref{sec:ces} introduces generalized continua and proves that each generalized continuum admits an efficient exhaustion by compacta.
Section~\ref{sec:es} develops the endpoint compactification of a generalized continuum.
We include a complete proof of Freudenthal's theorem: each path connected generalized continuum that is a topological group has at most two ends.
Section~\ref{sec:cem} introduces our notion of a compact exhaustion of a proper map and
proves that each proper map of generalized continua has an efficient compact exhaustion.
Section~\ref{sec:br} studies baserays in generalized continua,
including an example of a generalized continuum that contains no baseray.
We prove that each \textit{metrizable} generalized continuum contains a baseray pointing to any specified end.
Also, we prove that there exists a proper retract to a given baseray.
Section~\ref{sec:reducedcohomology} compares three well-known definitions of reduced ordinary cohomology.
In Section~\ref{sec:endcohomology}, we review end cohomology and define reduced end cohomology.
We prove a natural splitting for reduced end cohomology, study dimension-zero reduced end cohomology in detail,
and use those results to prove N\"obeling's theorem.
Section~\ref{sec:king} recalls the notion of an end sum of manifolds.
Using our definition of reduced end cohomology, we state and prove a version of King's theorem that
computes the end cohomology of an end sum.
Section~\ref{sec:fd} closes with several questions for further study.

Throughout, we use the following conventions.
\deffont{Spaces} are Hausdorff topological spaces.
As usual, $I=[0,1]\subseteq\R$ denotes the closed unit interval.
A \deffont{compactum} is a compact space---not necessarily metric.
A \deffont{map} is a continuous function.
A map is \deffont{proper} provided the preimage of each compactum is compact.
A \deffont{proper homotopy} is a homotopy that is a proper map.
Maps are \deffont{properly homotopic} provided they are homotopic by a proper homotopy;
properly homotopic maps are necessarily proper.
Let $X$ be a space, and let $A\subseteq X$ be a subspace of $X$.
Then, $\interior{A}$ denotes the topological \deffont{interior} of $A$ in $X$,
$\closure{A}$ denotes the topological \deffont{closure} of $A$ in $X$,
and $\Fr{A}{}$ denotes the topological \deffont{frontier} of $A$ in $X$.
If the ambient space $X$ must be specified, then we instead write
$\Int{A}{X}$, $\Cl{A}{X}$, and $\Fr{A}{X}$ for those subspaces.
We say $A$ is \deffont{bounded} provided $\closure{A}$ is compact,
and $A$ is \deffont{unbounded} provided $\closure{A}$ is noncompact.
\deffont{Component} means connected component.
We adopt the convention that each \deffont{connected} space has exactly one nonempty component.
In particular, connected implies nonempty,
and the empty space is neither connected nor disconnected.
The notation $Y\rightarrowtail Z$ denotes an injective function and
$Y\twoheadrightarrow Z$ denotes a surjective function.
If $Y\subseteq Z$, then $Y\hookrightarrow Z$ denotes inclusion.
We write $Y\approx Z$ to mean that $Y$ and $Z$ are homeomorphic.
We use the singular theory for (co)homology and---following Hatcher~\cite[p.~212]{hatcher}---we refer to the integer $k$
in any homology group $H_k$ or cohomology group $H^k$ as the \deffont{dimension} rather than the degree.

\section*{Acknowledgement}
We thank Craig Guilbault for the inspiring lecture series on ends he gave
at Oberlin College in April 2018 and for helpful conversations.
We also thank the referee for their careful reading and useful comments.

\section{Compact exhaustions of spaces}
\label{sec:ces}

Let $X$ be a space.
An \deffont{exhaustion of $X$ by compacta}---also called a \deffont{compact exhaustion} of $X$---is a countably indexed, nested sequence
$K_1 \subseteq K_2 \subseteq \cdots$ of compacta in $X$ such that:
\begin{enumerate*}[label=(\roman*)]
  \item\label{coverX} $X=\cup_i K_i$---meaning the $K_i$ cover $X$---and
  \item\label{inint} $K_i\subseteq \interior{K_{i+1}}$ for each $i\in\Z_+$.
\end{enumerate*}
Those properties ensure that for each compactum $K\subseteq X$ there exists $i\in\Z_+$ such that $K\subseteq \interior{K_i}$.
Using Guilbault's terminology~\cite[$\S$3.3.1]{guilbault},
a compactum $K\subseteq X$ is \deffont{efficient} provided $K$ is connected (hence, nonempty) and
each component of $X-K$ is unbounded.
A compact exhaustion $\cpa{K_i}$ of $X$ is \deffont{efficient} provided each $K_i$ is efficient.

\begin{examples}\label{eceexamples}\mbox{ \\ }
\begin{enumerate}[label=(\alph*)]
\item\label{ecennr} Each efficient compact exhaustion of $\nnr$ has the form $K_i=[0,b_i]$ for some increasing, unbounded sequence
$0\leq b_1<b_2<b_3<\cdots$ of real numbers.
\item Euclidean space $\R^n$ is efficiently exhausted by closed disks of radius $i\in\Z_+$ centered at the origin.
\item Each subsequence of an (efficient) compact exhaustion is an (efficient) compact exhaustion.
Those two facts will be used throughout.
\item A nested sequence $\cpa{K_i}$ of compacta in a space $X$ can cover $X$
without satisfying property~\ref{inint} in the definition of a compact exhaustion.
Consider $\R^2$ and let $K_i$ be the closed disk of radius $i\in\Z_+$ centered at the origin
minus the open sector of points with polar angle $0<\theta<\pi/i$ as in Figure~\ref{fig:pacman}.
\begin{figure}[htbp!]
    \centerline{\includegraphics[scale=1.0]{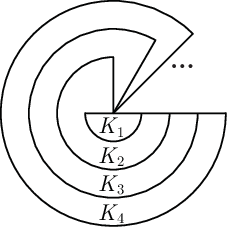}}
    \caption{Nested compacta $K_i$ covering $\R^2$.}
\label{fig:pacman}
\end{figure}
The nested compacta $\cpa{K_i}$ cover $\R^2$, but property~\ref{inint} in the definition of compact exhaustion is not satisfied.
No compactum in $\R^2$ containing an open neighborhood of the origin is contained in any $K_i$, let alone in the interior of any $K_i$.
\item For a compact exhaustion that is not efficient, consider $\R^2$ and
let $K_i$ be the closed disk of radius $i\in\Z_+$ centered at the origin minus
the open disk of radius $1/4$ centered at $(i-1/2,0)$.
Each complement $\R^2-K_i$ has one bounded component.
\item Not every space admits a compact exhaustion.
Consider the wedge sum $W$ of countably infinitely many copies of the closed unit interval $\br{0,1}\subseteq\R$ each based at $0$.
The simplicial $1$-complex $W$ fails to be locally finite at the wedge point $p$, and
no compact subspace of $W$ contains $p$ in its interior.
Thus, some conditions are required to ensure that a space admits an exhaustion by compacta.
\end{enumerate}
\end{examples}

Using terminology of Baues and Quintero~\cite[p.~58]{bq}, a \deffont{generalized continuum} is a Hausdorff space that is
connected, locally connected, $\sigma$-compact, and locally compact.
Recall that a space is \deffont{$\sigma$-compact} provided it is the union of countably many compacta.
Generalized continua include all:
\begin{enumerate}[label=(\arabic*)]
\item\label{topmflds} connected, Hausdorff, second countable manifolds (with or without boundary),
\item\label{lfscs} connected, locally finite simplicial complexes,
\item\label{lfcws} connected, locally finite CW complexes, and
\item\label{manrs} connected, separable, metric absolute neighborhood retracts (ANRs).
\end{enumerate}

\begin{remarks}\label{gcremarks}\mbox{ \\ }
\begin{enumerate}[label=(\alph*)]
\item To show those and other
connected spaces are generalized continua, useful references include
Lee~\cite[pp.~38--44~\&~93--110]{lee} and Spivak~\cite[pp.~1--6~\&~459--460]{spivak} for manifolds,
Hatcher~\cite[pp.~519--529]{hatcher} and Geoghegan~\cite[$\S$10.1]{geoghegan} for CW complexes,
Guilbault~\cite[$\S$3.1, 3.3, \& 3.12]{guilbault} for ANRs,
Munkres~\cite[p.~215]{munkres} for Urysohn's metrization theorem,
and Rudin~\cite{rudin} for a short proof that each metric space is paracompact.
\item\label{gcnice} Each generalized continuum is also paracompact, regular, normal, and Lindel{\"o}f.
There are several ways to deduce those properties from the literature.
For instance, locally compact, Hausdorff, and $\sigma$-compact imply paracompact by Hocking and Young~\cite[Thm.~2-65]{hockingyoung}.
Paracompact and Hausdorff imply regular and normal by Hocking and Young~\cite[Thms. \hbox{2-62} \& \hbox{2-63}]{hockingyoung}.
Regular, Hausdorff, and $\sigma$-compact imply Lindel{\"o}f by Engelking~\cite[Thm.~3.8.5]{engelking}.
\item Each space of type~\ref{topmflds}--\ref{manrs} above is metrizable.
Nevertheless, there exist nonmetrizable generalized continua---see Example~\ref{nmgc}.
\item\label{cgcuece} If $X$ is a \textit{compact} generalized continuum,
then there is a unique efficient compact exhaustion of $X$, namely $K_i=X$ for $i\in\Z_+$.
(Otherwise, some $X-K_i$ contains a component $C$.
Efficiency of $K_i$ implies $\closure{C}$ is noncompact.
But, $\closure{C}$ is compact since it is closed in the compactum $X$, a contradiction.)
\end{enumerate}
\end{remarks}

It is a well-known fact that each generalized continuum admits an \emph{efficient} exhaustion by compacta.
For instance, that fact---with the local connectedness hypothesis replaced
by local path connectedness---is Exercise~3.3.4 in Guilbault~\cite{guilbault}.
We include a proof to make this text more self contained,
since we are unaware of a published proof,
and we use it to prove novel extensions to pairs of spaces and to maps.
The following notation and terminology will be useful.

\begin{notation}
Given a subspace $A$ of a space $X$, let $\bc{A}$ denote the set of bounded components of $A$,
and let $\uc{A}$ denote the set of unbounded components of $A$.
If $K$ is a compactum in $X$,
then we define $K'=K\cup\bc{X-K}$ the \deffont{bounded filling} of $K$ in $X$.
We employ standard notation for the union of sets:
if $\mathcal{A}$ is a set of sets,
then $\cup\mathcal{A}$ and $\sqcup\mathcal{A}$ respectively denote the union and disjoint union of the sets in $\mathcal{A}$.
In particular, we have
\begin{equation}\label{XdisjointunionK}
X=K \sqcup \bc{X-K} \sqcup \uc{X-K} = K' \sqcup \uc{X-K}
\end{equation}
Note that $\sqcup$ merely emphasizes that sets are disjoint and is not used in a topological sense.
\end{notation}

\begin{example}
Consider the connected compactum $K\subseteq \R^2$ in Figure~\ref{fig:punctureddisk}.
\begin{figure}[htbp!]
    \centerline{\includegraphics[scale=1.0]{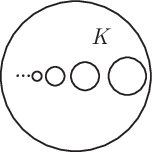}}
    \caption{Connected compactum $K$ in $\R^2$.}
\label{fig:punctureddisk}
\end{figure}
It is obtained by removing from the closed unit disk a sequence of disjoint open subdisks converging to a point.
The complement of $K$ in $\R^2$ has infinitely many bounded components.
Nevertheless, the bounded filling $K'$ of $K$ in $\R^2$ is compact and
the number of unbounded components of $\R^2-K$ is finite.
Those two properties hold generally as will be shown in Lemma~\ref{compactlemma}.
\end{example}

\begin{lemma}\label{compslemma}
Let $X$ be a generalized continuum, let $K\subseteq X$ be a nonempty compactum, and let $C$ be a component of $X-K$.
Then:
\begin{enumerate*}[label=(\roman*)]
  \item\label{compsopen} $C$ is open in $X$,
  \item\label{CmeetsnhbdK} $C$ meets each open neighborhood of $K$ in $X$,
	\item\label{clCinKuC} $\closure{C}\subseteq K \cup C$,
	\item\label{clCmeetsK} $\closure{C}$ meets $K$,
	\item\label{KconnKunioncompsconn} if $K$ is connected, then $K$ union any collection of components of $X-K$ is connected, and
	\item\label{KconnK'conn} if $K$ is connected, then the bounded filling $K'$ of $K$ in $X$ is connected.
\end{enumerate*}
\end{lemma}

\begin{proof}
As $K\subseteq X$ where $K$ is compact and $X$ is Hausdorff, $K$ is closed in $X$ and $X-K$ is open in $X$.
As $X$ is locally connected, $X-K$ is locally connected.
As $C$ is a component of $X-K$, $C$ is open in $X-K$.
As $X-K$ is open in $X$, $C$ is open in $X$ proving~\ref{compsopen}.

Let $N$ be an open neighborhood of $K$ in $X$.
Suppose, by way of contradiction, that $C$ is disjoint from $N$.
By~\ref{compsopen}, each component of $X-K$ is open in $X$.
So, the union $U$ of $N$ and all components of $X-K$ other than $C$ is open in $X$.
Thus, $C$ and $N$ separate $X$.
That contradicts the hypothesis that $X$ is connected and proves~\ref{CmeetsnhbdK}.

Note that $K\cup C$ is closed in $X$ since its complement in $X$ is the union of all components of $X-K$ other than $C$
and components of $X-K$ are open in $X$ by~\ref{compsopen}.
As \hbox{$K\cup C$} is closed in $X$ and contains $C$, $\closure{C}\subseteq K\cup C$ proving~\ref{clCinKuC}.
Suppose, by way of contradiction, that $\closure{C}$ is disjoint from $K$.
As $\closure{C}\subseteq K\cup C$, we have $\closure{C}\subseteq C$ and $\closure{C}= C$.
So, $C$ is open and closed in $X$, is nonempty, and has nonempty complement since $K$ is nonempty.
Thus, $X$ is disconnected, a contradiction, which proves~\ref{clCmeetsK}.

Assume $K$ is connected and let $A\subseteq X$ be the union of $K$ and an arbitrary collection of components of $X-K$.
Let $B\subseteq A$ be a subspace that is open, closed, and meets $K$.
As $K$ is connected, $K\subseteq B$.
Let $C$ be any component of $X-K$ that lies in $A$.
As $C$ is connected, $\closure{C}$ is connected.
We have $\closure{C}\subseteq K\cup C \subseteq A$ where the first containment holds by~\ref{clCinKuC}.
By~\ref{clCmeetsK}, $\closure{C}$ meets $K\subseteq B$.
Hence, $C\subseteq\closure{C}\subseteq B$ where the second containment holds since $\closure{C}$ is connected.
Therefore, $A\subseteq B$ and $A$ is connected, proving~\ref{KconnKunioncompsconn}.
Now,~\ref{KconnK'conn} is an immediate consequence of~\ref{KconnKunioncompsconn}.
\end{proof}

\begin{lemma}\label{compstechlemma}
Let $Y$ be a subspace of a space $Z$.
Then:
\begin{enumerate*}[label=(\roman*)]
  \item\label{connss} if $A\subseteq Y$, then $A$ is connected as a subspace of $Y$ if and only if
	$A$ is connected as a subspace of $Z$, and
  \item\label{compsofunionofcomps} if $Y$ is a union of components of $Z$, then the components of $Y$ are the components of $Z$ in that union.
\end{enumerate*}
\end{lemma}

\begin{proof}
The topology on $A$ as a subspace of $Y$ equals the topology on $A$ as a subspace of $Z$, which proves~\ref{connss}.
Let $\cpa{C_i \mid i\in T}$, where $T$ is an arbitrary index set, be the set of components of $Z$.
Suppose $Y=\cup_{i\in S} C_i$ for some index set $S\subseteq T$.
Consider $C_j$ for some $j\in S$.
As $C_j$ is connected in $Z$, $C_j$ is nonempty and is connected in $Y$ by~\ref{connss}.
So, $C_j$ lies in a unique component $C$ of $Y$.
As $C$ is connected in $Y$, $C$ is connected in $Z$ by~\ref{connss}.
As $C$ meets the component $C_j$ of $Z$, $C$ lies in $C_j$.
Thus, $C_j=C$ is a component of $Y$.
Conclusion~\ref{compsofunionofcomps} follows since $Y=\cup_{i\in S} C_i$.
\end{proof}

\begin{lemma}\label{compactlemma}
Let $X$ be a generalized continuum and let $K\subseteq X$ be a nonempty compactum.
Then:
\begin{enumerate*}[label=(\roman*)]
  \item\label{finunbdedcomps} $X-K$ has finitely many unbounded components,
  \item\label{K'compact} the bounded filling $K'$ of $K$ in $X$ is compact, and
	\item\label{unbdedcomplK'} $X-K'$ has only unbounded components and $\uc{X-K'}=\uc{X-K}$.
\end{enumerate*}
\end{lemma}

\begin{proof}
As $X$ is locally compact, there exists a bounded open neighborhood $N$ of $K$ in $X$.
So, $X$ has an open cover $\cpa{N}\cup\bc{X-K}\cup\uc{X-K}$.
As $\closure{N}$ is compact, $\closure{N}$ has a finite cover
\begin{equation}\label{finitecoverofclosureV}
\closure{N} \subseteq N\cup B_1\cup B_2\cup\cdots\cup B_k \cup U_1\cup U_2\cup\cdots\cup U_l
\end{equation}
where each $B_i$ is a bounded component of $X-K$ and each $U_i$ is an unbounded component of $X-K$.

We claim that any component $C$ of $X-K$ other than a $B_i$ or $U_i$ in~\eqref{finitecoverofclosureV} is contained in $N$.
Suppose, by way of contradiction, that $C$ is such a component not contained in $N$.
By Lemma~\ref{compslemma} parts~\ref{compsopen} and~\ref{CmeetsnhbdK}, $C\cap N$ is open in $X$ and in $C$, and $C\cap N$ is nonempty.
Also, $C-\closure{N}$ is open in $X$ and in $C$.
As $C$ is not contained in $N$, $C$ meets $X-N$.
By~\eqref{finitecoverofclosureV}, $C$ meets $X-\closure{N}$ and $C-\closure{N}$ is nonempty.
Thus, $C\cap N$ and $C-\closure{N}$ separate $C$.
But, $C$ is connected, and that contradiction proves the claim.

As $N$ is bounded, no unbounded component of $X-K$ is contained in $N$.
So, the claim implies that $U_1$, $U_2$,\ldots , $U_l$ must be all of the unbounded components of $X-K$ proving~\ref{finunbdedcomps}.
The claim also implies that all but finitely many bounded components of $X-K$ lie in $N$.

As $\cup \, \uc{X-K}$ is open in $X$, its complement $K'$ is closed in $X$.
Also
\[
K'=K\cup\bc{X-K} \subseteq\closure{N}\cup \closure{B_1}\cup\closure{B_2}\cup\cdots\cup\closure{B_k}
\]
where the latter is a finite union of compact spaces and thus is compact.
So, $K'$ is closed in $X$ and is contained in a compact subspace of $X$.
Hence, $K'$ is compact which proves~\ref{K'compact}.

By~\eqref{XdisjointunionK}, we have $X-K=\sqcup\bc{X-K}\sqcup\uc{X-K}$ and $X-K'=\sqcup\,\uc{X-K}$.
Lemma~\ref{compstechlemma}\ref{compsofunionofcomps} with $Z=X-K$ and $Y=X-K'$
implies that the components of $X-K'$ are exactly the unbounded components of $X-K$.
That proves~\ref{unbdedcomplK'} and the lemma.
\end{proof}

\begin{lemma}[Monotonicity]\label{monotonicity}
Let $X$ be a generalized continuum, and let $K\subseteq L$ be compacta in $X$.
Then:
\begin{enumerate*}[label=(\roman*)]
  \item\label{welldef} each unbounded component of $X-L$ lies in a unique unbounded component of $X-K$,
  \item\label{surj} each unbounded component of $X-K$ contains at least one unbounded component of $X-L$,
	\item\label{cansurj} there is a canonical surjection $\uc{X-L} \twoheadrightarrow \uc{X-K}$ of finite sets,
	\item\label{monoton} $0\leq\card{\uc{X-K}}\leq\card{\uc{X-L}}<\infty$, and
	\item\label{monotonint} if $K\subseteq \interior{L}$, then $K'\subseteq \interior{\pa{L'}}$.
\end{enumerate*}
\end{lemma}

\begin{proof}
As $K\subseteq L$, we have $X-L\subseteq X-K$.
Let $U$ be an unbounded component of $X-L$.
As $U$ is connected, $U$ is contained in a unique component $U_K$ of $X-K$.
As $U$ is unbounded, $U_K$ is unbounded proving~\ref{welldef}.
Next, let $V$ be an unbounded component of $X-K$.
By conclusion~\ref{K'compact} of Lemma~\ref{compactlemma}, $L'$ is compact.
As $V$ is unbounded, $V$ meets at least one unbounded component of $X-L$.
If $C$ is any such unbounded component of $X-L$,
then~\ref{welldef} implies that $C$ is contained in $V$, proving~\ref{surj}.
Define the function $\uc{X-L} \to \uc{X-K}$ by $U\mapsto U_K$.
That function is well-defined by~\ref{welldef}, is surjective by~\ref{surj}, and has finite domain and codomain by
conclusion~\ref{finunbdedcomps} of Lemma~\ref{compactlemma}, which proves~\ref{cansurj}.
Conclusion~\ref{monoton} follows immediately from~\ref{cansurj}.

Lastly, assume $K\subseteq \interior{L}$.
To prove~\ref{monotonint}, we must show that $K$ and each bounded component of $X-K$ lie in $\interior{\pa{L'}}$.
First, $L\subseteq L'$ implies that $K\subseteq \interior{L} \subseteq \interior{\pa{L'}}$.
Second, let $B\in \bc{X-K}$.
Conclusion~\ref{compsopen} of Lemma~\ref{compslemma} implies that $B$ is open in $X$,
so it suffices to show that $B\subseteq L'$.
Thus, it suffices to show that $B$ is disjoint from each unbounded component of $X-L$.
Let $U\in\uc{X-L}$.
Then, $U\subseteq U_K\in\uc{X-K}$ by~\ref{welldef}.
As $B\in \bc{X-K}$ is disjoint from $U_K\in\uc{X-K}$,
we see that $B$ is disjoint from $U$ proving~\ref{monotonint}.
\end{proof}

\begin{theorem}[Existence of efficient compact exhaustion]\label{gceee}
Let $X$ be a generalized continuum. Then, $X$ admits an efficient exhaustion by compacta $\cpa{K_i}$.
\end{theorem}

\begin{proof}
The compact case is covered by Remarks~\ref{gcremarks}\ref{cgcuece}.
So, assume that $X$ is noncompact.

Observe that each point $p\in X$ has a connected, bounded, open neighborhood $U_p$ in $X$.
To see that, local compactness of $X$ yields a compactum $C_p \subseteq X$ containing an open neighborhood $V_p$ of $p$.
Local connectedness of $X$ yields a connected, open neighborhood $U_p\subseteq V_p$ of $p$.
Then, $U_p\subseteq C_p$ implies that $\closure{U_p}\subseteq \closure{C_p}$.
As $C_p$ is compact and $X$ is Hausdorff, $C_p$ is closed in $X$.
So, $\closure{C_p}=C_p$ is compact.
Thus, $\closure{U_p}$ is a closed subspace of the compactum $C_p$ and hence is compact.
Therefore, $U_p$ is a connected, bounded, open neighborhood of $p$ in $X$ as desired.

As $X$ is $\sigma$-compact, $X=L_1\cup L_2\cup \cdots$ is a countable union of compacta.
By compactness, each $L_i$ is covered by finitely many connected, bounded, open sets $U_p$.
Reindexing, $X=U_1\cup U_2\cup\cdots$ where each $U_i$ is connected, bounded, and open in $X$.
Note that as $X$ is noncompact and each $U_i$ is bounded in $X$, no finite collection of the $U_i$ cover $X$.

Reindex the $U_i$ so that for each integer $j\geq2$, $U_j$ meets $U_1\cup U_2\cup\cdots\cup U_{j-1}$.
To achieve that, leave the index on $U_1$ unchanged and proceed inductively.
Let $m\geq2$ be the minimal index such that $U_m$ meets $U_1$.
Such an $m$ exists since otherwise $U_1$ and $\cup_{i\geq2}U_i$ separate $X$.
Swap the indices on $U_2$ and $U_m$, so $U_2$ meets $U_1$.
Assume that for some integer $k\geq2$, the desired property holds for each $j=2,3,\ldots,k$.
Let $m\geq k+1$ be the minimal index such that $U_m$ meets $U_1\cup U_2\cup\cdots\cup U_k$.
Such an $m$ exists since otherwise $U_1\cup U_2\cup\cdots\cup U_k$ and $\cup_{i\geq k+1}U_i$ separate $X$.
Swap the indices on $U_{k+1}$ and $U_m$, so $U_{k+1}$ meets $U_1\cup U_2\cup\cdots\cup U_k$.
That completes our reindexing of the $U_i$.

For each $i\in\Z_+$, define $J_i=U_1\cup U_2\cup\cdots\cup U_i$.
So, $\cpa{J_i}$ is a nested sequence of connected, bounded, open subspaces of $X$ that cover $X$.
Pass to a subsequence---still denoted $\cpa{J_i}$---so that further $\closure{J_i} \subseteq J_{i+1}$ for each $i\in\Z_+$.
That is possible since each $\closure{J_i}$ is compact and is covered by the open sets
$J_{i+1},J_{i+2},J_{i+3},\ldots$.
For each $i\in\Z_+$, let $K_i=\closure{J_i}'$ be the bounded filling of $\closure{J_i}$ in $X$.

We claim that $\cpa{K_i}$ is a desired efficient exhaustion of $X$ by compacta.
To prove that claim, each $K_i$ is compact since $J_i$ is bounded and by conclusion~\ref{K'compact} of Lemma~\ref{compactlemma}.
Each $K_i$ is connected since $J_i$ is connected, the closure of a connected subspace is connected,
and by conclusion~\ref{KconnK'conn} of Lemma~\ref{compslemma}.
Conclusion~\ref{unbdedcomplK'} of Lemma~\ref{compactlemma} implies that each component of $X-K_i$ is unbounded.
In particular, each $K_i$ is an efficient compactum in $X$.
As $J_i\subseteq K_i$ and the $J_i$ cover $X$, the $K_i$ cover $X$.
As $\closure{J_i}\subseteq J_{i+1} \subseteq \closure{J_{i+1}}$ and $J_{i+1}$ is open in $X$,
we have $\closure{J_i}\subseteq J_{i+1}\subseteq \interior{\closure{J_{i+1}}}$.
By Lemma~\ref{monotonicity} (monotonicity), we get
$\closure{J_i}'\subseteq \interior{\pa{\closure{J_{i+1}}'}}$.
That is, $K_i\subseteq \interior{K_{i+1}}$ proving the claim and the theorem.
\end{proof}

\section{End spaces}
\label{sec:es}

The topological theory of ends of spaces was initiated by Freudenthal in his 1931 thesis~\cite{freudenthal31}.
Ends were further studied by Freudenthal~\cite{freudenthal42} and Hopf~\cite{hopf} in the 1940's\footnote{As a historical aside,
Freudenthal completed his pioneering thesis on ends and topological groups before he approached Hopf to be his advisor~\cite[p.~578]{samelson}.
For further reading on Freudenthal's impacts on topology---including suspension and reduced cohomology, ends of spaces,
and direct and inverse limits---see Dieudonn\'{e}~\cite[pp.~217 \& 364]{dieudonne} and van Est~\cite{vanest93,vanest99}.}
and they now play a fundamental role in the study of noncompact spaces---see Peschke \cite{peschke} and Guilbault~\cite{guilbault}.
In this section, we review pertinent topological definitions and results necessary for our purposes.
For further references on ends of spaces, see Raymond~\cite{raymond}, Siebenmann's thesis~\cite[Ch.~1]{siebenmannthesis}, 
Porter's chapter~\cite[$\S$2]{porter}, and the books by
Hughes and Ranicki~\cite[Ch.~1]{hr}, Baues and Quintero~\cite[Ch.~I $\S$9]{bq}, and
Geoghegan~\cite[$\S$13.4]{geoghegan}.

Let $X$ be a space with a compact exhaustion $\cpa{K_i}$.
For each $i$, let $V_i=X-K_i$ and write $\cpa{V_i^j}=\uc{V_i}$ for the set of unbounded components of $V_i$.
An \deffont{end} of $X$ is a sequence\footnote{That notation for an end was shown to us by Craig Guilbault.
It inspired our notation for basic open sets in the end space and the endpoint compactification.}
$\e=\pa{V_1^{j_1},V_2^{j_2},\ldots}$ such that $V_1^{j_1}\supseteq V_2^{j_2} \supseteq\cdots$.
Let $E\pa{X;\cpa{K_i}}$ denote the \deffont{set of ends} of $X$ with respect to the compact exhaustion $\cpa{K_i}$.
The \deffont{number of ends} of $X$ equals the cardinality of $E\pa{X;\cpa{K_i}}$---which is independent of the compact
exhaustion by Lemma~\ref{canbij} ahead.
Lemma~\ref{monotonicity} (monotonicity) yields the inverse system\footnote{For background
on inverse systems and inverse limits,
see Eilenberg and Steenrod~\cite[Ch.~VIII]{eilenbergsteenrod} and Massey~\cite[A.1~\&~A.3]{massey}.}
\begin{equation}\label{invsys}
\uc{V_1} \twoheadleftarrow \uc{V_2} \twoheadleftarrow \uc{V_3} \twoheadleftarrow \cdots
\end{equation}
of finite sets with surjective bonding functions.
The inverse limit of~\eqref{invsys} equals the set of ends of $X$.
That is, $\ilim \uc{V_i} = E\pa{X;\cpa{K_i}}$.

\begin{lemma}\label{unbdedcomptoend}
Let $X$ be a generalized continuum and let $\cpa{K_i}$ be a compact exhaustion of $X$.
If $V_k^l$ is an unbounded component of $V_k=X-K_k$ for some $k\in\Z_+$,
then there exists at least one end $\e=\pa{W_1,W_2,\ldots}\in E\pa{X;\cpa{K_i}}$ such that $W_k=V_k^l$.
The terms $W_1,W_2,\ldots,W_{k-1}$ in $\e$ are \textit{uniquely} determined by $W_k=V_k^l$.
\end{lemma}

\begin{proof}
We construct such a sequence $\e$ using the montonicity Lemma~\ref{monotonicity}.
Lemma~\ref{monotonicity}\ref{welldef} uniquely determines the terms preceding $W_k=V_k^j$.
Lemma~\ref{monotonicity}\ref{surj} and induction show the existence of appropriate terms following $W_k$.
\end{proof}

\begin{lemma}\label{canbij}
Let $X$ be a generalized continuum and let $\cpa{K_i}$ and $\cpa{L_i}$ be compact exhaustions of $X$.
Then, there is a bijection $E\pa{X;\cpa{K_i}} \to E\pa{X;\cpa{L_i}}$.
\end{lemma}

\begin{proof}
First, consider a subsequence $\cpa{K_{i_k}}$ of $\cpa{K_i}$.
The natural function $\sigma:E\pa{X;\cpa{K_i}} \to E\pa{X;\cpa{K_{i_k}}}$
is defined by $\pa{V_i^{j_i}}\mapsto \pa{V_{i_k}^{j_{i_k}}}$.
By Lemma~\ref{unbdedcomptoend}, $\sigma$ has a unique inverse and, hence, is a bijection.

Next, we shuffle the exhaustions $\cpa{K_i}$ and $\cpa{L_i}$.
To do so, replace $\cpa{K_i}$ and $\cpa{L_i}$ with subsequences such that
$K_i\subseteq\interior{L_i}$ and $L_i\subseteq\interior{K_{i+1}}$ for each $i\in\Z_+$. 
Define the compact exhaustion $\cpa{J_i}$ of $X$ by
$J_{2i-1}=K_i$ and $J_{2i}=L_i$ for each $i\in\Z_+$.
Both $\cpa{K_i}$ and $\cpa{L_i}$ are subsequences of $\cpa{J_i}$.
So, we have bijections
\[
E\pa{X;\cpa{K_i}} \leftarrow E\pa{X;\cpa{J_i}} \to E\pa{X;\cpa{L_i}}
\]
as desired.
\end{proof}

\begin{remark}[The ``W" argument]\label{Wargument}
The technique used in the proof of Lemma~\ref{canbij} is noteworthy.
Given two sequences, one passes to appropriate subsequences that are themselves subsequences of a common sequence.
In more detail, suppose each sequence of some admissible type determines an object.
Suppose the objects determined are equivalent under the operation of passing to a subsequence.
Consider admissible sequences $\cpa{K_k}$ and $\cpa{L_l}$, and
suppose that there exist subsequences $\cpa{K_{k_i}}$ and $\cpa{L_{l_i}}$ such that
the sequence $\cpa{J_i}$ is admissible where $J_{2i-1}=K_{k_i}$ and $J_{2i}=L_{l_i}$ for each $i\in\Z_+$.
Those sequences form a ``W'' shaped pattern
\begin{equation}\label{eq:w}
\begin{tikzcd}
\cpa{K_k} \arrow[dr] &	&	\cpa{J_i}	\arrow[dl] \arrow[dr] &	&	\cpa{L_l} \arrow[dl] \\
	& \cpa{K_{k_i}}	&	& \cpa{L_{l_i}}
\end{tikzcd}
\end{equation}
where a downward arrow denotes the operation of passing to a subsequence.
Thus, our original sequences $\cpa{K_k}$ and $\cpa{L_l}$ determine equivalent objects.
The proof of Lemma~\ref{canbij} used the fact that all compact exhaustions of $X$ are related in that manner.
The ``W'' argument will be used to prove Corollary~\ref{canhom},
solves an exercise on pro-isomorphism of inverse systems in Guilbault~\cite[Ex.~3.4.5]{guilbault},
and is useful in general.
\end{remark}

Let $X$ be a generalized continuum.
Define $E(X)$ to be the \deffont{set of ends of $X$} with the understanding that $E(X)=E\pa{X;\cpa{K_i}}$
for some compact exhaustion of $X$---which exists by Theorem~\ref{gceee}---and is well-defined
up to bijection by Lemma~\ref{canbij}.
In fact, that bijection is a canonical homeomorphism as will be shown in Corollary~\ref{canhom}.

\begin{theorem}\label{endsthm}
Let $X$ be a generalized continuum. Consider the set of ends $E(X)=E\pa{X;\cpa{K_i}}$ where $\cpa{K_i}$ is a compact exhaustion of $X$.
Then:
\begin{enumerate*}[label=(\roman*)]
  \item\label{gencardends} $0\leq\card{E(X)}\leq\card{\R}$,
  \item\label{lowerboundnumends} if $K$ is a compactum in $X$, then $\card{E(X)} \geq \card{\uc{X-K}}$,
	\item\label{noncompactexistsend} $X$ is compact if and only if $X$ has no ends, and
	\item\label{finiteendsbondsbij} $X$ has finitely many ends if and only if cofinitely many of the bonding surjections in~\eqref{invsys} are bijections.
\end{enumerate*}
\end{theorem}

\begin{proof}
For \ref{gencardends}, use the definition of an end and the inverse system~\eqref{invsys} of finite sets and surjections.
Next, choose $k\in\Z_+$ such that $K\subseteq K_k$.
Then, $\card{\uc{X-K}} \leq \card{\uc{X-K_k}} \leq \card{E(X)}$
where the first inequality holds by Lemma~\ref{monotonicity} (monotonicity) and the second holds by Lemma~\ref{unbdedcomptoend}.
That proves~\ref{lowerboundnumends}.
For~\ref{noncompactexistsend}, Theorem~\ref{gceee} and Lemma~\ref{canbij} imply that we may assume $\cpa{K_i}$
is an efficient compact exhaustion of $X$.
By Remarks~\ref{gcremarks}\ref{cgcuece}, $X$ is compact if and only if $X=K_i$ for each $i\in\Z_+$.
The latter means the sets in the inverse system~\eqref{invsys} are empty, which proves~\ref{noncompactexistsend}.
For the forward implication of~\ref{finiteendsbondsbij}, if infinitely many bonding surjections in~\eqref{invsys} are not bijections,
then the sets $\uc{X-K_i}$ have unbounded (finite) cardinalities and~\ref{lowerboundnumends} implies that $E(X)$ is infinite.
For the reverse implication of~\ref{finiteendsbondsbij}, the hypothesis allows us to pass to a subsequence of
$\cpa{K_i}$ so that all bonding surjections in~\eqref{invsys} are bijections.
Thus, $\card{E(X)}=\card{\uc{X-K_1}}$ is finite.
\end{proof}

Let $X$ be a generalized continuum.
Informally, Freudenthal's \deffont{endpoint compactification} $F(X)$ of $X$ is obtained by adding one point to each end of $X$.
We present some examples before the precise definition.
\begin{examples}\label{endspaceexamples}\mbox{ \\ }
\begin{enumerate}[label=(\alph*)]
\item If $X$ is compact, then $F(X)=X$ by Theorem~\ref{endsthm}\ref{noncompactexistsend}.
\item\label{alexandroff} If $X$ is one-ended, then $F(X)$ is homeomorphic to Alexandroff's
one-point compactification of $X$ (see the end of Section~\ref{sec:br}).
\item Evidently, $E\pa{\nnr}=\cpa{\bullet}$ is a one-point space and $F\pa{\nnr}\approx\br{0,1}$.
Also, $E\pa{\R}=\cpa{\pm\infty}$ is a two-point discrete space and $F\pa{\R}\approx\br{0,1}$.
For each $n\geq2$, $E\pa{\R^n}=\cpa{\bullet}$ is a one-point space and $F\pa{\R^n}\approx S^n$.
\item If $X$ is the thrice punctured $2$-sphere,
then $E(X)$ is a three-point discrete space and $F(X)\approx S^2$.
\item Let $X$ be the infinite comb space depicted in Figure~\ref{fig:nondiscrete}.
\begin{figure}[htbp!]
    \centerline{\includegraphics[scale=1.0]{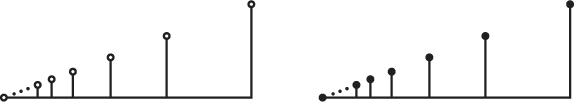}}
    \caption{Infinite comb space (left) and its endpoint compactification (right).}
\label{fig:nondiscrete}
\end{figure}
Namely, $X$ is the subspace of $\R^2$ equal to $(0,1]\times\cpa{0}$ union the vertical intervals $\cpa{1/i}\times[0,1/i)$ for $i\in\Z_+$.
Then, $F(X)$ is homeomorphic to the compact subspace of $\R^2$ depicted in Figure~\ref{fig:nondiscrete} (right), and
$E(X)$ is homeomorphic to $\cpa{0}\cup\cpa{1/i \mid i \in\Z_+}\subseteq\R$.
\item If $X$ is the infinite binary tree depicted in Figure~\ref{fig:infbintree}, then $E(X)$ is the Cantor set.
\begin{figure}[htbp!]
    \centerline{\includegraphics[scale=1.00]{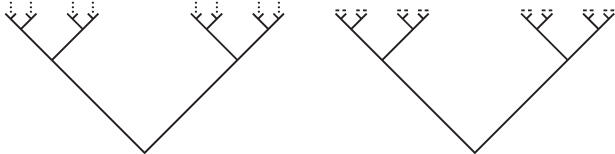}}
    \caption{Infinite binary tree (left) and its endpoint compactification (right).}
\label{fig:infbintree}
\end{figure}
\end{enumerate}
\end{examples}

More precisely, let $\cpa{K_i}$ be a compact exhaustion of a generalized continuum $X$.
Define the set $F\pa{X;\cpa{K_i}}=X \sqcup E\pa{X;\cpa{K_i}}$.
Let $V_i=X-K_i$ and $\cpa{V_i^j}=\uc{V_i}$ for each $i\in\Z_+$.
Given an unbounded component $V_k^l$ of some $V_k$, define the set of ends\footnote{Useful mnemonic:
$EV_k^l$ for \textit{end} and for $k$th term \textit{equal} to $V_k^l$.}
\begin{equation}\label{basisE}
EV_k^l=\cpa{\pa{W_1,W_2,W_3,\ldots} \in E\pa{X;\cpa{K_i}} \mid W_k=V_k^l}
\end{equation}
which is nonempty by Lemma~\ref{unbdedcomptoend}.
Consider the collection of all $EV_k^l$ for $k\in\Z_+$ and $1\leq l \leq \card{\uc{V_k}}$.
Those sets cover $E\pa{X;\cpa{K_i}}$, and if two of them meet, then one contains the other by Lemma~\ref{ebasislemma} below.
So, they form a basis for a topology on $E\pa{X;\cpa{K_i}}$.
The \deffont{space of ends} of $X$ is $E\pa{X;\cpa{K_i}}$ equipped with that topology.
Given some $V_k^l$, define the set
\begin{equation}\label{basisF}
FV_k^l = V_k^l \sqcup EV_k^l \subseteq F\pa{X;\cpa{K_i}}
\end{equation}
One may readily verify that those sets together with all open sets in $X$ form a basis for a topology on $F\pa{X;\cpa{K_i}}$.
Freudenthal's \deffont{endpoint compactification} of $X$ is $F\pa{X;\cpa{K_i}}$ equipped with that topology.
Note that the subspace topology on $E\pa{X;\cpa{K_i}}\subseteq F\pa{X;\cpa{K_i}}$ equals the topology generated by the basis~\eqref{basisE}.

\begin{lemma}\label{ebasislemma}
In the setup of the preceding paragraph, consider unbounded components $V_i^j$ and $V_k^l$ where $i\leq k$.
Then, the following are equivalent:
\begin{enumerate*}[label=(\roman*)]
	\item\label{compmeet} $V_i^j$ meets $V_k^l$,
	\item\label{compcont} $V_i^j$ contains $V_k^l$,
	\item\label{ebasismeets} $EV_i^j$ meets $EV_k^l$,
	\item\label{ebasiscont} $EV_i^j$ contains $EV_k^l$,
	\item\label{fbasismeets} $FV_i^j$ meets $FV_k^l$, and
	\item\label{fbasiscont} $FV_i^j$ contains $FV_k^l$.
\end{enumerate*}
\end{lemma}

\begin{proof}
By Lemma~\ref{monotonicity} (monotonicity), \ref{compmeet} implies \ref{compcont}.
As components are nonempty, \ref{compcont} implies \ref{compmeet}.
If $\e=\pa{W_1,W_2,\ldots}\in EV_i^j \cap EV_k^l$,
then $V_i^j = W_i \supseteq W_k =V_k^l$.
So, \ref{ebasismeets} implies \ref{compcont}.
If $V_i^j \supseteq V_k^l$, then Lemma~\ref{unbdedcomptoend} shows that $EV_i^j \supseteq EV_k^l$.
So, \ref{compcont} implies \ref{ebasiscont}.
Lemma~\ref{unbdedcomptoend} shows that $EV_k^l$ is nonempty, so \ref{ebasiscont} implies \ref{ebasismeets}.
Therefore, \ref{compmeet}--\ref{ebasiscont} are equivalent.
To complete the proof, recall the definition~\eqref{basisF}.
Thus: \ref{compmeet} implies \ref{fbasismeets}, 
\ref{fbasismeets} implies \ref{compmeet} or \ref{ebasismeets},
and \ref{compcont} and \ref{ebasiscont} if and only if \ref{fbasiscont}.
\end{proof}

\begin{theorem}[Freudenthal]\label{endspace}
Let $X$ be a generalized continuum and let $\cpa{K_i}$ be a compact exhaustion of $X$.
Then, the space of ends $E\pa{X;\cpa{K_i}}$ is Hausdorff, totally separated\footnote{A space $Y$ is \deffont{totally separated} provided:
for each pair of distinct points in $Y$, there exist disjoint open neighborhoods of those points whose union is $Y$.
A space is \deffont{totally disconnected} provided its components are points.
Totally separated implies totally disconnected.}, and compact.
\end{theorem}

\begin{proof}
Let $\e=\pa{V_1^{j_1},V_2^{j_2},\ldots}$ and $\e'=\pa{V_1^{l_1},V_2^{l_2},\ldots}$ be distinct ends of $X$.
Then, $V_m^{j_m} \neq V_m^{l_m}$ for some $m\in\Z_+$.
By Lemma~\ref{ebasislemma}, the basic open sets $EV_m^{j_m} \ni \e$ and $EV_m^{l_m} \ni \e'$ are disjoint, proving $E\pa{X;\cpa{K_i}}$ is Hausdorff.
Further, the complement in $E\pa{X;\cpa{K_i}}$ of any basic open set $EV_a^b$ equals
$\sqcup_{c\neq b} EV_a^{c}$ which is open in $E\pa{X;\cpa{K_i}}$.
Thus, $\e\in EV_m^{j_m}$ and $\e'\in E\pa{X;\cpa{K_i}} - EV_m^{j_m}$, proving $E\pa{X;\cpa{K_i}}$ is totally separated.

Suppose, by way of contradiction, that $E\pa{X;\cpa{K_i}}$ is not compact.
Then, there exists an open cover $\cpa{U_{\alpha}}_{\alpha\in S}$ of $E\pa{X;\cpa{K_i}}$ with no finite subcover.
For each $i\in\Z_+$, the set $\cpa{V_i^j}=\uc{V_i}$ is finite by Lemma~\ref{compactlemma}\ref{finunbdedcomps} 
and, thus, $\cpa{EV_i^j}$ is a finite set of disjoint basic open sets that cover $E\pa{X;\cpa{K_i}}$.
So, there exists $EV_1^{j_1}$ not covered by finitely many $U_{\alpha}$.
There also exists $V_2^{j_2}\subseteq V_1^{j_1}$ such that $EV_2^{j_2}$ is not covered by finitely many $U_{\alpha}$.
Inductively, we obtain $V_1^{j_1} \supseteq V_2^{j_2} \supseteq \cdots$
such that no $EV_i^{j_i}$ is covered by finitely many $U_{\alpha}$.
Those nested components determine an end $\e=\pa{V_1^{j_1},V_2^{j_2},\ldots}$.
As the $U_{\alpha}$ cover $E\pa{X;\cpa{K_i}}$, there exists $\beta\in S$ such that $\e\in U_{\beta}$.
There exists a basic open neighborhood $EV_k^l$ of $\e$ contained in $U_{\beta}$.
That implies $V_k^{j_k} = V_k^l$ and $EV_k^{j_k} = EV_k^l$.
But, $EV_k^{j_k}$ is not covered by finitely many $U_{\alpha}$ whereas $EV_k^l$ is covered by $U_{\beta}$ alone.
That contradiction proves $E\pa{X;\cpa{K_i}}$ is compact.
\end{proof}

\begin{theorem}[Freudenthal]\label{endcompthm}
Let $X$ be a generalized continuum and let $\cpa{K_i}$ be a compact exhaustion of $X$.
Then, the endpoint compactification $F\pa{X;\cpa{K_i}}$ is a compact generalized continuum,
and the inclusion $\iota:X \hookrightarrow F\pa{X;\cpa{K_i}}$ is an open embedding with dense image in $F\pa{X;\cpa{K_i}}$.
\end{theorem}

\begin{proof}
To show $F\pa{X;\cpa{K_i}}$ is Hausdorff, consider three cases:
two points in $X$, two points in $E(X)$, and one point in each of $X$ and $E(X)$.
The first case follows since $X$ is Hausdorff.
The second case follows as in the proof of Theorem~\ref{endspace}
except using the basic open sets $FV_m^{j_m} \ni \e$ and $FV_m^{l_m} \ni \e'$.
For the third case, let $x\in X$ and $\e=\pa{V_1^{j_1},V_2^{j_2},\ldots}\in E\pa{X;\cpa{K_i}}$.
As $x\in \interior{K_k}$ for some $k\in\Z_+$,
$FV_k^{j_k}$ is a basic open neighborhood of $\e$ in $F\pa{X;\cpa{K_i}}$ disjoint from $K_k$, as desired.

The inclusion $\iota$ is injective, has image $X$, is open, and is continuous
since each $V_k^l$ is open in $X$ by Lemma~\ref{compslemma}\ref{compsopen}.
Hence, $\iota$ is an open embedding.

We claim that $EV_k^l \subseteq \Cl{V_k^l}{F\pa{X;\cpa{K_i}}}$ for each $V_k^l \in \uc{V_k}$.
Let $\e=\pa{V_1^{j_1},V_2^{j_2},\ldots} \in EV_k^l$.
It suffices to prove that each basic open neighborhood of $\e$ in $F\pa{X;\cpa{K_i}}$ meets $V_k^l$.
So, consider $FV_a^b = V_a^b \cup EV_a^b$ a basic open neighborhood of $\e$ in $F\pa{X;\cpa{K_i}}$.
We have $\e\in EV_a^b$ and Lemma~\ref{ebasislemma} implies that $V_a^b$ meets $V_k^l$.
Hence, $FV_a^b$ meets $V_k^l$, proving the claim.

The claim implies that each basic open set $FV_k^l=V_k^l \sqcup EV_k^l$ is connected.
Indeed, $V_k^l$ is connected in $X$ and in $F\pa{X;\cpa{K_i}}$.
Now, apply the claim and Munkres~\cite[Thm.~23.4]{munkres}.
It follows readily that $F\pa{X;\cpa{K_i}}$ is locally connected.

The claim also implies that $\Cl{X}{F\pa{X;\cpa{K_i}}}=F\pa{X;\cpa{K_i}}$.
Indeed, if $\e \in E\pa{X;\cpa{K_i}}$, then $\e$ lies in some basic open set $EV_k^l$
and hence $\e\in \Cl{V_k^l}{F\pa{X;\cpa{K_i}}}$.
As $V_k^l\subseteq X$, $\Cl{V_k^l}{F\pa{X;\cpa{K_i}}} \subseteq \Cl{X}{F\pa{X;\cpa{K_i}}}$.
Hence, $\e \in \Cl{X}{F\pa{X;\cpa{K_i}}}$, as desired.
In particular, $X$ is dense in $F\pa{X;\cpa{K_i}}$ and $F\pa{X;\cpa{K_i}}$ is connected.

To see that $F\pa{X;\cpa{K_i}}$ is compact, let $\cpa{U_{\alpha}}_{\alpha\in S}$ be an open cover of $F\pa{X;\cpa{K_i}}$.
Consider, for all $\alpha \in S$, the collection of all basic open sets $FV_k^l$ contained in some $U_{\alpha}$.
That collection covers $E\pa{X;\cpa{K_i}}$---a compact subspace of $F\pa{X;\cpa{K_i}}$ by Theorem~\ref{endspace}.
Therefore, there exists $\mathcal{L}=\cpa{FV_{k_1}^{l_1},\ldots,FV_{k_n}^{l_n}}$ a finite cover of $E\pa{X;\cpa{K_i}}$
such that $FV_{k_i}^{l_i} \subseteq U_{\alpha_i}$ for each $i=1,\ldots,n$.
Let $k=\max\cpa{k_1,\ldots,k_n}$.
Then, $\mathcal{L}$ must cover every $V_k^a \in \uc{V_k}$
(otherwise, some $V_k^a$ is not covered by $\mathcal{L}$, Lemma~\ref{ebasislemma} implies that $EV_k^a$ is disjoint from $\mathcal{L}$,
and $EV_k^a\neq\emptyset$ contains an end not covered by $\mathcal{L}$, a contradiction).
Hence, $U_{\alpha_1},\ldots,U_{\alpha_n}$ cover $\cup \, \uc{V_k} \cup E\pa{X;\cpa{K_i}}$.
The bounded filling $K_k'$ of $K_k$ in $X$ is compact by Lemma~\ref{compactlemma}\ref{K'compact}.
So, finitely many $U_{\alpha}$ cover $K_k'$.
Taken together, finitely many $U_{\alpha}$ cover $F\pa{X;\cpa{K_i}}$.
Thus, $F\pa{X;\cpa{K_i}}$ is compact and, hence, locally compact and $\sigma$-compact.
It follows that $F\pa{X;\cpa{K_i}}$ is a generalized continuum.
\end{proof}

\begin{lemma}\label{endspacesshomeo}
Let $X$ be a generalized continuum, let $\cpa{K_i}$ be a compact exhaustion of $X$, and let $\cpa{K_{i_k}}$ be a subsequence of $\cpa{K_i}$.
Then, the bijection $\sigma:E\pa{X;\cpa{K_i}} \to E\pa{X;\cpa{K_{i_k}}}$ in Lemma~\ref{canbij} is a homeomorphism.
\end{lemma}

\begin{proof}
As above, $EV_k^l$ denotes a basic open set in $E\pa{X;\cpa{K_i}}$.
Each basic open set in $E\pa{X;\cpa{K_{i_k}}}$ has the form
\begin{equation}\label{basisEss}
\cpa{\pa{Z_1,Z_2,Z_3,\ldots} \in E\pa{X;\cpa{K_{i_k}}} \mid Z_m=V_{i_m}^n}
\end{equation}
for some unbounded component $V_{i_m}^n$ of $V_{i_m}$.
Recalling the bijection $\sigma$ from Lemma~\ref{canbij},
it is straightforward to verify that $\sigma\pa{EV_{i_m}^n}$ equals the basic open set~\eqref{basisEss}.
Therefore, $\sigma^{-1}$ sends each basic open set in $E\pa{X;\cpa{K_{i_k}}}$ to a basic open set in $E\pa{X;\cpa{K_i}}$.
Thus, $\sigma^{-1}$ is open and $\sigma$ is continuous.
Hence, it suffices to show that $\sigma$ is open or closed.
By Theorem~\ref{endspace}, $E\pa{X;\cpa{K_i}}$ is compact and $E\pa{X;\cpa{K_{i_k}}}$ is Hausdorff.
So, $\sigma$ is closed.
Alternatively, we may show $\sigma$ is open using bases.
We already observed that $\sigma\pa{EV_{i_m}^n}$ equals the basic open set~\eqref{basisEss}.
Consider a basic open set $EV_a^b$ such that $a \neq i_m$ for any $m\in\Z_+$.
There exists $l\in\Z_+$ minimal such that $a<i_l$.
Let $V_{i_l}^{c_1},\ldots,V_{i_l}^{c_n}$ be the unbounded components of $V_{i_l}$ contained in $V_a^b$.
So, $EV_a^b=EV_{i_l}^{c_1} \sqcup \cdots \sqcup EV_{i_l}^{c_n}$ and
$\sigma\pa{EV_a^b}=\sigma\pa{EV_{i_l}^{c_1}} \sqcup \cdots \sqcup \sigma\pa{EV_{i_l}^{c_n}}$
where the latter is a disjoint union of finitely many basic open sets by the observation.
That proves $\sigma$ is open.
\end{proof}

\begin{lemma}\label{endcompsshomeo}
Let $X$ be a generalized continuum, let $\cpa{K_i}$ be a compact exhaustion of $X$, and let $\cpa{K_{i_k}}$ be a subsequence of $\cpa{K_i}$.
Define $\tau:F\pa{X;\cpa{K_i}} \to F\pa{X;\cpa{K_{i_k}}}$ piecewise to be $\tn{id}:X\to X$ and $\sigma:E\pa{X;\cpa{K_i}} \to E\pa{X;\cpa{K_{i_k}}}$.
Then, $\tau$ is a homeomorphism.
\end{lemma}

\begin{proof}
By Lemma~\ref{canbij}, $\tau$ is a bijection.
We claim that $\tau^{-1}$ is open.
It suffices to verify that on a basis for the topology on $F\pa{X;\cpa{K_{i_k}}}$
and we consider the basis defined at~\eqref{basisF}.
As $\tau$ is the identity on $X$, it suffices to consider basic open sets that meet $E\pa{X;\cpa{K_{i_k}}}$.
As above, $FV_k^l$ denotes a basic open set in $F\pa{X;\cpa{K_i}}$.
Each basic open set in $F\pa{X;\cpa{K_{i_k}}}$ that meets $E\pa{X;\cpa{K_{i_k}}}$ has the form
\begin{equation}\label{basisFss}
V_{i_m}^n \sqcup \cpa{\pa{Z_1,Z_2,Z_3,\ldots} \in E\pa{X;\cpa{K_{i_k}}} \mid Z_m=V_{i_m}^n}
\end{equation}
for some unbounded component $V_{i_m}^n$ of $V_{i_m}$.
As in the proof of Lemma~\ref{endspacesshomeo}, 
it is straightforward to verify that $\tau\pa{FV_{i_m}^n}$ equals the basic open set~\eqref{basisFss}.
Therefore, $\tau^{-1}$ sends the basic open set~\eqref{basisFss} to the basic open set $FV_{i_m}^n$.
Hence, $\tau^{-1}$ is open, proving the claim.
That implies $\tau$ is a continuous bijection.
Hence, it suffices to show that $\tau$ is open or closed.
By Theorem~\ref{endcompthm}, $F\pa{X;\cpa{K_i}}$ is compact and $F\pa{X;\cpa{K_{i_k}}}$ is Hausdorff.
So, $\tau$ is closed.
Alternatively, we may show $\tau$ is open using bases exactly as in the proof of Lemma~\ref{endspacesshomeo}
except using $FV_a^b$ in place of $EV_a^b$.
\end{proof}

\begin{corollary}\label{canhom}
Let $X$ be a generalized continuum and let $\cpa{K_i}$ and $\cpa{L_i}$ be compact exhaustions of $X$.
Then, there exists a unique homeomorphism $h:F\pa{X;\cpa{K_i}} \to F\pa{X;\cpa{L_i}}$ such that $\rest{h}:X\to X$ is the identity.
Hence, $\rest{h}:E\pa{X;\cpa{K_i}} \to E\pa{X;\cpa{L_i}}$ is a canonical homeomorphism.
\end{corollary}

\begin{proof}
For existence of $h$, Lemma~\ref{endcompsshomeo} yields a desired homeomorphism for a subsequence of a given compact exhaustion.
The general case follows by four applications of passing to a subsequence and using the ``W'' argument from Remark~\ref{Wargument}.
By Theorem~\ref{endcompthm}, $X$ is dense in $F\pa{X;\cpa{K_i}}$
and the codomain $F\pa{X;\cpa{L_i}}$ of $h$ is Hausdorff.
So, $h$ is unique.
\end{proof}

Let $X$ be a generalized continuum.
Let $F(X)$ denote Freudenthal's \deffont{endpoint compactification} of $X$ and $E(X)$ denote the \deffont{end space} of $X$.
By definition, $F(X)=F\pa{X;\cpa{K_i}}$ and $E(X)=E\pa{X;\cpa{K_i}}$
for some compact exhaustion of $X$---which exists by Theorem~\ref{gceee}.
The spaces $F(X)$ and $E(X)$ are well-defined up to canonical homeomorphism by Corollary~\ref{canhom}.
Next, we show that each proper map of generalized continua induces a unique map of their endpoint compactifications and a unique map of their end spaces.

\begin{lemma}\label{invimceisce}
Let $X$ be a space with a compact exhaustion $\cpa{K_i}$.
If $f:Y\to X$ is a proper map, then $\cpa{f^{-1}\pa{K_i}}$ is a compact exhaustion of $Y$.
\end{lemma}

\begin{proof}
As $f$ is proper, each $f^{-1}\pa{K_i}$ is compact.
Evidently, the spaces $f^{-1}\pa{K_i}$ are nested and cover $Y$.
As $K_i \subseteq \interior{K_{i+1}}$, we have $f^{-1}\pa{K_i} \subseteq f^{-1}\pa{\interior{K_{i+1}}}$.
The latter subspace is open in $Y$ and contained in $f^{-1}\pa{K_{i+1}}$.
Thus
\[
f^{-1}\pa{K_i} \subseteq f^{-1}\pa{\interior{K_{i+1}}} \subseteq \interior{ f^{-1}\pa{K_{i+1}}}
\]
as desired.
\end{proof}

In Lemma~\ref{invimceisce}, one cannot ensure that $\cpa{f^{-1}\pa{K_i}}$ is efficient even if $\cpa{K_i}$ is efficient---see Example~\ref{invimceex}.

\begin{lemma}\label{inducedmaps}
Let $f:Y\to X$ be a proper map of generalized continua.
Then, there exists a unique map $F(f):F(Y) \to F(X)$ such that $\rest{F(f)}Y=f$.
That map $F(f)$ sends $E(Y)$ into $E(X)$.
Hence, there is a canonical map $E(f)=\rest{F(f)}: E(Y) \to E(X)$ induced by $f$.
Further, any choices of compact exhaustions of $X$ yield a canonical commutative diagram
\begin{equation}\label{Efds}
\begin{tikzcd}
F\pa{Y;\cpa{f^{-1}(K_i)}} \arrow[r,"F(f)"] \arrow{d}{\approx} & F\pa{X;\cpa{K_i}} \arrow{d}{\approx}\\
F\pa{Y;\cpa{f^{-1}(L_i)}} \arrow[r,"F(f)"] & F\pa{X;\cpa{L_i}}
\end{tikzcd}
\end{equation}
where the vertical homeomorphisms are given by Corollary~\ref{canhom}.
\end{lemma}

\begin{proof}
By Theorem~\ref{gceee}, there exists an efficient compact exhaustion $\cpa{K_i}$ of $X$.
By Lemma~\ref{invimceisce}, $\cpa{f^{-1}\pa{K_i}}$ is a compact exhaustion of $Y$.
Basic open sets in $F(X)$ meeting $E(X)$ will be denoted $FV_k^l$ as usual.
For each $i\in\Z_+$, we have $X=K_i\sqcup V_i$ where each component $V_i^j$ of $V_i$ is unbounded.
So, $Y=f^{-1}(K_i) \sqcup f^{-1}(V_i)$ where $f^{-1}(V_i)$ may contain both bounded and unbounded components.
We denote an unbounded component of $f^{-1}(V_i)$ by $U_i^j$.

For existence of $F(f)$, define $F(f)$ on $Y$ to be $f:Y\to X$.
Define $F(f)$ on $E(Y)$ as follows.
Let $\alpha=\pa{U_1^{j_1},U_2^{j_2},\ldots} \in E\pa{Y; \cpa{f^{-1}\pa{K_i}}}$ be an end of $Y$.
For each $i\in\Z_+$, $f\pa{U_i^{j_i}}$ lies in a unique component $V_i^{j_i}$ of $V_i$ by connectedness.
(Even if $\cpa{K_i}$ is not efficient, $f\pa{U_i^{j_i}}$ must lie in an unbounded component of $V_i$ since $f$ is proper
and $U_i^{j_i}$ is unbounded.)
As $\alpha$ is an end, we have $U_1^{j_1}\supseteq U_2^{j_2} \supseteq U_3^{j_3} \supseteq \cdots$.
Therefore, we have
\begin{equation}\label{eq:diagramofsurface}
\begin{tikzcd}[row sep=small]
f\pa{U_1^{j_1}} \arrow[d, phantom, sloped, "\subseteq"] & 
	\arrow[l, phantom, "\supseteq"] \arrow[d, phantom, sloped, "\subseteq"] f\pa{U_2^{j_2}} &
	\arrow[l, phantom, "\supseteq"] \arrow[d, phantom, sloped, "\subseteq"] f\pa{U_3^{j_3}} &
	\arrow[l, phantom, "\supseteq"] \cdots\\
V_1^{j_1} &  V_2^{j_2} & V_3^{j_3} & \cdots
\end{tikzcd}
\end{equation}
Lemma~\ref{ebasislemma} implies that $V_1^{j_1}\supseteq V_2^{j_2} \supseteq V_3^{j_3} \supseteq \cdots$.
Thus, $\beta=\pa{V_1^{j_1},V_2^{j_2},\ldots} \in E\pa{X; \cpa{K_i}}$ is an end of $X$.
We define $F(f)(\alpha)=\beta$.

We verify $F(f)$ is continuous using the standard basis for $F(X)$.
By our piecewise definition of $F(f)$, we need only consider basic open sets $FV_k^l$.
Recall that $FV_k^l = V_k^l \sqcup EV_k^l$.
So, $F(f)^{-1}\pa{FV_k^l} = f^{-1}\pa{V_k^l} \sqcup F(f)^{-1}\pa{EV_k^l}$.
Of the components of $f^{-1}(V_k^l)$, let $U_k^1,\ldots,U_k^{m_k}$ be the unbounded ones and let $B$ be the union of the bounded ones.
So, $f^{-1}\pa{V_k^l}= \bigsqcup_{b=1}^{m_k} U_k^b \sqcup B$.
By the defintion of $F(f)$ on ends, we have
\[
F(f)^{-1}\pa{EV_k^l} 	= \cpa{\pa{Z_1,Z_2,\ldots} \in E\pa{Y;\cpa{f^{-1}\pa{K_i}}} \mid f\pa{Z_k}\subseteq V_k^l} = \bigsqcup_{b=1}^{m_k} EU_k^b
\]
Hence, $F(f)^{-1}\pa{FV_k^l} = \bigsqcup_{b=1}^{m_k} FU_k^b \sqcup B$
is the disjoint union of $B$ and basic open sets of $F(Y)$, which is open in $F(Y)$.
Therefore, $F(f)$ is continuous.

Uniqueness of $F(f)$ such that $\rest{F(f)}Y=f$ follows from the facts that, by Theorem~\ref{endcompthm}, $Y$ is dense in $F\pa{Y}$
and the codomain $F\pa{X}$ of $F(f)$ is Hausdorff.
\end{proof}

\begin{remark}\label{fextend}
Uniqueness in Lemma~\ref{inducedmaps} proves that any continuous extension of $f$ to $F(Y)\to F(X)$
must send $E(Y)$ to a subset of $E(X)$.
That also follows directly using properness of $f$ and local compactness of $X$.
\end{remark}

The next lemma presents a useful criterion for verifying an end lies in the image of the induced map on end compactifications.

\begin{lemma}\label{endinimage}
Let $f:Y\to X$ be a proper map of generalized continua.
Consider an end $\e=\pa{V_1^{j_1},V_2^{j_2},\ldots}\in E\pa{X;\cpa{K_i}}$ where $\cpa{K_i}$ is a compact exhaustion of $X$.
Then, $\e\in\im{F(f)}$ if and only if $\im{f}$ meets $V_i^{j_i}$ for each $i\in\Z_+$.
\end{lemma}

\begin{proof}
For the forward implication, we are given that there exists
$\alpha=\pa{U_1^{j_1},U_2^{j_2},\ldots} \in E\pa{Y; \cpa{f^{-1}\pa{K_i}}}$ such that $F(f)(\alpha)=\e$.
By the definition of $F(f)$, $f(U_i^{j_i})\subseteq V_i^{j_i}$ for each $i\in\Z_+$.
So, $\im{f}$ meets $V_i^{j_i}$ for each $i\in\Z_+$.

We prove the contrapositive of the reverse implication.
By Theorem~\ref{endcompthm}, $F(Y)$ is compact and $F(X)$ is Hausdorff.
By Lemma~\ref{inducedmaps}, $F(f)$ is continuous.
So, $F(f)$ is a closed map.
As $\e\notin \im{F(f)}$, there exists, for some $i\in\Z_+$, a basic open neighborhood $FV_i^{j_i}=V_i^{j_i}\sqcup EV_i^{j_i}$
of $\e$ in $F(X)$ that is disjoint from $\im{F(f)}$.
In particular, $V_i^{j_i}$ is disjoint from $\im{f}$.
\end{proof}

\begin{theorem}\label{fefunctors}
Freudenthal's end compactification is a covariant functor $F$ from the category of generalized continua and proper maps to the category of
compact generalized continua and (proper) maps.
Freudenthal's end space is a covariant functor $E$ from the category of generalized continua and proper maps to the category of
Hausdorff, totally separated, compact spaces and (proper) maps.
\end{theorem}

\begin{proof}
By the results in this section, it remains to show that $F$ and $E$ preserve identity maps and compositions of proper maps.
Consider $\tn{id}:X\to X$ where $X$ is a generalized continuum.
By Theorem~\ref{endcompthm}, $X$ is dense in $F\pa{X}$ and $F\pa{X}$ is Hausdorff.
Thus, $F\pa{\tn{id}}=\tn{id}$ and $E\pa{\tn{id}}=\rest{F\pa{\tn{id}}}=\tn{id}$.
Consider proper maps $g:Z\to Y$ and $f:Y\to X$ where $X$, $Y$, and $Z$ are generalized continua.
Both $F\pa{g\circ f}$ and $F\pa{g}\circ F\pa{f}$ are maps $F(Z)\to F(X)$ that extend $g\circ f: Z\to X$.
By Theorem~\ref{endcompthm}, $Z$ is dense in $F\pa{Z}$ and $F\pa{X}$ is Hausdorff.
Hence, $F\pa{g\circ f} = F\pa{g}\circ F\pa{f}$.
As $E(g\circ f)=\rest{F(g\circ f)}:E(Z) \to E(X)$,
$E(g)=\rest{F(g)}:E(Z) \to E(Y)$, and
$E(f)=\rest{F(f)}:E(Y) \to E(X)$,
we get $E\pa{g\circ f} = E\pa{g}\circ E\pa{f}$.
\end{proof}

\begin{example}\label{noproperretractplanetoxaxis}
While the plane $\R^2$ properly retracts to the nonnegative $x$-axis, the plane does not properly retract to the entire $x$-axis $\R\times\cpa{0}$.
Suppose, by way of contradiction, that $\rho:\R^2\to\R\times\cpa{0}$ is a proper retract.
Let $i:\R\times\cpa{0} \hookrightarrow \R^2$ be the inclusion map, which is proper since $\R\times\cpa{0}$ is closed in $\R^2$.
Apply the end functor $E$ to the commutative diagram
\begin{equation}\label{r2retractxaxis}
\begin{tikzcd}
\R\times\cpa{0} \arrow[r, hookrightarrow, "i"] \arrow[rr,bend right, "\tn{id}"] & \R^2 \arrow[r,"\rho"] & \R\times\cpa{0}
\end{tikzcd}
\end{equation}
to get that the identity map between a two-point space factors through a one-point space, a contradiction.
\end{example}

Next, we show that properly homotopic maps induce the same map on end spaces.
For that purpose, we show that if $Y$ and $J$ are generalized continua and $J$ is compact,
then the end spaces of $Y$ and $Y\times J$ are canonically homeomorphic.

\begin{lemma}\label{compsprod}
Let $Z$ and $J$ be spaces where $J$ is compact and connected.
Then:
\begin{enumerate*}[label=(\roman*)]
  \item\label{compsZxI} each component of $Z\times J$ has the form $C\times J$ for some component $C$ of $Z$, and
  \item\label{bdedcompsZxI} a component $C\times J$ of $Z \times J$ is bounded if and only if $C$ is bounded in $Z$.
\end{enumerate*}
\end{lemma}

\begin{proof}
Let $C$ be a component of $Z$.
As $C\times J$ is connected, $C\times J$ lies in a unique component $E$ of $Z\times J$.
Let $p:Z\times J\to Z$ be projection which is continuous.
As $\cpa{z}\times J$ is connected for each $z\in Z$, we see that $E=p(E)\times J$.
Thus, $C\subseteq p(E)$, $p(E)$ is connected, and $p(E)\subseteq C$. Therefore, $E=C\times J$ proving~\ref{compsZxI}.
For~\ref{bdedcompsZxI}, note that $\closure{C\times J} = \closure{C}\times\closure{J}$.
\end{proof}

The proof of the following lemma is straightforward.

\begin{lemma}\label{AxBgc}
If $A$ and $B$ are generalized continua, then $A\times B$ is a generalized continuum.
\end{lemma}

\begin{lemma}\label{gcprodIlemma}
Let $Y$ and $J$ be generalized continua where $J$ is compact.
Then:
\begin{enumerate*}[label=(\roman*)]
  \item\label{YxI} $Y\times J$ is a generalized continuum,
  \item\label{prodce} if $\cpa{K_i}$ is a compact exhaustion of $Y$, then $\cpa{K_i \times J}$ is a compact exhaustion of $Y\times J$, and
	\item\label{prodece} if $\cpa{K_i}$ is an efficient compact exhaustion of $Y$, then $\cpa{K_i \times J}$ is an efficient compact exhaustion of $Y\times J$.
\end{enumerate*}
\end{lemma}

\begin{proof}
Lemma~\ref{AxBgc} implies~\ref{YxI}.
For~\ref{prodce}, note that $K_i \subseteq \interior{K_{i+1}}$ for each $i\in\Z_+$, so $K_i \times J \subseteq \interior{K_{i+1}} \times J$.
The latter subspace is open in $Y\times J$ and contained in $K_{i+1} \times J$.
Thus
\[
K_i \times J \subseteq \interior{K_{i+1}} \times J \subseteq \interior{\pa{K_{i+1} \times J}}
\]
as desired.
Lastly, let $\cpa{K_i}$ be an efficient compact exhaustion of $Y$.
By~\ref{prodce}, $\cpa{K_i \times J}$ is a compact exhaustion of $Y\times J$.
For each $i\in\Z_+$, $K_i$ is connected, so $K_i\times J$ is connected.
Notice that $Y\times J - K_i \times J = \pa{Y-K_i} \times J$ as spaces.
As $Y-K_i$ has only unbounded components, Lemma~\ref{compsprod} with $Z=Y-K_i$ implies that
$Y\times J - K_i \times J$ has only unbounded components, proving~\ref{prodece}.
\end{proof}

\begin{lemma}\label{endsYxI}
Let $Y$ and $J$ be generalized continua where $J$ is compact.
For an arbitrary but fixed $t\in J$, define $\iota_t:Y \rightarrowtail Y\times J$ by $\iota_t(y)=(y,t)$.
Let $p:Y\times J \to Y$ be projection.
Consider the commutative diagram
\begin{equation}\label{Ydiag}
\begin{tikzcd}
Y \arrow[r,"\iota_t"] \arrow[rr,bend right, "\tn{id}"] & Y\times J \arrow[r,"p"] & Y
\end{tikzcd}
\end{equation}
Then:
\begin{enumerate*}[label=(\roman*)]
  \item\label{pcpm} $p$ is a closed, proper map,
  \item\label{itcpm} $\iota_t$ is a closed, proper map, and
	\item\label{Ephomeo} in the induced diagram
\end{enumerate*}
\begin{equation}\label{EYdiag}
\begin{tikzcd}
E(Y) \arrow[r,"E(\iota_t)"] \arrow[rr,bend right, "\tn{id}"] & E(Y\times J) \arrow[r,"E(p)"] & E(Y)
\end{tikzcd}
\end{equation}
$E(p)$ is a homeomorphism and, hence, $E(\iota_t)=E(p)^{-1}$ is a homeomorphism.
In particular, $E(\iota_s)=E(\iota_t)$ for all $s,t\in J$.
\end{lemma}

\begin{proof}
By Lemma~\ref{gcprodIlemma}\ref{YxI}, $Y\times J$ is a generalized continuum.
To see that $p$ is closed, use the fact that $J$ is compact and the tube lemma~\cite[p.~168]{munkres}.
The image of $\iota_t$ is $\im{\iota_t}=Y\times \cpa{t}$ which is closed in $Y\times J$.
A simple argument shows that $\iota_t$ is closed.
Observe that for each $A\subseteq Y\times J$, we have $\iota_t^{-1}(A)=p\pa{\im{\iota_t} \cap A}$.
That observation, together with the facts that $\im{\iota_t}$ is closed in $Y\times J$ and $Y$ and $J$ are Hausdorff, show that $\iota_t$ is proper.
Now, the end space functor $E$ from Theorem~\ref{fefunctors} yields the commutative diagram~\eqref{EYdiag}.
That diagram implies that $E(p)$ is surjective.
By Theorem~\ref{fefunctors}, $E(p)$ is continuous with compact domain and Hausdorff codomain.
So, it suffices to prove that $E(p)$ is injective.

By Theorem~\ref{gceee}, there exists an efficient compact exhaustion $\cpa{K_i}$ of $Y$.
By Lemma~\ref{gcprodIlemma}\ref{prodece}, $\cpa{K_i \times J}$ is an efficient compact exhaustion of $Y\times J$.
For each $i\in\Z_+$, $V_i=Y-K_i$ has components $\cpa{V_i^j}$ and all are unbounded (since $K_i$ is efficient).
By Lemma~\ref{compsprod}, $Y\times J - K_i\times J = V_i \times J$ has components $\cpa{V_i^j \times J}$ and all are unbounded.
Suppose $\alpha=\pa{V_1^{a_1}\times J,V_2^{a_2}\times J,\ldots}$
and $\beta=\pa{V_1^{b_1}\times J,V_2^{b_2}\times J,\ldots}$
are ends of $Y\times J$ and $E(p)(\alpha)=E(p)(\beta)$.
For each $i\in\Z_+$, $p\pa{V_i^{a_i}\times J} = V_i^{a_i}$ and $p\pa{V_i^{b_i}\times J} = V_i^{b_i}$.
So, $E(p)(\alpha)=\pa{V_1^{a_1},V_2^{a_2},\ldots}$ and
$E(p)(\beta)=\pa{V_1^{b_1},V_2^{b_2},\ldots}$.
As $E(p)(\alpha)=E(p)(\beta)$, we see that $V_i^{a_i} = V_i^{b_i}$ for each $i\in\Z_+$.
Hence, $\alpha=\beta$ and $E(p)$ is injective, completing the proof.
\end{proof}

\begin{corollary}\label{phemaps}
Let $f,g:Y\to X$ be proper maps of generalized continua.
If $f$ and $g$ are properly homotopic, then $E(f)=E(g)$.
\end{corollary}

\begin{proof}
Let $H:Y\times I \to X$ be a proper map such that $H(y,0)=f(y)$ and $H(y,1)=g(y)$ for each $y\in Y$.
Using the end space functor, we have commutative diagrams
\begin{equation}\label{fgH}
\begin{tikzcd}[row sep=3]
Y \arrow[rd, "\iota_0"'] \arrow[rrd,bend left, "f"'] & & & E(Y) \arrow[rd, "E(\iota_0)"'] \arrow[rrd,bend left, "E(f)"']\\
& Y\times I \arrow[r,"H"] & X & & E(Y\times I) \arrow[r,"E(H)"] & E(X)\\
Y \arrow[ru, "\iota_1"] \arrow[rru,bend right,"g"] & & & E(Y) \arrow[ru, "E(\iota_1)"] \arrow[rru,bend right,"E(g)"]
\end{tikzcd}
\end{equation}
By Lemma~\ref{endsYxI}\ref{Ephomeo}, $E(\iota_0)=E(\iota_1)$.
Hence,~\eqref{fgH} implies that $E(f)=E(g)$ as desired.
\end{proof}

As an application of the theory of ends, the remainder of this section is devoted to
proving Freudenthal's beautiful and fundamental theorem on the possible number of ends of a topological group.
A space $X$---always assumed Hausdorff---equipped with a group structure is a \deffont{topological group} provided the binary group operation
and the unary inverse operation
\begin{equation}\label{topgroupmultinv}
\begin{tikzcd}[row sep=0]
X\times X \arrow[r,"\mu"] & X				& X \arrow[r,"\iota"] & X\\
(x,y) \arrow[r, mapsto] & x\cdot y	& x \arrow[r, mapsto] & x^{-1}
\end{tikzcd}
\end{equation}
are continuous.
In any topological group $X$, we let $e\in X$ denote the identity element---not to be confused with an end of $X$.
See Lee~\cite[pp.~77--80 \& 311--336]{lee} for an introduction to topological groups.
If $X$ is a topological group, then inversion $\iota:X\to X$ is a homeomorphism, as is left translation by each $g\in X$
\begin{equation}\label{lefttranslation}
\begin{tikzcd}[row sep=0]
X \arrow[r,"L_g"] & X\\
x \arrow[r, mapsto] & g\cdot x
\end{tikzcd}
\end{equation}
So, $X$ acts on itself continuously, freely, and transitively by left translation.
In particular, $X$ is topologically homogeneous~\cite[p.~78]{lee}.

\begin{examples}\label{topgroupexamples}
The following are topological groups.
\begin{enumerate}[label=(\alph*)]
\item Every group with the discrete topology.
\item The unit circle $S^1\subseteq \C$ under multiplication.
\item The unit sphere $S^3\subseteq \H$ in the real quaternions under multiplication.
\item Euclidean space $\R^n$ under addition for each $n\in\Z_+$.
\item The product of two topological groups.
\item The real and complex general linear groups $\tn{GL}(n)$ and orthogonal groups $\tn{O}(n)$ for each $n\in\Z_+$.
\end{enumerate}
\end{examples}

Note that the unit circle $S^1\subseteq \C$ has zero ends,
$\R^2$ under addition has one end, and $\R$ under addition has two ends.
Freudenthal's theorem states that every topological group---that is path connected and a generalized continuum---has at most two ends.
We encourage the reader unfamiliar with Freudenthal's theorem to skip forward to Theorem~\ref{endstopgroup},
grasp the main idea of the proof (including why it does not apply to groups such as $\R^2$ and $S^1\times\R$),
and then return here for further necessary details.

\begin{lemma}\label{ABandAinvcompact}
Let $X$ be a topological group.
If $A,B\subseteq X$ are compact, then $A\cdot B =\cpa{a\cdot b \mid a\in A \tn{ and } b\in B}$ and $A^{-1}=\cpa{a^{-1} \mid a \in A}$ are compact.
\end{lemma}

\begin{proof}
As $A\times B$ is compact, $\mu\pa{A\times B} = A \cdot B$ is compact.
As $A$ is compact, $\iota\pa{A}=A^{-1}$ is compact.
\end{proof}

The action of a topological group $X$ on itself by left translation is \deffont{proper} provided the map
\begin{equation}\label{hproper}
\begin{tikzcd}[row sep=0]
X \times X \arrow[r,"h"] & X \times X\\
(g,x) \arrow[r, mapsto] & (g\cdot x, x)
\end{tikzcd}
\end{equation}
is proper (see Lee~\cite[p.~319]{lee}).

\begin{lemma}\label{ltproperaction}
Let $X$ be a topological group.
Then, the action of $X$ on itself by left translation is proper.
\end{lemma}

\begin{proof}
Let $V\subseteq X$ be compact.
Then, we have
\[
h^{-1}\pa{V\times V} = \cpa{(g,x)\in X\times X \mid g\cdot x \in V \tn{ and } x\in V} \subseteq \pa{V\cdot V^{-1}} \times V
\]
and the latter space is compact by Lemma~\ref{ABandAinvcompact}.
As $h^{-1}\pa{V\times V}$ is also closed in $X\times X$, we get that $h^{-1}\pa{V\times V}$  is compact.
Next, let $K\subseteq X\times X$ be compact.
Let $\tn{pr}_1: X\times X \to X$ and $\tn{pr}_2:X\times X\to X$ be the projection maps to the first and second coordinate respectively.
Let $U=\tn{pr}_1(K) \cup \tn{pr}_2(K)$ which is compact, and note that $K\subseteq U\times U$.
As $h^{-1}(K)$ is closed in $X\times X$ and contained in the compactum $h^{-1}(U\times U)$, we see that $h^{-1}(K)$ is compact, as desired.
\end{proof}

\begin{remark}
The action of $X$ on itself by left translation being proper does not imply that the group operation $\mu:X\times X\to X$ is proper.
Consider $\R$ under addition.
On the other hand, the converse of that implication is true (see also Lee~\cite[p.~319]{lee}).
\end{remark}

\begin{corollary}\label{X_Kcompact}
Let $X$ be a topological group.
Then, for each compactum $K\subseteq X$, the subspace
\[
X_K = \cpa{g\in X \mid g\cdot K \cap K \neq\emptyset} \subseteq X
\]
is compact.
\end{corollary}

\begin{proof}
Let $K\subseteq X$ be compact.
By Lemma~\ref{ltproperaction}, the map $h$ is proper.
So, $h^{-1}(K\times K)$ is compact.
Let $\tn{pr}_1: X\times X \to X$ be the projection map to the first coordinate.
Then, $\tn{pr}_1\pa{h^{-1}(K\times K)}=X_K$ is compact, as desired.
\end{proof}

\begin{lemma}\label{pctgph}
Let $X$ be a path connected topological group.
If $g\in X$, then left translation $L_g$ by $g$ is properly homotopic to the identity $L_e=\tn{id}:X\to X$.
\end{lemma}

\begin{proof}
As $X$ is path connected, there exists a continuous path $\gamma:\br{0,1}\to X$ from $e$ to $g$.
Define $H:X\times I \to X$ by $H(x,t)=\gamma(t)\cdot x$.
The function $H$ is continuous since it equals the following composition of maps
\begin{equation}\label{Hmap}
\begin{tikzcd}[row sep=0]
X \times I \arrow[r] & X \times X \arrow[r,"\mu"] & X\\
(x,t) \arrow[r, mapsto] & (\gamma(t),x) \arrow[r, mapsto] & \gamma(t)\cdot x
\end{tikzcd}
\end{equation}
Let $K\subseteq X$ be compact.
As $H$ is continuous, $H^{-1}(K)$ is closed in $X\times I$.
Lemma~\ref{ABandAinvcompact} implies that $\pa{\im{\gamma}}^{-1} \cdot K$ is compact.
So, $\pa{\pa{\im{\gamma}}^{-1} \cdot K}\times I \subseteq X\times I$ is compact.
As $H^{-1}(K)$ is contained in the compactum $\pa{\pa{\im{\gamma}}^{-1} \cdot K}\times I$ and is closed in $X\times I$,
we see that $H^{-1}(K)$ is compact, as desired.
\end{proof}

We now have the tools to prove Freudenthal's fundamental theorem~\cite{freudenthal31}.

\begin{theorem}[Freudenthal]\label{endstopgroup}
Let $X$ be a path connected generalized continumm that is a topological group.
Then, $X$ has at most two ends.
\end{theorem}

\begin{proof}
Suppose, by way of contradiction, that $X$ has three ends $\e_1$, $\e_2$, and $\e_3$.
Let $\cpa{K_i}$ be an efficient compact exhaustion of $X$ given by Theorem~\ref{gceee}.
By Theorem~\ref{endsthm}\ref{finiteendsbondsbij}, there exists $k\in\Z_+$ such that $V_k=X-K_k$ has three components $V_k^1$, $V_k^2$, and $V_k^3$,
each unbounded in $X$.
See Figure~\ref{fig:threeends} (left) where $K_k$ is depicted in blue.
\begin{figure}[htbp!]
    \centerline{\includegraphics[scale=1.0]{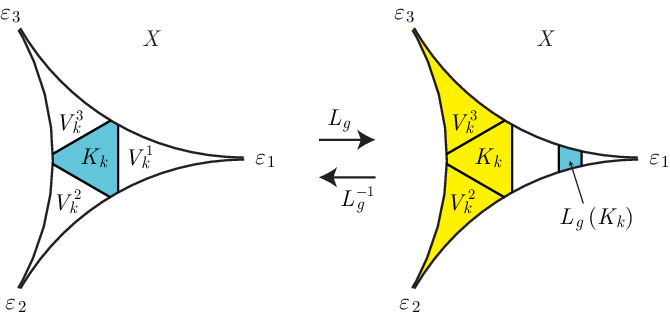}}
    \caption{Compactum $K_k$ (blue) in three-ended space $X$ (left). Image of $K_k$ (blue) under left translation $L_g$ of $X$ and
		connected subspace $K_k \cup V_k^2 \cup V_k^3$ (yellow) of $X$.}
\label{fig:threeends}
\end{figure}
Corollary~\ref{X_Kcompact} implies that $X_{K_k}$ is compact.
As $X$ has an end, $X$ is noncompact by Theorem~\ref{endsthm}\ref{noncompactexistsend}.
So, there exists $g\in X$ such that $g\cdot K_k$ is disjoint from $K_k$.
That is, left translation $L_g$ is a self-homeomorphism of $X$ that carries the connected compactum $K_k$ into one of the
three components $V_k^1$, $V_k^2$, or $V_k^3$.
Without loss of generality, assume $L_g\pa{K_k} \subseteq V_k^1$ as in Figure~\ref{fig:threeends} (right).
Lemma~\ref{pctgph} implies that $L_g$ and $L_g^{-1}=L_{g^{-1}}$ are both properly homotopic to the identity $\tn{id}:X\to X$.
Corollary~\ref{phemaps} implies that $L_g$ and $L_g^{-1}$ both induce the identity map on the end space of $X$.
Consider the subspace $K_k \cup V_k^2 \cup V_k^3$ of $X$ depicted in yellow in Figure~\ref{fig:threeends} (right).
That subspace is connected by Lemma~\ref{compslemma}\ref{KconnKunioncompsconn}.
Thus, $L_g^{-1}$ carries $K_k \cup V_k^2 \cup V_k^3$ into one of the components $V_k^1$, $V_k^2$, or $V_k^3$ of $X-K_k$.
As $E\pa{L_g^{-1}}\pa{\e_2}=\e_2$, the definition of $E\pa{L_g^{-1}}$ (see Lemma~\ref{inducedmaps})
implies that $L_g^{-1}$ carries $K_k \cup V_k^2 \cup V_k^3$ into $V_k^2$.
Similarly, $E\pa{L_g^{-1}}\pa{\e_3}=\e_3$ implies that $L_g^{-1}$ carries $K_k \cup V_k^2 \cup V_k^3$ into $V_k^3$.
That contradiction completes the proof that the topological group $X$ cannot have three ends.
It is straightforward to adapt the proof to the general case where $X$ has any number of ends greater than two.
\end{proof}

\section{Compact exhaustions of maps}
\label{sec:cem}

We consider spaces $X$ containing a baseray.
A \deffont{baseray} is a proper embedding $r:\nnr \rightarrowtail X$.
A \deffont{ray-based space} is a pair $(X,r)$ where $X$ is a (necessarily noncompact) space
and $r$ is a baseray.
We often refer to $r$ and its image in $X$ interchangeably, where no confusion should arise.
It will be convenient to have an efficient compact exhaustion of $X$ that simultaneously exhausts $r$.
We prove a more general fact that we regard as an \emph{exhaustion of a map}.

Let $f:Y\to X$ be a map of spaces with image denoted $\im{f}$.
A compact exhaustion $\cpa{K_i}$ of $X$ is \deffont{$f$-efficient} provided:
\begin{enumerate*}[label=(\roman*)]
  \item\label{Kieff} $\cpa{K_i}$ is an efficient compact exhaustion of $X$,
  \item\label{imageece} $\cpa{\im{f} \cap K_i}$ is an efficient compact exhaustion of $\im{f}$, and
	\item\label{invimagece} $\cpa{f^{-1}\pa{K_i}}$ is a compact exhaustion of $Y$.
\end{enumerate*}

In Theorem~\ref{exhaustmap}, we prove that each proper map $f:Y\to X$ of generalized continua admits an $f$-efficient compact exhaustion.
Note that part~\ref{invimagece} of the definition of $f$-efficient
does not require the compact exhaustion $\cpa{f^{-1}\pa{K_i}}$ of $Y$ to be efficient.
The next example shows that conclusion is generally not possible.
If $f$ is also injective, then Theorem~\ref{exhaustmap} shows that conclusion does hold.

\begin{example}\label{invimceex}
Consider the piecewise linear, proper map $f:\nnr\to \R$ in Figure~\ref{fig:doublebackmapgraph}.
\begin{figure}[htbp!]
    \centerline{\includegraphics[scale=1.0]{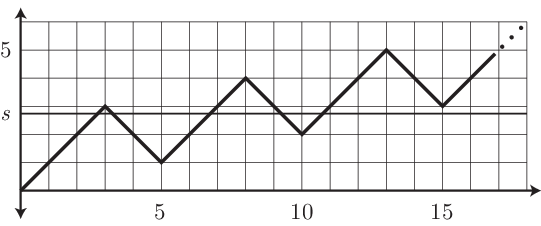}}
    \caption{Graph of a piecewise linear, proper map $f:\nnr\to \R$.}
\label{fig:doublebackmapgraph}
\end{figure}
Here, $f$ repeatedly doubles back moving three units positively and then two units negatively.
Let $K\subseteq\R$ be any compactum with $s=\sup K\geq 2$.
Notice that $f^{-1}(s)\subseteq f^{-1}(K) \subseteq f^{-1}\pa{\br{0,s}}$
where $f^{-1}\pa{\br{0,s}}$ has three connected components and each component meets $f^{-1}(s)$.
So, $f^{-1}(K)$ is disconnected and not efficient.
Thus, $\cpa{f^{-1}\pa{K_i}}$ is never efficient for any compact exhaustion $\cpa{K_i}$ of $\R$.
We obtain examples in all dimensions by composing $f$ with the inclusion $i: \R \to \R\times\R^k$ given by $t\mapsto (t,0)$
or by taking the product of $f$ with the identity on $[0,1]^k$ for any $k\in\Z_+$.
\end{example}

\begin{lemma}\label{propermapgc}
Let $f:Y\to X$ be a proper map of generalized continua.
Then:
\begin{enumerate*}[label=(\roman*)]
  \item\label{fclosed} $f$ is a closed map, and
  \item\label{imfgc} $\im{f}$ is a generalized continuum.
\end{enumerate*}
\end{lemma}

\begin{proof}
As $f$ is continuous, $Y$ is connected, and $X$ is Hausdorff, we get $\im{f}$ is connected and Hausdorff.
As $f$ is proper and $X$ is locally compact, $f$ is closed proving~\ref{fclosed}.
In particular, $\im{f}$ is closed in $X$.
So, $\im{f}$ is $\sigma$-compact and locally compact.
As $f$ is closed and $\im{f}$ is closed in $X$, the restriction map $f|: Y \to \im{f}$ is closed.
As $f|$ is a closed, surjective map, it is a quotient map.
Thus, $\im{f}$ is locally connected by Dugundji~\cite[p.~125]{dugundji}.
That completes the proof of~\ref{imfgc}.
\end{proof}

\begin{lemma}\label{Acapce}
Let $X$ be a space with a compact exhaustion $\cpa{K_i}$.
If $A\subseteq X$ is closed, then $\cpa{A \cap K_i}$ is a compact exhaustion of $A$.
\end{lemma}

\begin{proof}
As $A$ is closed in $X$, $X$ is Hausdorff, and each $K_i$ is compact, we get $A\cap K_i$ is compact.
Evidently, the spaces $A\cap K_i$ are nested and cover $A$.
As $K_i \subseteq \interior{K_{i+1}}$, and $A \cap \interior{K_{i + 1}}$ is open in $A$ and contained in $A \cap K_{i+1}$, we have
\[A \cap K_i \subseteq A \cap \interior{K_{i+1}} \subseteq \interior{\pa{A \cap K_{i+1}}}\]
as desired.
\end{proof}

\begin{lemma}\label{AXece}
Let $X$ be a space and $A\subseteq X$ a closed subspace.
Suppose we are given efficient compact exhaustions $\cpa{J_i}$ of $X$ and $\cpa{L_i}$ of $A$.
Then, there exists an efficient compact exhaustion $\cpa{K_i}$ of $X$
such that $\cpa{A\cap K_i}$ is a subsequence of $\cpa{L_i}$.
In particular, $\cpa{A\cap K_i}$ is an efficient compact exhaustion of $A$.
\end{lemma}

\begin{proof}
As $X$ is Hausdorff and $A$ is closed in $X$, the intersection of $A$ with each compactum in $X$ is compact.
In particular, $A\cap J_i$ is compact for each $i\in\Z_+$.
Reasoning inductively, we replace $\cpa{J_i}$ and $\cpa{L_i}$ with subsequences so that
\begin{equation}\label{feffshuffle}
L_i\subseteq\interior{J_i} \ \  \tn{and} \ \ A \cap J_i \subseteq\interior{L_{i+1}} \ \ \tn{for each} \ i\in\Z_+
\end{equation}
where $\interior{J_i}$ is the interior of $J_i$ in $X$ and $\interior{L_{i+1}}$ is the interior of $L_{i+1}$ in $A$.
\begin{figure}[htbp!]
    \centerline{\includegraphics[scale=1.0]{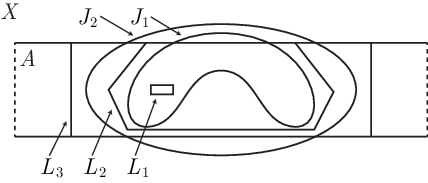}}
    \caption{Closed subspace $A$ of $X$ and shuffled efficient compact exhaustions $\cpa{L_i}$ of $A$ and $\cpa{J_i}$ of $X$.}
\label{fig:shufflieJiLiu}
\end{figure}
In more detail, begin with $L_1$.
As $\cpa{J_i}$ is a compact exhaustion of $X$, there exists some $J_k$ that contains $L_1$ in its interior.
Call that $J_1$.
Then, $A\cap J_1$ is a compact subspace of $A$.
As $\cpa{L_i}$ is a compact exhaustion of $A$, there exists some $L_k$ that contains $A\cap J_1$ in its interior.
Call that $L_2$. Now, repeat that process.
Figure~\ref{fig:shufflieJiLiu} depicts the resulting shuffled efficient compact exhaustions.
For each $i\in\Z_+$, define $K_i=\pa{J_i \cup L_{i+1}}'$ to be the bounded filling of $J_i \cup L_{i+1}$ in $X$.
We will show that $\cpa{K_i}$ is a desired exhaustion of $X$.

Each $J_i$ and $L_i$ is connected (hence, nonempty) and compact.
For each $i\in\Z_+$, $J_i$ and $L_{i+1}$ contain $L_i$, the former by~\eqref{feffshuffle}.
So, $J_i \cup L_{i+1}$ is connected and compact.
By Lemmas~\ref{compslemma}\ref{KconnK'conn} and~\ref{compactlemma}\ref{K'compact},
$K_i=\pa{J_i \cup L_{i+1}}'$ is connected and compact.
By Lemma~\ref{compactlemma}\ref{unbdedcomplK'}, $X-K_i$ has only unbounded components.

As $J_i$ and $L_{i+1}$ are both contained in $\interior{J_{i+1}}$, the latter by~\eqref{feffshuffle},
we get
\[
J_i \cup L_{i+1} \subseteq \interior{J_{i+1}} \subseteq J_{i+1} \cup L_{i+2}
\]
As $\interior{J_{i+1}}$ is open in $X$,
the previous equation implies that $J_i \cup L_{i+1} \subseteq \interior{\pa{J_{i+1} \cup L_{i+2}}}$.
Lemma~\ref{monotonicity} (monotonicity) implies that $K_i\subseteq \interior{K_{i+1}}$.
Hence, $\cpa{K_i}$ is an efficient compact exhaustion of $X$.

We claim that each bounded component of $X-\pa{J_i \cup L_{i+1}}$ is disjoint from $A$.
Before we prove that claim, we show it implies $A\cap K_i = L_{i+1}$,
which in turn implies the remaining desired conclusions in the statement of the lemma.
For each $i\in\Z_+$, we have
\[
K_i = \pa{J_i \cup L_{i+1}}' = \pa{J_i \cup L_{i+1}} \cup \bc{X-\pa{J_i \cup L_{i+1}}}
\]
where $\cup \, \bc{X-\pa{J_i \cup L_{i+1}}}$ is disjoint from $A$ by the claim.
Hence
\[
A\cap K_i	= A \cap \pa{J_i \cup L_{i+1}} = \pa{A \cap J_i} \cup \pa{A \cap L_{i+1}}
					= \pa{A \cap J_i} \cup L_{i+1} = L_{i+1}
\]
where the third equality holds since $L_{i+1}\subseteq A$
and the fourth holds by~\eqref{feffshuffle}.
So, it remains to prove the claim.

Suppose, by way of contradiction, that a bounded component $B$ of $X-\pa{J_i \cup L_{i+1}}$ meets $A$ for some $i\in\Z_+$.
Then, $B$ meets a component $U$ of $A-L_{i+1}$ that is necessarily unbounded in $A$.
We have $U\subseteq A-L_{i+1} \subseteq X-\pa{J_i \cup L_{i+1}}$ by~\eqref{feffshuffle}.
Thus, as $U$ is connected and $B$ is a component of $X-\pa{J_i \cup L_{i+1}}$, we have $U\subseteq B$.
As $A$ is closed in $X$, the closure of $U$ in $X$ equals the closure of $U$ in $A$.
However, $U$ is unbounded in $A$ and is bounded in $X$ (since $U\subseteq B$ where $B$ is bounded).
That contradiction proves the claim and the lemma.
\end{proof}

\begin{theorem}\label{exhaustmap}
Let $f:Y\to X$ be a proper map of generalized continua.
Then:
\begin{enumerate*}[label=(\roman*)]
  \item\label{efece} there exists an $f$-efficient compact exhaustion $\cpa{K_i}$ of $X$, and
  \item\label{finjeceY} if, furthermore, $f$ is injective,
	then $f$ is an embedding and $\cpa{f^{-1}\pa{K_i}}$ is an efficient compact exhaustion of $Y$.
\end{enumerate*}
\end{theorem}

\begin{proof}
By Lemma~\ref{propermapgc}\ref{imfgc}, $\im{f}$ is a generalized continuum.
By Theorem~\ref{gceee}, there exist efficient compact exhaustions $\cpa{J_i}$ of $X$ and $\cpa{L_i}$ of $\im{f}$.
By Lemma~\ref{propermapgc}\ref{fclosed}, $\im{f}$ is closed in $X$.
By Lemma~\ref{AXece}, there exists $\cpa{K_i}$ an efficient compact exhaustion of $X$
such that $\cpa{\im{f}\cap K_i}$ is an efficient compact exhaustion of $\im{f}$.
By Lemma~\ref{invimceisce}, $\cpa{f^{-1}\pa{K_i}}$ is a compact exhaustion of $Y$.
Hence, $\cpa{K_i}$ is a desired $f$-efficient compact exhaustion of $X$.

Suppose, further, that $f$ is injective.
As $f$ is a closed map (Lemma~\ref{propermapgc}\ref{fclosed}), $f$ is an embedding.
So, the restriction map $f|^{-1}: \im{f} \to Y$ is a homeomorphism
and preserves connectedness, boundedness, and unboundedness.
Therefore, $\cpa{f^{-1}\pa{K_i}}$ is an efficient compact exhaustion of $Y$.
\end{proof}

\begin{corollary}\label{receexist}
Let $(X,r)$ be a ray-based generalized continuum.
That is, $X$ is a generalized continuum and $r:\nnr \rightarrowtail X$ is a proper embedding.
Then, there exists an $r$-efficient compact exhaustion $\cpa{K_i}$ of $X$.
\end{corollary}

\begin{remarks}\label{rece}\mbox{ \\ }
\begin{enumerate}[label=(\alph*)]
\item\label{injemd} If $X$ is a generalized continuum and $r:\nnr \rightarrowtail X$ is a proper, injective map,
then $f$ is an embedding by Lemma~\ref{propermapgc}\ref{fclosed}.
\item All efficient compact exhaustions of $\nnr$ were described in Remarks~\ref{eceexamples}\ref{ecennr}.
\item While Lemma~\ref{AXece} on exhausting a closed pair $(X,A)$ is used to prove Theorem~\ref{exhaustmap} on exhausting a proper map,
the latter applied to inclusion $i:A \hookrightarrow X$ yields a nice exhaustion of the closed pair $(X,A)$.
\end{enumerate}
\end{remarks}

\section{Baserays}
\label{sec:br}

The previous sections showed that the theory of ends is robust for generalized continua.
However, the following example shows that class of spaces is too broad when a baseray is required.
Recall that a \deffont{baseray} is a proper embedding $r:\nnr \rightarrowtail X$.

\begin{example}\label{nmgc}
Let $X=[0,1)\times[0,1]$ equipped with the dictionary order topology
where $(x_1,y_1)<(x_2,y_2)$ means $x_1<x_2$, or $x_1=x_2$ and $y_1<y_2$.
We prove that $X$ is a one-ended generalized continuum that
admits no proper map $r:\nnr \to X$, is not locally path-connected, does not have a countable basis, and is not metrizable.
The dictionary order on $X$ is a linear order.
Let $x_0=(0,0)\in X$ be the minimal element.
The dictionary order topology on $X$ is defined by the basis consisting of all
open intervals $(a,b)_X$ in $X$ and all half-open intervals $[x_0,b)_X$ in $X$~\cite[p.~84]{munkres}.
\begin{figure}[htbp!]
    \centerline{\includegraphics[scale=1.0]{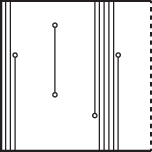}}
    \caption{Three basic open sets for the dictionary order topology on $X=[0,1)\times[0,1]$.}
\label{fig:dictordsquare}
\end{figure}
We use the notation $(a,b)_X$, $\br{a,b}_X$, $[a,b)_X$, and $[a,\infty)_X=\cpa{x\in X \mid a\leq x}$ to denote intervals in $X$ to avoid confusion with intervals in $\R$ and points in $X$.
Figure~\ref{fig:dictordsquare} depicts three basic open sets in $X$.
As $X$ has an order topology, $X$ is Hausdorff.
For each $x\in I=[0,1]\subseteq\R$, the function $I\to X$ defined by $t\mapsto (x,t)$ is an embedding.
The function $X\to I$ defined by $(x,y)\mapsto x$ is continuous.
Using the last two observations, $X$ has the least upper bound property~\cite[pp.~27 \&~155]{munkres}.
Each closed interval $\br{a,b}_X$ in $X$ is compact~\cite[Thm.~27.1]{munkres}.
The space $X$ has the intermediate value property---meaning if $a<b$ in $X$, then there exists $c\in X$ such that $a<c<b$.
As $X$ has an order topology, the least upper bound property, and the intermediate value property,
$X$ is, by definition, a linear continuum~\cite[p.~153]{munkres}.
The linear continuum $X$ is connected as are all nonempty intervals in $X$~\cite[Thm.~24.1]{munkres}.
As all nonempty basic open sets in $X$ are connected, $X$ is locally connected.
For each $i\in\Z_+$, define $K_i = \br{0,1-1/i}\times[0,1]$ which is the closed interval $\br{x_0,(1-1/i,1)}_X$.
So, $\cpa{K_i}$ is an efficient compact exhaustion of $X$.
For each $i\in\Z_+$, $V_i=X-K_i$ is unbounded and connected.
The space $X$ is $\sigma$-compact, locally compact, and a one-ended generalized continuum.
Suppose, by way of contradiction, that there exists a proper map $r:\nnr\to X$.
As $r$ is a proper map and $\nnr$ is connected, $\im{r}$ contains the interval $[r(0),\infty)_X$.
The interval $[r(0),\infty)_X$ contains uncountably many pairwise disjoint, nonempty open intervals $((x,0),(x,1))_X=\cpa{x}\times (0,1)$.
Taking inverse images under $r$, $\nnr\subseteq\R$ contains uncountably many pairwise disjoint, nonempty open sets, a contradiction.
Therefore, $X$ admits no proper map $r:\nnr\to X$.
A similar argument shows that $X$ is not locally path-connected---consider basic open neighborhoods of $p=\pa{1/2,0}$ in $X$.
As $X$ contains uncountably many pairwise disjoint, nonempty open intervals, $X$ does not have a countable basis
(that is, $X$ is not second countable).
As $X$ is a generalized continuum, $X$ is Lindel{\"o}f by Remarks~\ref{gcremarks}\ref{gcnice}.
Each Lindel{\"o}f and metrizable space has a countable basis~\cite[p.~194]{munkres}.
Hence, $X$ is not metrizable.
\end{example}

That example shows that an additional hypothesis on noncompact generalized continua is required to ensure the existence of a baseray.
Metrizability suffices as we now show.
Note that metrizable generalized continua include the four collections of spaces listed above Remarks~\ref{gcremarks}.
Those manifolds, simplicial complexes, CW complexes, and ANRs are the spaces that most interest us.
In that sense, the metrizability hypothesis is not overly restrictive.
Recall that an \deffont{arc} in a space is an embedding of the closed unit interval $I=[0,1]\subseteq\R$.

\begin{theorem}[Arcwise connectedness theorem]
Let $X$ be a connected, locally connected, locally compact metric space.
If $a$ and $b$ are distinct points of $X$, then there exists an arc in $X$ from $a$ to $b$.
If, furthermore, $A$ is a connected, open subspace of $X$ and $a$ and $b$ are distinct points of $A$,
then there exists an arc in $A$ from $a$ to $b$.
\end{theorem}

\begin{proof}
The first conclusion is proved in Schurle~\cite[Thm.~4.2.5]{schurle}.
The second conclusion follows from the first since $A$ is necessarily a
connected, locally connected, locally compact metric space
(for local compactness, see Munkres~\cite[Cor.~29.3]{munkres}).
\end{proof}

\begin{remark}\label{mgcniceH0}
The arcwise connectedness theorem implies that if $X$ is a metrizable generalized continuum and $V$ is an open subspace of $X$,
then the components and path-components of $V$ coincide.
Indeed, the arcwise connectedness theorem implies that $X$ is locally path-connected.
As $V$ is open in $X$, $V$ is locally path-connected.
The result follows by Munkres~\cite[Thm.~25.5]{munkres}.
\end{remark}

We pause to introduce some useful terminology.
Let $X$ be a generalized continuum and let $r:\nnr\to X$ be a proper map (not necessarily an embedding).
We say that $r$ \deffont{points to $\e \in E(X)$} provided the induced map $E(r):E(\nnr) \to E(X)$ on end spaces
sends the end $\infty$ of $\nnr$ to the end $\e$.
See the proof of Lemma~\ref{inducedmaps} above for the definition of the induced map on end spaces.

\begin{theorem}[Existence of baseray]\label{raysexist}
Let $X$ be a noncompact, metrizable generalized continuum.
Then, $X$ has at least one end.
Furthermore, if $\e$ is any end of $X$,
then there exists a proper embedding $r:\nnr \rightarrowtail X$ such that $r$ points to $\e$.
\end{theorem}

\begin{proof}
By Theorem~\ref{gceee}, $X$ admits an efficient exhaustion by compacta $\cpa{K_i}$.
For notational convenience, in this proof we index by $i\in\Z_{\geq0}$.
For each $i\in\Z_{\geq0}$, let $V_i=X-K_i$ and write $\cpa{V_i^j}=\uc{V_i}$ for the set of unbounded components of $V_i$.
By Theorem~\ref{endsthm}\ref{noncompactexistsend}, $X$ has at least one end.
Let $\e=\pa{V_0^{j_0},V_1^{j_1},\ldots}\in E\pa{X;\cpa{K_i}}$ be an end of $X$.
By definition, we have $V_0^{j_0}\supseteq V_1^{j_1} \supseteq\cdots$.
Recall that for each $i\in\Z_{\geq0}$, $V_i^{j_i}$ is connected (hence, nonempty) and open in $X$ by Lemma~\ref{compslemma}\ref{compsopen}.
As we pass to subsequences of $\cpa{K_i}$, we analogously pass to subsequences of $V_i$ and $\e$.

Let $p_0\in V_0^{j_0}$.
By compactness, $p_0\in\interior{K_m}$ for some $m>0$.
Pass to a subsequence of $\cpa{K_i}$ while retaining $K_0$ so that $p_0\in\interior{K_1}$.
Let $p_1\in V_1^{j_1}$.
As $p_0$ and $p_1$ lie in the connected, open subspace $V_0^{j_0}$ of $X$,
the arcwise connectedness theorem yields an arc $\alpha_0:[0,1]\to V_0^{j_0}$ from $p_0$ to $p_1$.
By compactness, $\im{\alpha_{0}} \subseteq \interior{K_m}$ for some $m>1$.
Pass to a subsequence of $\cpa{K_i}$ while retaining $K_0$ and $K_1$ so that $\im{\alpha_{0}} \subseteq \interior{K_2}$.
Let $p_2\in V_2^{j_2}$.
As $p_1$ and $p_2$ lie in the connected, open subspace $V_1^{j_1}$ of $X$,
the arcwise connectedness theorem yields an arc $\alpha_1:[1,2]\to V_1^{j_1}$ from $p_1$ to $p_2$ as in Figure~\ref{fig:buildray} (left).
\begin{figure}[htbp!]
    \centerline{\includegraphics[scale=1.0]{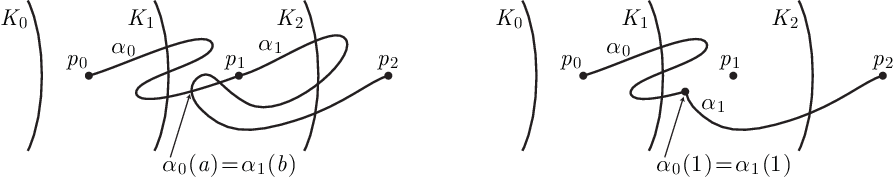}}
    \caption{Arcs $\alpha_0$ and $\alpha_1$ (left) and trimmed and reparameterized arcs $\alpha_0$ and $\alpha_1$ (right).}
\label{fig:buildray}
\end{figure}

The images of the arcs $\alpha_0$ and $\alpha_1$ meet at $p_1$, though they may meet at other points.
By compactness, there is a first point of $\im{\alpha_0}$ that meets $\im{\alpha_1}$.
Let that point be $\alpha_0(a)=\alpha_1(b)$ where $0<a\leq1$ and $1\leq b <2$.
We trim and reparameterize the arcs $\alpha_0$ and $\alpha_1$ so that they concatenate to
an arc $\br{0,2}\to V_0^{j_0}$ from $p_0$ to $p_2$.
The resulting concatenated arc may not pass through $p_1$, but that is irrelevant.
Trim the domain of $\alpha_0$ to be $\br{0,a}$ and reparameterize the result to have domain $\br{0,1}$;
let $\alpha_0$ denote the resulting arc.
Trim the domain of $\alpha_1$ to be $\br{b,2}$ and reparameterize the result to have domain $\br{1,2}$;
let $\alpha_1$ denote the resulting arc.
The images of $\alpha_0$ and $\alpha_1$ meet only at $\alpha_0(1)=\alpha_1(1)$ as in Figure~\ref{fig:buildray} (right).
By the pasting lemma (Munkres~\cite[Thm.~18.3]{munkres}), $\alpha_0$ and $\alpha_1$ concatenate to an arc.
Note that $\im{\alpha_0}\subseteq V_0^{j_0}$,
$\im{\alpha_{0}} \subseteq \interior{K_2}$,
and $\im{\alpha_1}\subseteq V_1^{j_1}$.

By compactness, $\im{\alpha_{1}} \subseteq \interior{K_m}$ for some $m>2$.
Pass to a subsequence of $\cpa{K_i}$ while retaining $K_0$, $K_1$, and $K_2$ so that $\im{\alpha_{1}} \subseteq \interior{K_3}$.
Let $p_3\in V_3^{j_3}$.
As $p_2$ and $p_3$ lie in the connected, open subspace $V_2^{j_2}$ of $X$,
the arcwise connectedness theorem yields an arc $\alpha_2:[2,3]\to V_2^{j_2}$ from $p_2$ to $p_3$.
As before, we trim and reparameterize $\alpha_1$ and $\alpha_2$---without altering $\alpha_0$---so that
$\alpha_1$ and $\alpha_2$ concatenate to an arc $\br{1,3}\to V_1^{j_1}$ from $\alpha_1(1)$ to $p_3$.
Note that $\im{\alpha_1}\subseteq V_1^{j_1}$,
$\im{\alpha_{1}} \subseteq \interior{K_3}$,
and $\im{\alpha_2}\subseteq V_2^{j_2}$.

Repeat that process inductively to obtain arcs $\alpha_i:\br{i,i+1}\to V_i^{j_i}$
that concatenate to an injective map $r:\nnr \rightarrowtail X$.
The map $r$ is proper and points to $\e$ since $\im{\alpha_i}\subseteq V_i^{j_i}$ for each $i\geq0$.
By Remarks~\ref{rece}\ref{injemd}, $r$ is an embedding as desired.
\end{proof}

In the proof of Theorem~\ref{raysexist}, no attempt was made to prevent the arc $\alpha_0$ from entering some other component $V_1^k$
before heading into $V_1^{j_1}$ to terminate at $p_1$ as in Figure~\ref{fig:rayprevent}.
\begin{figure}[htbp!]
    \centerline{\includegraphics[scale=1.0]{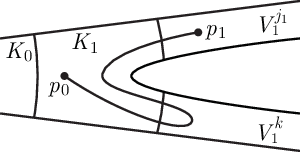}}
    \caption{Arc $\alpha_0$ entering a component $V_1^k$ before heading into $V_1^{j_1}$ to terminate at $p_1$.}
\label{fig:rayprevent}
\end{figure}
Without some additional hypothesis on $X$, it is not even possible to prevent that as shown by the following example.

\begin{example}\label{topsine}
Let $X\subseteq\R^3$ be the metrizable generalized continuum shown in Figure~\ref{fig:gcpath}.
\begin{figure}[htbp!]
    \centerline{\includegraphics[scale=1.0]{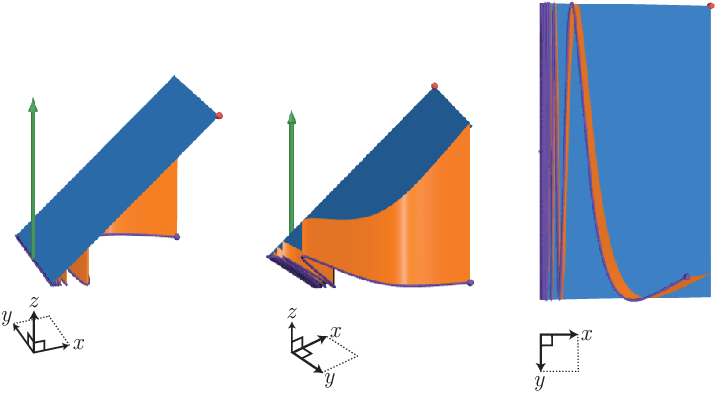}}
    \caption{Three views of a metrizable generalized continuum $X$ in $\R^3$ with positive coordinate axes indicated.
		The origin of $\R^3$ is the initial point of the green ray.}
\label{fig:gcpath}
\end{figure}
We define $X$ to be the union of the following three subsets of $\R^3$.
First, $A$ is the blue rectangle minus the red corner point as in Figure~\ref{fig:gcpath}.
That is,
\[
A = \cpa{(x,y,x) \mid 0\leq x\leq 1 \ \& \ -1\leq y\leq 1} - \cpa{(1,-1,1)}
\]
Second, $B$ is the orange surface above a portion of the topologist's sine curve and below $A$.
That is,
\[
B=\cpa{(x,\sin(1/x),z) \mid 0<x\leq1 \ \& \ 0\leq z\leq x}
\]
Third, $C$ is the green ray equal to the nonnegative $z$-axis.
We define $X=A\cup B\cup C$, which is a metric subspace of $\R^3$.
The reader may verify that $X$ is a generalized continuum with two ends (one where the red point was removed and one pointed to by the green ray).
Let $K$ be the intersection of $X$ with the plane $z=0$, which is indicated in purple in Figure~\ref{fig:gcpath}.
That is, $K$ is the union of the line segment $\cpa{(0,y,0) \mid -1\leq y\leq 1}$ and the portion of the topologist's sine curve
$\cpa{(x,\sin(1/x),0) \mid 0<x\leq 1}$.
So, $K$ is compact, connected, and efficient, but not path connected.
Let $p=(1,\sin(1),0)\in K$ which is the purple point indicated in Figure~\ref{fig:gcpath}.
The complement of $K$ in $X$ has two unbounded components $V^1$ and $V^2$, where $V^1$ is the positive $z$-axis.
Notice that every path in $X$ from $p$ to any point in $V^1$ must pass through $V^2$.
\end{example}

Just as it is useful to retract a based space to its basepoint,
it is useful to properly retract a ray-based space to its baseray.
Before proving such retracts exist, we make an observation on frontiers.

\begin{lemma}\label{frontcont}
Let $X$ be a generalized continuum and let $A\subseteq X$ be a closed subspace.
Let $K\subseteq X$ be a compactum and define $L=A\cap K$ as in Figure~\ref{fig:frontiers}.
\begin{figure}[htbp!]
    \centerline{\includegraphics[scale=1.0]{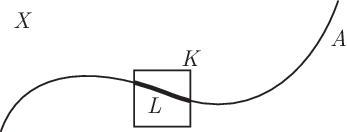}}
    \caption{Closed subspace $A$ of $X$, compactum $K$ in $X$, and intersection $L$ of $A$ and $K$.}
\label{fig:frontiers}
\end{figure}
Then, the frontier of $L$ in $A$ is contained in the frontier of $K$ in $X$.
That is, $\Fr{L}{A} \subseteq \Fr{K}{X}$.
\end{lemma}

\begin{proof}
The hypotheses imply that $L$ is compact.
Thus, $\Fr{L}{A} = L-\Int{L}{A}$ and $\Fr{K}{X} = K-\Int{K}{X}$.
Let $p\in \Fr{L}{A}$, so $p\in L$ and $p\notin \Int{L}{A}$.
As $L\subseteq K$, we have $p\in K$.
Suppose, by way of contradiction, that $p\in \Int{K}{X}$.
Then, there exists a neighborhood $U$ of $p$ that is open in $X$ and contained in $K$.
Hence, $p\in A\cap U \subseteq A\cap K=L$ where $A\cap U$ is open in $A$.
That implies $p\in\Int{L}{A}$, a contradiction completing the proof.
\end{proof}

\begin{remark}
In the setup of Lemma~\ref{frontcont}, $\Int{L}{A}$ need not be contained in $\Int{K}{X}$.
\begin{figure}[htbp!]
    \centerline{\includegraphics[scale=1.0]{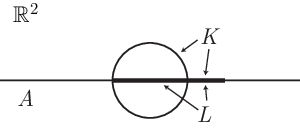}}
    \caption{The $x$-axis $A$ in $\R^2$, compactum $K$ equal to a disk union a segment, and intersection $L$ of $A$ and $K$.}
\label{fig:intce}
\end{figure}
Consider the example in Figure~\ref{fig:intce} where $X=\R^2$, $A$ is the $x$-axis, and
$K$ is a closed disk union the indicated line segment.
\end{remark}

We use Lemma~\ref{frontcont} in the following key setup.
Let $(X,r)$ be a ray-based generalized continuum, and let $\cpa{K_i}$ be an $r$-efficient compact exhaustion of $X$.
For each $i\in\Z_+$, define $L_i=\im{r} \cap K_i$.
So, we have efficient compact exhaustions $\cpa{L_i}$ of $\im{r}$ and $\cpa{r^{-1}\pa{K_i}}$ of $\nnr$.
By definition, each efficient compactum is connected and, hence, nonempty.
In particular, $r(0)\in K_1$.
If necessary, then replace $\cpa{K_i}$ with the subsequence obtained by deleting the first term
so that $r(0)\in\interior{K_1}$.
Recalling Examples~\ref{eceexamples}\ref{ecennr}, there exists an unbounded sequence of real numbers
$0< b_1<b_2<b_3<\cdots$ such that $r^{-1}\pa{K_i}=\br{0,b_i}$ for each $i\in\Z_+$ ($b_1>0$ since we arranged that $r(0)\in\interior{K_1}$).
\begin{figure}[htbp!]
    \centerline{\includegraphics[scale=1.0]{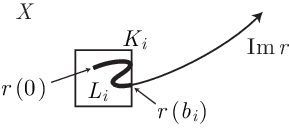}}
    \caption{Efficient compactum $K_i\subseteq X$ meeting the ray $\im{r}$ in the arc $L_i=r\pa{\br{0,b_i}}$.}
\label{fig:frontrays}
\end{figure}
So, $L_i$ is the arc $r\pa{\br{0,b_i}}$ with endpoints $r(0)$ and $r\pa{b_i}$ as depicted in Figure~\ref{fig:frontrays}.
As $\Fr{L_i}{\im{r}}=\cpa{r\pa{b_i}}$, Lemma~\ref{frontcont} implies that $r\pa{b_i} \in \Fr{K_i}{X}$.
For each $i\in\Z_+$, define $a_i=\inf r^{-1}\pa{\Fr{K_i}{X}}$.
As $K_i=\Fr{K_i}{X} \sqcup \Int{K_i}{X}$, we see that $a_i\leq b_i$
(both $a_i=b_i$ and $a_i<b_i$ occur in simple examples).
As $0\in r^{-1}\pa{\interior{K_1}}$, we have $0<a_1$.
As $r^{-1}\pa{K_i}=\br{0,b_i}$ and $K_i\subseteq \Int{K_{i+1}}{X}$ is disjoint from $\Fr{K_{i+1}}{X}$,
we see that $r^{-1}\pa{\Fr{K_{i+1}}{X}}$ is a nonempty, compact subset of the interval $\pa{b_i,\infty}$.
Hence, $b_i<a_{i+1}$.
Thus, $0<a_1 \leq b_1 < a_2 \leq b_2 < \cdots$
and $r^{-1}\pa{\Fr{K_i}{X}} \subseteq \br{a_i,b_i}$ for each $i\in\Z_+$.

\begin{theorem}[Existence of retract to baseray]\label{retract}
Let $(X,r)$ be a ray-based generalized continuum.
That is, $X$ is a generalized continuum and $r:\nnr \rightarrowtail X$ is a proper embedding.
Then, there exists a proper map $\rho:X\to\nnr$ such that $\rho\circ r=\tn{id}:\nnr\to\nnr$.
In particular, $r\circ\rho:X\to\im{r}$ is a proper retraction of $X$ onto the baseray.
\end{theorem}

\begin{proof}
By Corollary~\ref{receexist}, there exists an $r$-efficient compact exhaustion $\cpa{K_i}$ of $X$.
We assume the setup and notation from the paragraph preceding Theorem~\ref{retract}.

We will repeatedly apply the pasting lemma and the Tietze extension theorem for continuous functions defined on closed subspaces
(see Munkres~\cite[Thms.~18.3 \& 35.1(a)]{munkres}).
The Tietze extension theorem applies since, as $X$ is a generalized continuum, $X$ is normal by Remarks~\ref{gcremarks}\ref{gcnice}.
Note that each closed subspace of a normal space is normal.
And, for each $i\in\Z_+$, we have $K_i\subseteq\Int{K_{i+1}}{X}$ and $K_i=\Fr{K_i}{X} \sqcup \Int{K_i}{X}$.

First, we define a map $\rho_1: K_1 \to \br{0,b_1}$.
Define the restriction homeomorphism $f_1=\rest{r}^{-1}:r\pa{\br{0,b_1}}\to\br{0,b_1}$.
Define the restriction map $g_1=\rest{f_1}:r\pa{\br{0,b_1}} \cap \Fr{K_1}{X} \to \br{a_1,b_1}$.
Tietze extend $g_1$ to a map $h_1: \Fr{K_1}{X} \to \br{a_1,b_1}$.
Paste $f_1$ and $h_1$ to the map $k_1:r\pa{\br{0,b_1}} \cup \Fr{K_1}{X} \to \br{0,b_1}$.
Tietze extend $k_1$ to a map $\rho_1: K_1 \to \br{0,b_1}$.

Second, we extend $\rho_1$ to a map $\rho_2: K_2 \to \br{0,b_2}$.
Define the restriction homeomorphism $f_2=\rest{r}^{-1}:r\pa{\br{b_1,b_2}}\to\br{b_1,b_2}$.
Define the restriction map $g_2=\rest{f_2}:r\pa{\br{b_1,b_2}} \cap \Fr{K_2}{X} \to \br{a_2,b_2}$.
Tietze extend $g_2$ to a map $h_2: \Fr{K_2}{X} \to \br{a_2,b_2}$.
Paste $h_1$, $f_2$, and $h_2$ to the map $k_2:\Fr{K_1}{X} \cup r\pa{\br{b_1,b_2}} \cup \Fr{K_2}{X} \to \br{a_1,b_2}$.
Tietze extend $k_2$ to a map $l_2: K_2 - \Int{K_1}{X} \to \br{a_1,b_2}$.
Paste $\rho_1$ and $l_2$ to the map $\rho_2: K_2 \to \br{0,b_2}$.

Repeat that second step inductively to obtain $\rho_i:K_i\to \br{0,b_i}$ for each $i\geq3$.
Define $\rho:X\to \nnr$ to equal $\rho_i$ on $K_i$.
By construction, $\rho$ is continuous and restricts to $r^{-1}$ on the ray $\im{r}$.
The latter implies that $\rho\circ r=\tn{id}:\nnr\to\nnr$.
Also by construction, the following hold for each $i\in\Z_+$:
$\rho\pa{K_i}=\br{0,b_i}$, $\rho\pa{\Fr{K_i}{X}} \subseteq \br{a_i,b_i}$,
and $\rho\pa{K_{i+1}-\Int{K_i}{X}} \subseteq \br{a_i,b_{i+1}}$.
The latter implies that $\rho\pa{X-K_i}\subseteq [a_i,\infty)$ for each $i\in\Z_+$
(since $\cpa{K_i}$ exhausts $X$ and $\cpa{a_i}$ is increasing).
In particular, $\rho^{-1}\pa{\br{0,a_i}} \subseteq K_{i+1}$ for each $i\in\Z_+$.
That implies $\rho$ is proper as follows.
Let $K\subseteq \nnr$ be compact.
As $\rho$ is continuous, $\rho^{-1}\pa{K}$ is closed in $X$.
As $K\subseteq \nnr$ is compact and $\cpa{a_i}$ is increasing and unbounded,
there exists $i\in\Z_+$ such that $K\subseteq \br{0,a_i}$.
So, $\rho^{-1}\pa{K} \subseteq K_{i+1}$.
As $\rho^{-1}\pa{K}$ is closed in $X$ and contained in the compactum $K_{i+1}$, we see that $\rho^{-1}\pa{K}$ is compact, as desired.

Lastly, $r\circ\rho:X\to\im{r}$ is a composition of proper maps and, hence, is a proper map.
We must show that $r\circ\rho$ equals the identity on $\im{r}$.
Let $r(a)\in\im{r}$.
Then, $r\circ\rho \pa{r(a)}=r\pa{\rho\circ r(a)}=r(a)$ since $\rho\circ r=\tn{id}:\nnr\to\nnr$.
Therefore, $r\circ\rho:X\to\im{r}$ is a proper retraction of $X$ onto the baseray.
\end{proof}

We close this section with observations on Alexandroff's and Freudenthal's compactifications of metrizable generalized continua.
Given a generalized continuum $X$, we let $A(X)=X\sqcup\cpa{\infty}$ denote Alexandroff's one-point compactification of $X$
(see Dugundji~\cite[p.~246]{dugundji} or Munkres~\cite[Thm.~29.1]{munkres}).
Recall that $A(X)$ is equipped with the topology consisting of sets that are open in $X$
or have the form $A(X)-K$ for some compactum $K\subseteq X$.

\begin{lemma}\label{Anbi}
Let $X$ be a generalized continuum, and let $\cpa{K_i}$ be an efficient compact exhaustion of $X$.
Then, $A(X)-K_i$ for $i\in\Z_+$ is a connected, countable neighborhood basis of $\infty$ in $A(X)$.
In particular, $A(X)$ is connected if and only if $X$ is noncompact.
\end{lemma}

\begin{proof}
First, suppose $X$ is compact.
Then, efficiency implies that $K_i=X$ for each $i\in\Z_+$.
In particular, $\cpa{\infty}$ is open in $A(X)$ and  
$A(X)$ is homeomorphic to the topological disjoint union of $X$ and the isolated point $\infty$.

Next, suppose $X$ is noncompact.
For each $i\in\Z_+$, let $V_i=X-K_i$ with (unbounded) components $V_i^1,\ldots,V_i^{j_i}$.
We claim that $\infty \in \Cl{V_i^j}{A(X)}$ for each $V_i^j$.
Consider $A(X)-K$ a neighborhood of $\infty$ in $A(X)$ where $K\subseteq X$ is compact.
As $V_i^j$ is unbounded, $V_i^j$ is not contained in $K$.
Thus, $V_i^j$ meets $X-K_i$ and meets $A(X)-K_i$.
That proves the claim.
As $V_i^j$ is connected and $\infty \in \Cl{V_i^j}{A(X)}$, it follows that $V_i^j\sqcup \cpa{\infty}$ is connected.
Thus, $\cup_{j=1}^{j_i} V_i^j\sqcup \cpa{\infty}$ is connected.
As $K_i$ is efficient, $V_i$ has no bounded component and that union equals $A(X)-K_i$.
So, each $A(X)-K_i$ is connected.
Let $A(X)-K$ be a neighborhood of $\infty$ in $A(X)$ where $K\subseteq X$ is compact.
There exists $i\in\Z_+$ such that $K\subseteq K_i$.
Hence, $A(X)-K_i \subseteq A(X)-K$, as desired.
\end{proof}

\begin{lemma}\label{Xcountablebasis}
If $X$ is a metrizable generalized continuum, then $X$ has a countable dense set and a countable basis.
\end{lemma}

\begin{proof}
As $X$ is a generalized continuum, $X$ is Lindel{\"o}f by Remarks~\ref{gcremarks}\ref{gcnice}.
As $X$ is a metric space, the following are equivalent: $X$ is Lindel{\"o}f, $X$ has a countable dense set,
and $X$ has a countable basis---see Dugundji~\cite[p.~187]{dugundji} or Engelking~\cite[p.~256]{engelking}.
\end{proof}

\begin{lemma}\label{AFcountablebasis}
If $X$ is a generalized continuum with a countable basis, then both $A(X)$ and $F(X)$ have a countable basis.
\end{lemma}

\begin{proof}
By Theorem~\ref{gceee}, $X$ has an efficient compact exhaustion $\cpa{K_i}$ where $i\in\Z_+$.
For $A(X)$, use the countable basis of $X$ together with the countable neighborhood basis of $\infty$ from Lemma~\ref{Anbi}
given by $A(X)-K_i$ where $i\in\Z_+$.
For $F(X)$, use the countable basis of $X$ together with the countable collection of all basic open sets $FV_k^l$ defined in~\eqref{basisF}.
\end{proof}

\begin{lemma}\label{AFmetrizable}
If $X$ is a metrizable generalized continuum, then both $A(X)$ and $F(X)$ are metrizable.
\end{lemma}

\begin{proof}
By Lemma~\ref{Xcountablebasis}, $X$ has a countable basis.
By Lemma~\ref{AFcountablebasis}, both $A(X)$ and $F(X)$ have a countable basis.
As both $A(X)$ and $F(X)$ are compact and Hausdorff, $A(X)$ and $F(X)$ are regular (see also Munkres~\cite[Thm.~32.3]{munkres}).
By Urysohn's metrization theorem (see Munkres~\cite[p.~215]{munkres}), $A(X)$ and $F(X)$ are metrizable.
\end{proof}

\begin{theorem}\label{AFnice}
Let $X$ be a metrizable generalized continuum.
Then, $F(X)$ is a compact, metrizable generalized continuum.
If further $X$ is noncompact, then $A(X)$ is a compact, metrizable generalized continuum.
\end{theorem}

\begin{proof}
By Theorem~\ref{endcompthm}, $F(X)$ is a compact generalized continuum.
By Lemma~\ref{AFmetrizable}, $F(X)$ is metrizable.
Whether or not $X$ is noncompact, $A(X)$ is compact (hence $\sigma$-compact and locally compact), Hausdorff,
locally connected, and metrizable (by Lemma~\ref{AFmetrizable}).
For local connectedness of $A(X)$, use local connectedness of $X$ and Lemma~\ref{Anbi}
(see also de Groot and McDowell~\cite[Thm.~4.1]{dgmc}).
Finally, $A(X)$ is connected if and only if $X$ is noncompact by Lemma~\ref{Anbi}.
\end{proof}

\begin{remark}\label{shortproofraysexist}
We may now give a slick proof of the existence of a baseray in any noncompact, metrizable generalized continuum $X$:
apply the arcwise connectedness theorem to obtain an arc in $A(X)$ from $p\in X$ to $\infty$,
and then delete the endpoint $\infty$ to obtain a baseray.
While that proof is conceptually simple, it does not allow one to ensure that the resulting baseray points to a specified end of $X$
as was obtained in Theorem~\ref{raysexist}.
One may apply the same argument using $F(X)$ in place of $A(X)$.
Doing so, one may ensure that the resulting baseray points to a specified \textit{isolated} end of $X$.
Recall that an end $\e$ of $X$ is \deffont{isolated} provided $\e$ is an isolated point of $E(X)$. 
\end{remark}

\begin{remark}
We believe that the main results of this section---existence of a baseray and a proper retract to each baseray---extend to trees.
See Section~\ref{sec:fd} for more precise statements.
\end{remark}

\section{Reduced cohomology}
\label{sec:reducedcohomology}

In this section, we compare three standard definitions of reduced cohomology.
For clarity, we use integer coefficients.
Let $X$ be a nonempty space with finitely many path-components $X_1,\ldots,X_n$.
Let $\cpa{\bullet}$ denote a one-point space.
For the first definition, consider the augmented singular chain complex of $X$.
Apply the $\hom{}{-}{\Z}$ functor to that complex to obtain the dual augmented cochain complex of $X$.
Define reduced cohomology $\rzco{\ast}{X}$ of $X$ to be kernel mod image of that dual complex.
Equivalently, there is a unique map $q:X\to \cpa{\bullet}$.
That map induces a ring homomorphism $q^{\ast}:\zco{\ast}{\bullet}\to \zco{\ast}{X}$ and a short exact sequence
\begin{equation}\label{cokernelses}
0 \leftarrow \cokernel{q^{\ast}} \leftarrow \zco{\ast}{X} \xleftarrow{q^{\ast}} \zco{\ast}{\bullet} \leftarrow 0
\end{equation}
Reduced cohomology of $X$ is the cokernel of $q^{\ast}$.
For the second definition, choose a basepoint $\bullet\in X$.
Equivalently, choose a map $i:\cpa{\bullet}\to X$.
Define reduced cohomology of $X$ to be the relative cohomology of the pair $\pa{X,\bullet}$, namely $\rzco{\ast}{X}=\zco{\ast}{X,\bullet}$.
Equivalently, the map $i:\cpa{\bullet}\to X$ induces a ring homomorphism $i^{\ast}:\zco{\ast}{X}\to \zco{\ast}{\bullet}$ and a short exact sequence
\begin{equation}\label{kernelses}
0 \leftarrow \zco{\ast}{\bullet} \xleftarrow{i^{\ast}} \zco{\ast}{X} \leftarrow \zco{\ast}{X,\bullet} \leftarrow 0
\end{equation}
Reduced cohomology of $X$ is the kernel of $i^{\ast}$, which is an ideal of $\zco{\ast}{X}$.

For any map $i:\cpa{\bullet}\to X$, the composition $\cpa{\bullet} \xrightarrow{i} X \xrightarrow{q} \cpa{\bullet}$ is the identity
and both sequences~\eqref{cokernelses} and~\eqref{kernelses} split.
So, both approaches yield splittings
\begin{equation}\label{reducedsplitting2}
\zco{\ast}{X} \cong \rzco{\ast}{X} \eds \zco{\ast}{\bullet}
\end{equation}
and isomorphic $\Z$-modules $\rzco{k}{X}$ for each dimension $k$.
Recall that $\zho{0}{X}$ is canonically isomorphic to the free abelian group $\fg{X_1,\ldots,X_n}\cong\Z^n$
with basis the set $\cpa{X_1,\ldots,X_n}$ of path-components of $X$.
By the universal coefficient theorem, $\zco{0}{X}\cong \hom{}{\zho{0}{X}}{\Z}$ is canonically isomorphic
to the free abelian group $\fg{\delta X_1,\ldots,\delta X_n}\cong \Z^n$
where $\delta X_i$ denotes the dual generator of $X_i$ that sends $X_i\mapsto1$ and $X_j\mapsto0$ for all $j\neq i$.
Thus, $q^{\ast}:\zco{\ast}{\bullet}\to \zco{\ast}{X}$ is the function $\zco{0}{\bullet}\to \zco{0}{X}$
sending $\delta\bullet \mapsto \delta X_1+\cdots+\delta X_n$.
Using the canonical isomorphisms, $q^{\ast}$ is the diagonal function
$\Z\cong\zco{0}{\bullet}\to \zco{0}{X}\cong \Z^n$ given by $1\mapsto \pa{1,\ldots,1}$.
So, $\im{q^{\ast}} = \Delta \cong \Z$ is the diagonal subgroup of $\zco{0}{X}$ generated
by the element $\delta X_1+\cdots+\delta X_n$.
Therefore, the first definition gives $\rzco{0}{X} \cong \zco{0}{X}/\Delta$, where
$\Delta$ is an ideal of $\zco{\ast}{X}$ if and only if $X$ is acyclic.
And, the second definition gives $\rzco{0}{X} = \fg{\delta X_1,\ldots,\delta X_{k-1},\delta X_{k+1},\ldots, \delta X_n}$
where the basepoint $\bullet$ lies in $X_k$.

If $X$ is path-connected, then the two definitions yield canonically isomorphic, graded, associative $\Z$-algebras $\rzco{\ast}{X}$ with $\rzco{0}{X}=\cpa{0}$.
Those algebras are unital precisely when $X$ is acyclic and $\rzco{\ast}{X}=\cpa{0}$ is the zero ring.
The fact that the two definitions agree when $X$ is path-connected may be the reason the standard literature does not emphasize
that they differ, as we now explain.

Suppose $X$ has more than one path-component.
Then, the first definition yields the sequence~\eqref{cokernelses} which does not split naturally in $X$
and multiplication in $\rzco{0}{X}$ is not well-defined.
For a simple example, consider the zero-sphere $S^0=\cpa{\pm1}$.
Notice that $\rzco{\ast}{S^0}=\rzco{0}{S^0}\cong \Z^2 /\Delta$
where the diagonal subgroup $\Delta$ is not an ideal of $\Z^2$ and
coordinatewise multiplication is not well-defined on $\Z^2 /\Delta$.
The sequence~\eqref{cokernelses} becomes
\begin{equation}\label{seszs}
\begin{tikzcd}
0	& \arrow[l] \Z^2/\Delta	& \arrow[l] \Z^2	\arrow[r,bend left=45,dashed,"s"]	& \arrow[l,"q^{\ast}"] \Z	& \arrow[l] 0
\end{tikzcd}
\end{equation}
where $q^{\ast}$ is the diagonal function.
Each splitting of~\eqref{seszs} is given by a matrix $s=\begin{bmatrix} a & 1-a \end{bmatrix}$ for some $a\in\Z$.
The only splitting compatible with the map $f:\cpa{\bullet}\to S^0$ defined by $f(\bullet)=1$ is $\begin{bmatrix} 1 & 0 \end{bmatrix}$.
The only splitting compatible with the map $g:\cpa{\bullet}\to S^0$ defined by $g(\bullet)=-1$ is $\begin{bmatrix} 0 & 1 \end{bmatrix}$.
So, no splitting of~\eqref{seszs} is natural.
On the other hand, the sequence~\eqref{kernelses} always splits naturally in $X$ for based maps,
and all relative cup products are well-defined---see Hatcher~\cite[pp.~199--214]{hatcher} and May~\cite[pp.~143--148]{maycc}.
The unique natural splitting of~\eqref{kernelses} is given by $q^{\ast}:\zco{\ast}{\bullet}\to \zco{\ast}{X}$---uniqueness follows
by considering naturality for the retraction map $q:X\to\cpa{\bullet}$.

For the homotopical approach, let $\br{X,Y}$ denote the set of based homotopy classes of based maps from $X$ to $Y$.
For each integer $n\geq 0$, define $\rzco{n}{X}=\br{X,K(\Z,n)}$ where $K(\Z,n)$ is an Eilenberg-Mac Lane space---see May~\cite[Ch.~22]{maycc}.
Recall that $K(\Z,n)$ is connected for each $n>0$ and that $\Z$ is a $K(\Z,0)$.
So,
\[
\rzco{0}{X}=\br{X,\Z}\cong \fg{\delta X_1,\ldots,\delta X_{k-1},\delta X_{k+1},\ldots, \delta X_n}
\]
where the basepoint $\bullet$ lies in $X_k$.
Thus, the homotopical approach and the approach using a basepoint yield isomorphic results.

To summarize, defining reduced cohomology via augmentation has the apparent advantage of avoiding a choice of basepoint,
but then~\eqref{reducedsplitting2} does not split naturally and cup products involving dimension-zero classes are generally not defined.
The alternative definition requires an initial choice of basepoint.
In return, the resulting graded, associative $\Z$-algebra has all cup products well-defined and~\eqref{reducedsplitting2} splits naturally.
For those reasons, the latter approach is preferred\footnote{Axioms
for a reduced homology theory of based spaces first appeared in Dold and Thom~\cite{doldthom} who attribute them to Puppe.
For a sampling of textbooks that use the augmentation approach, see
Eilenberg and Steenrod~\cite[pp.~18--19 \& 190--191]{eilenbergsteenrod},
Spanier~\cite[pp.~168 \& 237]{spanier},
Dold~\cite[pp.~33--34]{dold},
Massey~\cite[pp.~40--43]{massey}, and
Bredon~\cite[pp.~172--185]{bredon}.
Textbooks that use the based approach or mention both approaches include
May~\cite[pp.~97, 105--109, \& 143--148]{maycc},
Hatcher~\cite[pp.~110--126, 160--161, \& 199--214]{hatcher}, and
tom Dieck~\cite[pp.~252 \& 407]{tomdieck}.}
and is the one we will use.

\section{End cohomology}
\label{sec:endcohomology}

In this section, we review the definition of end cohomology and its basic properties.
Then, we introduce our definition of reduced end cohomology and prove a splitting result.
Finally, we study dimension-zero end cohomology in greater detail.
Along the way, we give several illustrative examples.
Throughout this section the coefficient ring $R$ is an arbitrary PID,
except where integer coefficients are explicitly indicated.

We adopt the direct limit approach to end cohomology\footnote{For background on direct systems and direct limits,
see Eilenberg and Steenrod~\cite[Ch.~VIII]{eilenbergsteenrod} and Massey~\cite[A.1~\&~A.2]{massey}.
In particular, the direct limit of a direct system is the \emph{universal repelling object} in the sense that
any compatible collection of homomorphisms from the direct system to an object yields a unique homomorphism from the direct limit to that object.}.
Let $X$ be a generalized continuum.
By Theorem~\ref{gceee}, $X$ may be exhausted by compacta.
Define the poset $\pa{\mathcal{K},\leq}$ where $\mathcal{K}$ is the set of compacta in $X$, and $K\leq K'$ means \mbox{$K\subseteq K'$}.
We have a direct system of graded $R$-algebras $\zco{\ast}{X-K}$ where $K\in\mathcal{K}$ and
the morphisms are inclusion induced restrictions.
Define $\zeco{\ast}{X}$, the \deffont{end cohomology $R$-algebra} of $X$, to be the direct limit of that direct system.
For the relative version, let $(X,A)$ be a \deffont{closed pair}, namely a generalized continuum $X$ together with a closed subspace $A\subseteq X$.
Closed pairs are required to ensure properness of induced maps.
Regard $X$ as the closed pair $(X,\emptyset)$.
Consider the direct system $\zco{\ast}{X-K,A-K}$ where $K\in\mathcal{K}$ and the morphisms are inclusion induced restrictions.
Define $\zeco{\ast}{X,A}$, the \deffont{end cohomology $R$-algebra} of $(X,A)$, to be the direct limit of that direct system.

\sloppy We employ a standard explicit model of the direct limit---see~\cite[p.~222]{eilenbergsteenrod} or \hbox{\cite[p.~384]{massey}}---where
an element of $\zeco{\ast}{X,A}$ is represented by an element of $\zco{\ast}{X-K,A-K}$ for some compactum $K\subseteq X$.
Two representatives $\alpha \in \zco{\ast}{X-K,A-K}$ and $\alpha' \in \zco{\ast}{X-K',A-K'}$
are equivalent provided they have the same restriction in some \hbox{$\zco{\ast}{X-K'',A-K''}$},
where $K,K'\subseteq K''$ are compacta in $X$.

Let $\cpa{K_i}$ be an exhaustion of $X$ by compacta.
As $\cpa{K_i}$ is cofinal in $\mathcal{K}$,
we may compute $\zeco{\ast}{X,A}$ using the direct system indexed by $\Z_+$.
Namely, there is a canonical isomorphism (see~\cite[p.~224]{eilenbergsteenrod})
\begin{equation}\label{eq:direct_limit}
    \zeco{\ast}{X,A} \cong \dlim \zco{\ast}{X - K_i,A-K_i}
\end{equation}
Note that we may delete instead the topological interior $K_i^\circ$ of $K_i$ to obtain the canonical isomorphism (see~\cite[p.~467]{cgh})
\begin{equation}\label{eq:direct_limit_int}
    \zeco{\ast}{X,A} \cong \dlim \zco{\ast}{X - K_i^\circ,A-K_i^\circ}
\end{equation}

Let $(Y,B)$ be another closed pair.
Let $f:(Y,B) \to (X,A)$ be a proper map of closed pairs, meaning $f:Y\to X$ is a proper map and $f(B)\subseteq A$.
It follows that $\rest{f}:B\to A$ is proper and $f$ induces an $R$-algebra morphism $f^{\ast}_{\tn{e}}:\zeco{\ast}{X,A}\to \zeco{\ast}{Y,B}$.
If $\tn{id}:(X,A)\to(X,A)$, then $\tn{id}^{\ast}_{\tn{e}}=\tn{id}$.
If $g:(Z,C) \to (Y,B)$ is another proper map of closed pairs, then $(f\circ g)^{\ast}_{\tn{e}}=g^{\ast}_{\tn{e}} \circ f^{\ast}_{\tn{e}}$.
So, end cohomology is a contravariant functor from the category of proper maps of closed pairs to the category of $R$-algebras.

Let $(X,A)$ and $(Y,B)$ be closed pairs.
Let $f$ and $g$ be proper homotopy equivalent proper maps $(Y,B) \to (X,A)$.
Then, $f^{\ast}_{\tn{e}}=g^{\ast}_{\tn{e}}$.
It follows that if $(X,A)$ and $(Y,B)$ are proper homotopy equivalent, then $\zeco{\ast}{X,A}$ and $\zeco{\ast}{Y,B}$ are isomorphic.
So, the isomorphism type of the end cohomology $R$-algebra of $(X,A)$ is an invariant of the proper homotopy type of $(X,A)$.

End cohomology has a long exact sequence for each closed pair and closed triple, satisfies excision for each excisive triad,
and satisfies a Mayer-Vietoris theorem---see~\cite[$\S$~2.3]{cgh} for details on those facts and others stated above.
Those long exact sequences are natural\footnote{For naturality, let $f:\pa{Y,B}\to\pa{X,A}$ be a proper map of closed pairs.
Consider the biinfinite commutative diagram for $\pa{X,A}$ in~\cite[p.~469]{cgh}
and the analogous diagram for $\pa{Y,B}$ (see~\cite[p.~468]{cgh}).
Stack the former directly atop the latter.
Restrictions of $f$ induce vertical homomorphisms straight down from the former to the latter.
Altogether, the entire diagram commutes.
The direct limit of that diagram yields the desired commutative diagram with two rows---the long exact end cohomology sequences
for $\pa{X,A}$ and $\pa{Y,B}$ respectively---and induced homomorphisms from the $(X,A)$ sequence to the $(Y,B)$ sequence.}.
The end cohomology algebra of each topological disjoint union of finitely many closed pairs is canonically isomorphic to the product of the
end cohomology algebras of the closed pairs.

As we discussed in Section~\ref{sec:reducedcohomology}, 
reduced cohomology of a space $X$ is defined to be the relative cohomology of the based space $(X,\bullet)$
for some choice of basepoint $\bullet\in X$.
Analogously, we define reduced end cohomology as follows.
Let $(X,r)$ be a ray-based generalized continuum.
Recall that a \deffont{baseray} is a proper embedding $r:\nnr \rightarrowtail X$.
If $X$ is a noncompact, metrizable generalized continuum, then Theorem~\ref{raysexist} ensures the existence of a baseray $r$ in $X$.
We define the \deffont{reduced end cohomology $R$-algebra} of $(X,r)$ to be
\begin{equation}\label{eq:recdef}
\rzeco{\ast}{X}=\zeco{\ast}{X,r}
\end{equation}
For each integer $k$, we define the $k$-dimensional \deffont{reduced end cohomology $R$-module} of $(X,r)$ to be
\[
\rzeco{k}{X}=\zeco{k}{X,r}
\]
When the baseray is clear from context, we write $\rzeco{\ast}{X}$.
Otherwise, we write $\zeco{\ast}{X,r}$.
Given another ray-based generalized continuum $(Y,s)$ and a proper map $f:Y\to X$,
we say that $f$ is \deffont{ray-based} provided $r=f\circ s$.

\begin{example}\label{ecfirstexample}\mbox{ \\ }
Let $X$ be a compact generalized continuum.
Define \hbox{$S(X)=\nnr\times X$} to be the \deffont{stringer based on} $X$.
Notice that $S(X)$ is a one-ended generalized continuum with efficient compact exhaustion $\br{0,i}\times X$, for $i\in\Z_+$,
by Lemmas~\ref{AxBgc}, \ref{gcprodIlemma}, and~\ref{endsYxI}.
If $X$ is based at $p\in X$, then the stringer $S(X)$ is equipped with the baseray $r=\nnr\times\cpa{p}$
as depicted in Figure~\ref{fig:sfunctor}.
\begin{figure}[htbp!]
    \centerline{\includegraphics[scale=1.0]{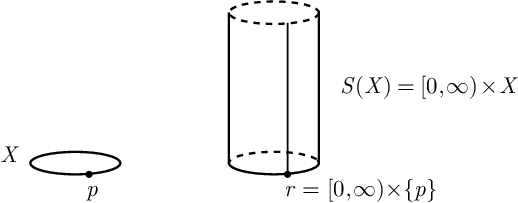}}
    \caption{Based space $X$ and ray-based stringer $S(X)=\nnr\times X$.}
\label{fig:sfunctor}
\end{figure}

There are canonical isomorphisms
\begin{align*}
\zeco{\ast}{S(X)}	&\cong	\zco{\ast}{X}\\
\rzeco{\ast}{S(X)} &= \zeco{\ast}{S(X),r} \cong \zco{\ast}{X,p} =\rzco{\ast}{X}
\end{align*}
In particular, for a ray we have
\begin{align*}
\zeco{\ast}{\nnr}	&=\zeco{0}{\nnr} \cong \zco{0}{\cpa{0}} \cong R\\
\rzeco{\ast}{\nnr} &=\cpa{0}
\end{align*}
If $f:(Y,q)\to (X,p)$ is a based map of compact generalized continua,
then we define $S(f):S(Y)\to S(X)$ by $(t,y)\mapsto (t,f(y))$.
It is straightforward to verify that $S(f)$ is a proper, ray-based map.
And, $S$ is a covariant functor from the category of based, compact generalized continua and based maps
to the category of ray-based, one-ended generalized continua and proper, ray-based maps.
\end{example}

\begin{remarks}\label{ecremarks}\mbox{ \\ }
\begin{enumerate}[label=(\alph*)]
\item Reduced end cohomology is a proper homotopy invariant for ray-based, proper maps.
\item\label{exampleDSs} Reduced end cohomology generally depends on ray choice.
Let $S^2$ denote the $2$-sphere and $T^2$ denote the $2$-torus.
Consider the topological disjoint union $X=S(S^2) \sqcup S(T^2)$.
Let $r=\nnr\times\cpa{p} \subseteq S(S^2)$ and $s=\nnr\times\cpa{q} \subseteq S(T^2)$ be straight rays.
The nonzero $R$-modules in the end cohomology algebra of $X$ are
\begin{alignat*}{2}
\zeco{2}{X}	&\cong R	&&\times	R\\
\zeco{1}{X}	&\cong 0	&&\times	R^2\\
\zeco{0}{X}	&\cong R	&&\times	R
\end{alignat*}
where the cup product is coordinatewise in the direct product.
Therefore, we have
\begin{center}
\begin{tabular}{ c|c|c }
$k$ & $\zeco{k}{X,r}$ & $\zeco{k}{X,s}$ \\ \hline
$2$ & $R \times R \ $				& $R\times R \ $ \\
$1$ & $\ 0 \times R^2$			& $\ 0 \times R^2$ \\
$0$ &	$0 \times R \, $				& $R\times 0 \ \  $
\end{tabular}
\end{center}
For each integer $k$, the $R$-modules $\zeco{k}{X,r}$ and $\zeco{k}{X,s}$ are isomorphic.
Nevertheless, the $R$-algebras $\zeco{\ast}{X,r}$ and $\zeco{\ast}{X,s}$ are not isomorphic.
To see that, note that $(1,0)\in \zeco{0}{X,s}$ annihilates all of $\zeco{1}{X,s}$,
whereas $\zeco{0}{X,r}$ does not contain a nonzero element that annihilates all of $\zeco{1}{X,r}$.
\item For a connected example where ray choice alters the isomorphism type of the reduced end cohomology $R$-algebra, begin with $X$ from the previous example.
Attach a compact, oriented $1$-handle to the boundary sphere of $X$,
and then glue together by a homeomorphism the resulting two boundary tori.
That yields $Y$, a connected, two-ended $3$-manifold without boundary.
We have $\zeco{\ast}{Y,r} \cong \zeco{\ast}{X,r}$ and $\zeco{\ast}{Y,s} \cong \zeco{\ast}{X,s}$.
Hence, the isomorphism type of the reduced end cohomology $R$-algebra of $Y$ depends on ray choice.
\item For more intricate computations of end cohomology algebras, see~\cite{ch,cgh}.
\end{enumerate}
\end{remarks}

Next, we prove a splitting result for reduced end cohomology analogous to that for ordinary cohomology in~\eqref{reducedsplitting2}.

\begin{theorem}\label{mainrecotheorem}
Let $(X,r)$ be a ray-based, metrizable generalized continuum.
There is a splitting
\begin{equation}\label{recosplittingtheorem}
\zeco{\ast}{X} \cong \rzeco{\ast}{X} \eds \zeco{\ast}{\nnr}
\end{equation}
that is natural in $(X,r)$ for ray-based proper maps.
In each positive dimension $k$, $\zeco{k}{X} \cong \rzeco{k}{X}$.
\end{theorem}

\begin{proof}
By definition, $r:\nnr\rightarrowtail X$ is a proper, injective map.
By Theorem~\ref{retract}, there exists a proper map $\rho:X\twoheadrightarrow\nnr$ such that $\rho\circ r=\tn{id}:\nnr\to\nnr$.
Thus, $r^{\ast}_{\tn{e}} \circ \rho^{\ast}_{\tn{e}}=\tn{id}$, $\rho^{\ast}_{\tn{e}}$ is injective, and $r^{\ast}_{\tn{e}}$ is surjective.
Consider the long exact end cohomology sequence for the closed pair $(X,r)$.
As $\zeco{\ast}{\nnr}	=\zeco{0}{\nnr} \cong R$ and $r^{\ast}_{\tn{e}}$ is surjective,
we get $\zeco{k}{X} \cong \rzeco{k}{X}$ for each $k\in\Z_+$.
We also get the short exact sequence
\begin{equation}\label{eq:sesrecosplit}
\begin{tikzcd}
0 & \zeco{0}{r} \arrow[l] \arrow[r, dashed, bend right, "\rho^{\ast}_{\tn{e}}"]  & \zeco{0}{X} \arrow[l, "r^{\ast}_{\tn{e}}"'] & \rzeco{0}{X} \arrow[l] & 0 \arrow[l]
\end{tikzcd}
\end{equation}
with splitting homomorphism $\rho^{\ast}_{\tn{e}}$.
Thus, $\zeco{0}{X} \cong \rzeco{0}{X} \eds \zeco{0}{r}$ which proves~\eqref{recosplittingtheorem}.
It remains to prove naturality, which we address after the next two lemmas.
\end{proof}

Motivated by the fact---from the homotopical approach to ordinary cohomology discussed in
Section~\ref{sec:reducedcohomology}---that $\rzco{n}{X}\cong\br{X,K(R,n)}$,
we give in the next lemma a natural description of dimension-zero end cohomology.
Equip the ring $R$ with the discrete topology.
If $X$ is a space, then let $C(X,R)$ denote the $R$-algebra of maps $X\to R$.
Let $\mathds{1}_X:X\to R$ denote the constant map with image $1\in R$.
In case $X=\cpa{\bullet}$ is a one-point space---such as the end space of a ray---we
write $\mathds{1}_{\bullet}$ in place of $\mathds{1}_X$
and note that $\mathds{1}_{\bullet}$ is a canonical generator of $C(\bullet,R)\cong R$.
If $G=\fg{g_1,\ldots,g_n}$ is a free $\Z$-module with basis $\cpa{g_1,\ldots,g_n}$,
then let $\delta g_i$ denote the dual generator of the free $R$-module $\hom{}{G}{R}\cong\fg{\delta g_1,\ldots,\delta g_n}$ defined by
$g_i\mapsto 1$ and $g_j\mapsto 0$ for all $j \ne i$.

\begin{lemma}\label{dimzeroendcoho}
Let $X$ be a metrizable generalized continuum.
Then, there is an isomorphism $\eta:\zeco{0}{X} \to C(E(X),R)$.
\end{lemma}

\begin{proof}
By Theorem~\ref{gceee}, $X$ admits an efficient exhaustion by compacta $\cpa{K_i}$.
For each $i\in\Z_+$, let $V_i=X-K_i$ and write $\cpa{V_i^j}=\uc{V_i}$ for the finite (see Lemma~\ref{compactlemma}) set of unbounded components of $V_i$.
As $X$ is metrizable, Remark~\ref{mgcniceH0} implies that the components and the path-components of $V_i$ coincide.
So, $\zho{0}{V_i}=\fg{V_i^j}$ is the free $\Z$-module with basis $\cpa{V_i^j}$.
By the universal coefficient theorem, $\zco{0}{V_i}=\fg{\delta V_i^j}$ is the free $R$-module with basis $\cpa{\delta V_i^j}$.
Consider the direct system $\zco{0}{V_i}$ with direct limit $\zeco{0}{X}$, and
recall the basis $EV_k^l$ for the topology on $E(X)$---see~\eqref{basisE}.
For each $i\in\Z_+$, we define the homomorphism $\eta_i:\zco{0}{V_i}\to C(E(X),R)$ on the canonical
basis $\cpa{\delta V_i^j}$.
Namely, for each $j$ define $\eta_i\pa{\delta V_i^j}:E(X)\to R$ to be the map that sends each end $\e\in EV_i^j$ of $X$ to $1$
and sends all other ends of $X$ to $0$.
Notice that $\eta_i\pa{\delta V_i^j}$ is continuous since $R$ is discrete, the inverse image of $1$ is the basic open set $EV_i^j$,
and the inverse image of $0$ is the union of the basic open sets $EV_i^k$ where $k\ne j$.
The direct system $\zco{0}{V_i}$ together with the homomorphisms $\eta_i:\zco{0}{V_i}\to C(E(X),R)$ form a commutative diagram---to see that,
use Lemmas~\ref{monotonicity} (monotonicity) and~\ref{unbdedcomptoend}.
The universal property of direct limits yields a unique homomorphism $\eta:\zeco{0}{X} \to C(E(X),R)$ such that the whole diagram commutes.
For each unbounded component $V_i^j$ of $V_i$, Lemma~\ref{unbdedcomptoend} ensures the existence of an end of $X$ in the basic open set $EV_i^j$.
That implies each $\eta_i$ is injective.
Hence, $\eta$ is injective.
For surjectivity of $\eta$, let $\psi:E(X)\to R$ be continuous.
Recall that, by Theorem~\ref{endspace}, the space of ends $E(X)$ is compact.
As $R$ is discrete and $\psi:E(X)\to R$ is continuous, $\im{\psi}$ is compact and finite.
Taking the preimages of the points in $\im{\psi}$, we get a covering of $E(X)$ by finitely many disjoint open sets $U_1,\ldots,U_n$
and $\psi$ is constant on each $U_k$.
Say $\psi(U_k)= r_k \in R$ for each $k$.
Each $U_k$ is a union of basic open sets in $E(X)$.
As $E(X)$ is compact, finitely many $EV_{i_1}^{j_1},\ldots,EV_{i_m}^{j_m}$ of those basic open sets cover $E(X)$.
Let $i=\max\cpa{i_1,\ldots,i_m}$.
By Lemma~\ref{ebasislemma}, $K_i\subseteq X$ is a compactum such that, for each $j$, the basic open set $EV_i^j$
lies completely within a unique $U_k$, where $k=k(j)$ depends on $j$.
Define $z=\sum_j r_{k(j)} \delta V_i^j \in \zco{0}{V_i}$.
Let $\zeta\in \zeco{0}{X}$ be the element represented by $z$.
By our definition of $z$, $\eta_i(z)=\psi$.
By commutativity, $\eta\pa{\zeta}=\psi$ and $\eta$ is surjective.
Hence, $\eta$ is an isomorphism.

The argument above adapts readily to the case where $\cpa{K_i}$ is a (not necessarily efficient) compact exhaustion of $X$.
That case is necessary when we consider a proper map $f:Y\to X$ since the inverse image of an efficient compact
exhaustion of $X$ need not be an \textit{efficient} compact exhaustion of $Y$ by Example~\ref{invimceex}.
To adapt the argument, let $\cpa{K_i}$ be a compact exhaustion of $X$ and, for each $i\in\Z_+$, let $V_i=X-K_i$.
We have $\cpa{V_i^j}=\uc{V_i}$ the finite set of unbounded components of $V_i$ (see Lemma~\ref{compactlemma}),
$\cpa{B_i^k}=\bc{V_i}$ the (possibly infinite) set of bounded components of $V_i$, and $V_i=\uc{V_i}\sqcup\bc{V_i}$.
So, $\zho{0}{V_i}=\fg{V_i^j}\eds\fg{B_i^k}$ is a free $\Z$-module and
\[
\zco{0}{V_i}\cong \hom{}{\zho{0}{V_i}}{R} \cong \hom{}{\fg{V_i^j}}{R} \eds \hom{}{\fg{B_i^k}}{R}
\]
In the direct system $\zco{0}{V_i}$, the homomorphisms are restrictions induced by inclusions, and compactness implies
that each $\delta B_i^k\in \zco{0}{V_i}$ is eventually sent to $0$.
Thus, we define $\eta_i:\zco{0}{V_i}\to C(E(X),R)$ as above on each $\delta V_i^j$ and to send each $\delta B_i^k$ to $0$.
The direct system $\zco{0}{V_i}$ together with the homomorphisms $\eta_i:\zco{0}{V_i}\to C(E(X),R)$ form a commutative diagram
that induces a unique homomorphism $\eta:\zeco{0}{X} \to C(E(X),R)$ such that the whole diagram commutes.
Arguments similar to those above imply $\eta$ is an isomorphism.
\end{proof}

\begin{lemma}\label{dimzeroendcoho2}
Let $f:Y\to X$ be a proper map of metrizable generalized continua, then $f$ induces the $R$-algebra morphism $f^{\ast}:C(E(X),R) \to C(E(Y),R)$
given by $\psi\mapsto \psi\circ E(f)$.
In particular, $f^{\ast}\pa{\mathds{1}_{E(X)}}=\mathds{1}_{E(Y)}$.
\end{lemma}

\begin{proof}
By Theorem~\ref{gceee}, $X$ admits an efficient exhaustion by compacta $\cpa{K_i}$.
By Lemma~\ref{invimceisce}, $\cpa{L_i=f^{-1}\cpa{K_i}}$ is a (not necessarily efficient) compact exhaustion of $Y$.
Write $V_i=X-K_i$, $\cpa{V_i^j}=\uc{V_i}$, $U_i=Y-L_i$, and $\cpa{U_i^j}=\uc{U_i}$.
Consider the following commutative diagram where $f^{\ast}:C(E(X),R) \to C(E(Y),R)$
is defined by $f^{\ast}=\eta\circ f^{\ast}_{\tn{e}}\circ\eta^{-1}$
and $\eta$ is given by Lemma~\ref{dimzeroendcoho}.
By a common abuse of notation, $\eta$ is used twice to mean different isomorphisms, one for $X$ and one for $Y$.
\begin{equation}\label{eq:etanat}
\begin{tikzcd}
\zeco{0}{Y} \arrow[dr,"\eta","\cong"'] & \zco{0}{U_i} \arrow[l] \arrow[d,"\eta_i"'] & \zco{0}{V_i} \arrow[l,"f^{\ast}"'] \arrow[r] \arrow[d,"\eta_i"] & \zeco{0}{X} \arrow[lll,bend right=20,"f^{\ast}_{\tn{e}}"] \arrow[dl,"\eta"',"\cong"]\\
& C(E(Y),R)  &  C(E(X),R) \arrow[l,"f^{\ast}"']
\end{tikzcd}
\end{equation}
It remains to show that $f^{\ast}:C(E(X),R) \to C(E(Y),R)$ is given by $\psi\mapsto \psi\circ E(f)$.
Let $i\in\Z_+$. It suffices to consider the case where $\delta V_i^j$ is an element of the canonical basis of $\zco{0}{V_i}$
and $\psi= \eta_i\pa{\delta V_i^j}$.
Let $U_i^{m}$, $m\in M_j$, denote the (finite, possibly empty) collection of unbounded components of $U_i$ that $f$ maps into $V_i^j$.
So, $f^{\ast}:\zco{0}{V_i}\to \zco{0}{U_i}$ sends $\delta V_i^j\mapsto \sum_{m\in M_j} \delta U_i^m$.
By commutativity of~\eqref{eq:etanat}, $f^{\ast}:C(E(X),R) \to C(E(Y),R)$ sends
$\eta_i\pa{\delta V_i^j} \mapsto \sum_{m\in M_j} \eta_i\pa{\delta U_i^m}$.
By the definition of $\eta_i$, $\eta_i\pa{\delta V_i^j}$ sends each end of $X$ in $EV_i^j$ to $1$ and sends all other ends of $X$ to $0$.
Similarly, $\sum_{m\in M_j} \eta_i\pa{\delta U_i^m}$ sends each end of $Y$ in $\bigsqcup_{m\in M_j} U_i^m$ to $1$ and
sends all other ends of $Y$ to $0$.
By the definition of $E(f):E(Y)\to E(X)$ in (the proof of) Lemma~\ref{inducedmaps}, we have
$\sum_{m\in M_j} \eta_i\pa{\delta U_i^m}=\eta_i\pa{\delta V_i^j} \circ E(f)$.
That is, with $\psi= \eta_i\pa{\delta V_i^j}$ we have $f^{\ast}(\psi)=\psi\circ E(f)$, as desired.
That completes the proof of the first conclusion.
The second conclusion follows immediately from the first.
\end{proof}

With the previous two lemmas in hand, we prove that the splitting in Theorem~\ref{mainrecotheorem} is natural.
Let $(Y,s)$ be another ray-based, metrizable generalized continuum.
By Theorem~\ref{retract}, there exists a proper map $\sigma:Y\twoheadrightarrow\nnr$ such that $\sigma\circ s=\tn{id}:\nnr\to\nnr$.
Let $f:(Y,s)\to(X,r)$ be a ray-based proper map.
Consider the following diagram where each of the four outer trapezoids is given by Lemmas~\ref{dimzeroendcoho} and~\ref{dimzeroendcoho2} and thus commutes.
\begin{equation}\label{eq:splittingcommdiag}
\begin{tikzcd}
C(\bullet,R) \arrow[rrr,"i^{\ast}"] \arrow[ddd,"f^{\ast}"] & & & C(E(X),R) \arrow[ddd,"f^{\ast}"] \\
 & \zeco{0}{r} \arrow[r,"\rho^{\ast}_{\tn{e}}"] \arrow[d,"f^{\ast}_{\tn{e}}"] \arrow[ul,"\eta"',"\cong"] & \zeco{0}{X} \arrow[d,"f^{\ast}_{\tn{e}}"] \arrow[ur,"\eta","\cong"'] \\
 & \zeco{0}{s} \arrow[r,"\sigma^{\ast}_{\tn{e}}"] \arrow[dl,"\eta"',"\cong"] & \zeco{0}{Y} \arrow[dr,"\eta","\cong"'] \\
C(\bullet,R) \arrow[rrr,"i^{\ast}"] & & & C(E(Y),R)
\end{tikzcd}
\end{equation}
We use $\cpa{\bullet}$ to denote both one-point spaces $E(r)\subseteq E(X)$ and $E(s)\subseteq E(Y)$ where no confusion should arise.
It suffices to show that the central square in~\eqref{eq:splittingcommdiag} commutes.
By Lemma~\ref{dimzeroendcoho2}, the outer rectangle in~\eqref{eq:splittingcommdiag} satisfies
\begin{equation}\label{eq:outersquare}
\begin{tikzcd}
\mathds{1}_{\bullet} \arrow[r, mapsto] \arrow[d, mapsto] & \mathds{1}_{E(X)} \arrow[d, mapsto]\\
\mathds{1}_{\bullet} \arrow[r, mapsto] & \mathds{1}_{E(Y)}
\end{tikzcd}
\end{equation}
and thus commutes. It follows that the central square in~\eqref{eq:splittingcommdiag} commutes, as desired.
That completes our proof that the splitting in Theorem~\ref{mainrecotheorem} is natural.

\begin{remarks}\label{ecmstR}\mbox{ \\ } 
\begin{enumerate}[label=(\alph*)]
\item To summarize, recall the sequence~\eqref{eq:sesrecosplit}.
Write $\cpa{\bullet}=E(r)\subseteq E(X)$.
Let $C((E(X),\bullet),R)$ denote the elements of $C(E(X),R)$ that send $\bullet\mapsto0$.
Notice that $C((E(X),\bullet),R)$ is both an $R$-algebra and an ideal of $C(E(X),R)$.
We have the following commutative diagram with exact rows
\begin{equation}\label{eq:sesrecosplitC}
\begin{tikzcd}
0 & \zeco{0}{r} \arrow[l] \arrow[r, dashed, bend left, "\rho^{\ast}_{\tn{e}}"] \arrow[d,"\eta"',"\cong"] & \zeco{0}{X} \arrow[l, "r^{\ast}_{\tn{e}}"'] \arrow[d,"\eta"',"\cong"] & \rzeco{0}{X} \arrow[l] \arrow[d,"\eta"',"\cong"] & 0 \arrow[l]\\
0 & C(\bullet,R) \arrow[l] \arrow[r, dashed, bend right, "i^{\ast}"]  & C(E(X),R) \arrow[l, "r^{\ast}"'] \arrow[r, hookleftarrow,"j^{\ast}"] & C((E(X),\bullet),R) & 0 \arrow[l]
\end{tikzcd}
\end{equation}
where $r^{\ast}:\psi\mapsto \psi\circ E(r)$, $i^{\ast}:\mathds{1}_{\bullet}\mapsto \mathds{1}_{E(X)}$, 
and $j^{\ast}$ is inclusion.
The left square in~\eqref{eq:sesrecosplitC} that uses the two splitting homomorphisms (dashed) commutes.
Both splitting homomorphisms $\rho^{\ast}_{\tn{e}}$ and $i^{\ast}$ are natural with respect to ray-based proper maps.
In fact, those are the unique natural splitting homomorphisms---for uniqueness of $i^{\ast}$, consider naturality for $\rho:X\to\nnr$. 
\item Lemma~\ref{dimzeroendcoho} showed that $\zeco{0}{X} \cong C(E(X),R)$ for each metrizable generalized continuum $X$.
Thus, if $E(X)$ is finite, then $\zeco{0}{X}$ and $\zco{0}{E(X)}$ are isomorphic since both are isomorphic to $R^{\card{E(X)}}$.
However, if $E(X)$ is infinite, then $\zeco{0}{X}$ and $\zco{0}{E(X)}$ are not isomorphic.\footnote{Raymond erroneously
claimed~\cite[Thm.~1.13]{raymond} that $\zco{0}{E(X)}$ and $\zeco{0}{X}$ are isomorphic.
Apparently, the misstep occurred in the paragraph preceding that theorem where ``continuity'' meant cohomology commutes
with the inverse limit of topological inclusions. That is false by the infinite comb space example.}
For the infinite comb space $X$ from Examples~\ref{endspaceexamples} and shown in Figure~\ref{fig:combends}, $\zeco{0}{X} \cong C(E(X),\Z)$
is countably infinite, whereas
\[
\zco{0}{E(X)} \cong \hom{}{\zho{0}{E(X)}}{\Z} \cong \prod_{E(X)} \Z
\]
is uncountable---see also~\cite{schroeer} and~\cite[App.~A]{cgh}.

Continuing with the comb space example, let $E(X)=\cpa{\e_0,\e_1,\e_2,\ldots}$ as in Figure~\ref{fig:combends}.
\begin{figure}[htbp!]
    \centerline{\includegraphics[scale=1.00]{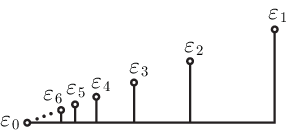}}
    \caption{Infinite comb space with ends labelled.}
\label{fig:combends}
\end{figure}
Let $\mathds{1}_{E(X)}:E(X)\to R$ be the constant $1$ map and, for each $i\in\Z_+$, let $\mathds{1}_{\cpa{\e_i}}:E(X)\to R$
denote the map sending $\e_i\mapsto 1$ and all other ends of $X$ to $0$.
As $E(X)$ is homeomorphic to $\cpa{0}\cup\cpa{1/i \mid i \in\Z_+}\subseteq\R$,
$\zeco{0}{X} \cong C(E(X),R)$ is the free $R$-algebra with basis $\cpa{\mathds{1}_{E(X)},\mathds{1}_{\cpa{\e_1}},\mathds{1}_{\cpa{\e_2}},\ldots}$.
\item A celebrated theorem of N\"obeling~\cite{nobeling} shows that if $A$ is a profinite space---meaning the limit
of an inverse system of finite sets---then $C(A,\Z)$ is a free $\Z$-module (see also Fuchs~\cite[Cor.~97.7]{fuchs}).
N\"obeling's theorem implies that $C(E(X),\Z)$ is a free $\Z$-module for each metrizable generalized continuum $X$.
In Theorem~\ref{recdirect} below, we give a very direct proof---not using N\"obeling's theorem---that $\zeco{0}{X}$ is a free $R$-module of countable rank for each metrizable generalized continuum $X$.
Thus, our work provides a proof of N\"obeling's theorem.
Namely, let $A=\ilim A_i$ be a profinite space where $i\in\Z_+$.
Replace each $A_i$ with the image of the projection $A\to A_i$ and restrict the bonding functions $A_i \leftarrow A_{i+1}$
to obtain an inverse system---still denoted $A_i$---of surjections with the same limit $A$.
Construct a locally finite simplicial tree $T$ with efficient compact exhaustion $\cpa{K_i}$ such that inverse system $\uc{T-K_i}$ in~\eqref{invsys} is isomorphic to $A_i$.
So, $E(T)\approx A$.
By Lemma~\ref{dimzeroendcoho}, $C(E(T),R) \cong \zeco{0}{T}$.
Thus
\[
C(A,R) \cong C(E(T),R) \cong \zeco{0}{T}
\]
where $\zeco{0}{T}$ is a free $R$-module of countable rank by Theorem~\ref{recdirect}.
That completes our proof of N\"obeling's theorem.
Our proof of N\"obeling's theorem constructs an explicit basis and does not use transfinite induction---compare~\cite[pp.~1--2]{asg}.

\item\label{rem:reasons} There are several reasons for requiring baserays to be embedded: topological simplicity, the end sum operation on manifolds in the next section,
the existence of a retract to each baseray as in Theorem~\ref{retract},
and the splitting in Theorem~\ref{mainrecotheorem}.
In fact, the splitting in Theorem~\ref{mainrecotheorem} can fail dramatically for \textit{nonembedded} rays as we now show.
Fix a point $p$ in the circle $S^1$.
Consider $Y=S(S^1)=\nnr\times S^1$ where $\eco{\ast}{Y}{R}	\cong	\co{\ast}{S^1}{R}$.
\begin{figure}[htbp!]
    \centerline{\includegraphics[scale=1.0]{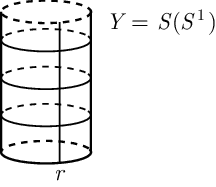}}
    \caption{Stringer $Y=S(S^1)$ on $S^1$ with (nonembedded) proper ray $r$.}
\label{fig:loopback}
\end{figure}
Let $r$ be a (nonembedded) proper ray in $Y$ that:
starts at $(0,p)$,
runs once around the circle $\cpa{0}\times S^1$,
runs along the segment $\br{0,1}\times\cpa{p}$,
runs once around the circle $\cpa{1}\times S^1$,
runs along the segment $\br{1,2}\times\cpa{p}$,
and so on, outwards towards the end of $Y$ as in Figure~\ref{fig:loopback}.
Using arguments as in~\cite{ch,cgh}, one may verify that $\eco{\ast}{Y,r}{R}=\eco{2}{Y,r}{R}\cong R\eqc{x}/R[x]$.
Here, $R\eqc{x}/R[x]$ is the $R$-module of formal power series modulo polynomials.
Thus, the conclusions of Theorem~\ref{mainrecotheorem} are radically false for $Y$ and the \textit{nonembedded} proper ray $r$.
\end{enumerate}
\end{remarks}

The remainder of this section is devoted to a direct proof of the following.

\begin{theorem}\label{recdirect}
Let $(X,r)$ be a ray-based, metrizable generalized continuum.
Then, there is an exact sequence of free $R$-modules of countable rank
\begin{equation}\label{eq:sesrecosplitdirect}
\begin{tikzcd}
0 & \zeco{0}{r} \arrow[l]  & \zeco{0}{X} \arrow[l] & \rzeco{0}{X} \arrow[l] & 0 \arrow[l]
\end{tikzcd}
\end{equation}
that splits naturally in $(X,r)$.
Further, $\card{E(X)}=\rank{\zeco{0}{X}}=\tn{rank}\rzeco{0}{X}+1$ (infinities not distinguished).
\end{theorem}

\begin{remark}
There is some overlap between Theorems~\ref{mainrecotheorem} and~\ref{recdirect}, although the conclusions and the approaches differ.
The fact that $\zeco{0}{X}$ is free of rank equal to the number of ends of $X$---infinities not distinguished---has been observed by others,
namely by Specker~\cite[$\S$5]{specker} for simplicial and CW complexes and integer coefficients,
Epstein~\cite[p.~110]{epstein} for locally finite simplicial complexes and field coefficients,
and Geoghegan~\cite[p.~298]{geoghegan} for strongly locally finite CW complexes and PID coefficients.
\end{remark}

Before we prove Theorem~\ref{recdirect}, we observe some useful algebraic facts.
The direct limit of a direct system of free $R$-modules of finite rank is a flat $R$-module by Lazard's theorem~\cite[p.~134]{lam},
however, it need not be a free $R$-module.
Consider the direct system
\begin{equation}\label{bad-direct-system}
\begin{tikzcd}
    \Z \arrow[r, "\times 2"] &
    \Z \arrow[r, "\times 2"] &
    \Z \arrow[r, "\times 2"] &
    \cdots
\end{tikzcd}
\end{equation}
with direct limit $\Z[1/2]$.
The element $1\in\Z[1/2]$ is divisible by arbitrarily large powers of $2$, so $\Z[1/2]$ is not a free $\Z$-module.
Thus, an additional condition is required to ensure freeness of a direct limit of free $R$-modules of finite rank.
One such condition---suitable for our purposes---is the following.
An injective homomorphism $f: G \to H$  of finitely generated $R$-modules is \deffont{extension preserving}
if whenever a set $S \subseteq G$ can be extended to a basis of $G$, then $f(S)$ can be extended to a basis of $H$.

\begin{lemma}\label{epequiv}
Let $f:G\to H$ be an injective homomorphism of finitely generated $R$-modules.
Then, $f$ is extension preserving if and only if there exists a basis $B$ of $G$ such that $f(B)$ extends to a basis of $H$.
\end{lemma}

\begin{proof}
The forward direction is clear.
For the reverse implication, we are given a basis $B$ of $G$ such that $f(B)$ extends to a basis $f(B)\sqcup C$ of $H$.
So, $\im{f}=\myspan{f(B)}$ and $H=\im{f}\eds \myspan{C}$---an internal direct sum of $R$-modules.
Consider a set $S \subseteq G$ that can be extended to a basis $S\sqcup T$ of $G$.
That implies $\im{f}=\myspan{f(S)\sqcup f(T)}$ and $f(S)\sqcup f(T) \sqcup C$ is a basis of $H$ extending $f(S)$, as desired.
\end{proof}

\begin{lemma}\label{lemma:diagonal_extension_preserving}
For each $n\in\Z_+$, the diagonal homomorphism $\Delta: R \to R^n$ defined by $1\mapsto(1,\ldots,1)$ is extension preserving.
\end{lemma}

\begin{proof}
Let $e_i$ denote the $i$th standard basis element of $R^n$.
Then, $\cpa{\Delta(1), e_2, e_3, \ldots, e_n}$ is a basis of $R^n$ and
the result follows by Lemma~\ref{epequiv}.
\end{proof}

\begin{lemma}\label{lemma:extension_preserving_direct_sum}
Given extension preserving homomorphisms $f_i: G_i \to H_i$ for $1 \leq i \leq n$,
define $G=\eds_{i=1}^n G_i$ and $H=\eds_{i=1}^n H_i$---external direct sums of $R$-modules.
Then, the homomorphism $f=\eds_{i = 1}^n f_i:G \to H$ is extension preserving.
\end{lemma}

\begin{proof}
For each $1\leq i \leq n$, let $B_i$ be a basis of $G_i$, and let $\varphi_i: G_i \to G$ and $\psi_i:H_i\to H$ be the canonical inclusions.
Notice that $f\circ \varphi_i=\psi_i\circ f_i$, and $B=\sqcup_{i=1}^n \varphi_i\pa{B_i}$ is a basis of $G$.
By hypothesis, $f_i\pa{B_i}$ extends to a basis $f_i\pa{B_i} \sqcup C_i$ of $H_i$.
Hence, $f(B)$ extends to a basis $f(B) \sqcup \psi_1\pa{C_1}\sqcup \cdots \sqcup \psi_n\pa{C_n}$ of $H$
and the result follows by Lemma~\ref{epequiv}.
\end{proof}

\begin{lemma}\label{lemma:basis-preserving-DL}
Let $G_1\xrightarrow{f_1}G_2\xrightarrow{f_2}G_3\xrightarrow{f_3}\cdots$
be a direct system of free $R$-modules of finite rank where each $f_i$ is extension preserving.
Then, the direct limit $G=\dlim G_i$ is a free $R$-module of countable rank.
\end{lemma}

\begin{proof}
The direct limit gives us a commutative diagram
\begin{equation} \label{DS-of-lemma}
\begin{tikzcd}
    G_1 \arrow[r, "f_1"] \arrow[rrrr, bend right=28, "\theta_1"'] &
    G_2 \arrow[r, "f_2"] \arrow[rrr, bend right=20, "\theta_2"'] &
    G_3 \arrow[r, "f_3"] \arrow[rr, bend right=10, "\theta_3"'] &
    \cdots &
		G=\dlim G_i
\end{tikzcd}
\end{equation}
As each $f_i$ is injective, the standard model of the direct limit implies that each $\theta_i$ is injective.
Inductively, we define a basis $A_i$ of $G_i$ and a subset $B_i\subseteq A_i$ for $i=1,2,3,\ldots$.
Let $A_1$ be a basis of $G_1$, and define $B_1=A_1$.
As $f_1$ is extension preserving, $f_1\pa{A_1}$ extends to a basis $A_2=f_1\pa{A_1} \sqcup B_2$ of $G_2$.
Given $A_i=f_{i-1}\pa{A_{i-1}} \sqcup B_i$ a basis of $G_i$,
$f_i$ is extension preserving and so $f_i\pa{A_i}$ extends to a basis $A_{i+1}=f_i\pa{A_i} \sqcup B_{i+1}$ of $G_{i+1}$.
That completes our definition of the bases $A_i$ of $G_i$ and the subsets $B_i\subseteq A_i$.
We define $\mathcal{B} = \cup_{i = 1}^\infty \theta_i\pa{A_i}$ and claim that
\begin{equation} \label{basis-of-G}
\mathcal{B} = \cup_{i = 1}^\infty \theta_i\pa{A_i} = \sqcup_{i = 1}^\infty \theta_i\pa{B_i} 
\end{equation}
is a basis of $G$.
Evidently, that claim proves the lemma since a countable collection of finite sets is countable.
To prove that claim, induction shows that $\cup_{i = 1}^n \theta_i\pa{A_i} = \cup_{i = 1}^n \theta_i\pa{B_i}$ for each $n\in\Z_+$.
Hence, $\cup_{i = 1}^\infty \theta_i\pa{A_i} = \cup_{i = 1}^\infty \theta_i\pa{B_i}$.
For disjointedness of the $\theta_i\pa{B_i}$, suppose, by way of contradiction, that $\theta_i\pa{b_i}=\theta_j\pa{b_j}$ for some $i<j$, $b_i\in B_i$, and $b_j\in B_j$.
Let $a=f_{j-1}\circ\cdots f_{i+1}\circ f_i (b_i) \in f_{j-1}\pa{A_{j-1}}\subseteq A_j-B_j$.
By commutativity, $\theta_j(a)=\theta_j(b_j)$.
As $\theta_j$ is injective, $a=b_j$.
However, $b_j\in B_j$ and $a\notin B_j$.
That contradiction completes the proof of disjointedness.
It remains to prove $\mathcal{B}$ is a basis of $G$.
To see $\mathcal{B}$ spans $G$, note that each element of $G$ has a representative in some $G_i$ and $A_i$ is a basis of $G_i$.
Finally, $\theta_1\pa{A_1}\subseteq \theta_2\pa{A_2}\subseteq\cdots$ is
a nested sequence of linearly independent sets---since each $A_i$ is linearly independent and each $\theta_i$ is injective---and so
$\mathcal{B}$ is linearly independent.
\end{proof}

Now, we prove Theorem~\ref{recdirect}.

\begin{proof}[Proof of Theorem~\ref{recdirect}]
We are given $(X,r)$, a ray-based, metrizable generalized continuum.
By Corollary~\ref{receexist}, there exists an $r$-efficient compact exhaustion $\cpa{K_i}$ of $X$.
For each $i\in \Z_+$, let $V_i=X-K_i$ and let $r_i=\im{r}-K_i$.
The components and the path-components of each $V_i$ coincide since $X$ is metrizable---see Remark~\ref{mgcniceH0}.
As $\cpa{K_i}$ is an $r$-efficient compact exhaustion of $X$, each $K_i$ is connected,
the components $\cpa{V_i^j}_{j = 1}^{\mu_i}$ of $V_i$ are all unbounded and finite in number (see Lemma~\ref{compactlemma}),
$r^{-1}\pa{K_i}$ is an efficient compact exhaustion of $\nnr$, and $r_i$ is connected.
We have the following commutative diagram of free $\Z$-modules of finite rank.
\begin{equation}\label{eq:inverse_system}
\begin{tikzcd}
    0 \arrow[r] & 
    \zho{0}{r_1} \arrow[r] & 
    \zho{0}{V_1} \arrow[r] \arrow[l, dashed, bend left, "s_1"] & \zho{0}{V_1,r_1} \arrow[r] & 0 \\
    0 \arrow[r] & 
    \zho{0}{r_2} \arrow[r] \arrow[u, "\cong"] & 
    \zho{0}{V_2} \arrow[r] \arrow[l, dashed, bend left, "s_2"] \arrow[u] & 
    \zho{0}{V_2,r_2}  \arrow[r] \arrow[u] & 0 \\
    & \vdots   \arrow[u, "\cong"]           
    & \vdots   \arrow[u]           
    & \vdots   \arrow[u] 
    & \\
\end{tikzcd}
\end{equation}
In~\eqref{eq:inverse_system}, the rows come from the long exact sequences of the pairs $(V_i,r_i)$ and the vertical homomorphisms
are induced by the inclusions $(V_1, r_1) \hookleftarrow (V_2, r_2) \hookleftarrow \cdots$.
We have $\zho{0}{V_i}=\fg{V_i^j}_{j=1}^{\mu_i}\cong\Z^{\mu_i}$ and $\zho{0}{r_i}=\fg{r_i}\cong \Z$.
Lemmas~\ref{monotonicity} (montonicity) and~\ref{unbdedcomptoend} imply that $1\leq \mu_1 \leq \mu_2 \leq \mu_3 \leq\cdots$
is a nondecreasing sequence of positive integers.
Those lemmas also imply that the vertical homomorphism $\zho{0}{V_i}\to\zho{0}{V_{i-1}}$
sends $V_i^j$ to the unique component of $V_{i-1}$ containing $V_i^j$ and that vertical homomorphism is surjective.
Essentially, that homomorphism is a finite direct sum of homomorphisms of the form $\Z^n\to\Z$ where $e_k\mapsto1$ for all $k$.
The unique natural splitting $s_i:\zho{0}{V_i}\to\zho{0}{r_i}$ of the $i$th row sends $V_i^j\mapsto r_i$ for all $j$.
That splitting is essentially $\Z^{\mu_i}\to\Z$ where $e_k\mapsto1$ for all $k$.
Diagram~\eqref{eq:inverse_system} commutes, including the splittings.

Apply the $\hom{\Z}{-}{R}$ functor and then the universal coefficients theorem to~\eqref{eq:inverse_system} to
obtain the following commutative diagram of free $R$-modules of finite rank.
\begin{equation}\label{eq:direct_system}
\begin{tikzcd}
0  & \zco{0}{r_1} \arrow[l] \arrow[r, dashed, bend right, "\sigma_1"'] \arrow[d, "\cong"'] & \zco{0}{V_1} \arrow[l] \arrow[d]  &
\zco{0}{V_1,r_1}  \arrow[l] \arrow[d] & 0 \arrow[l] \\
0  & \zco{0}{r_2} \arrow[l] \arrow[r, dashed, bend right, "\sigma_2"'] \arrow[d, "\cong"'] & \zco{0}{V_2} \arrow[l]  \arrow[d] &
\zco{0}{V_2,r_2} \arrow[l] \arrow[d] & 0 \arrow[l] \\
& \vdots & \vdots & \vdots & \\
\end{tikzcd}
\end{equation}
We obtain short exact sequences in each row by left exactness of the $\hom{\Z}{-}{R}$ functor
and since each row in~\eqref{eq:inverse_system} splits.
By the universal coefficients theorem, the homomorphisms in~\eqref{eq:direct_system} are dual to those in~\eqref{eq:inverse_system}.
In particular, for each $i\in\Z_+$ we have $\zco{0}{V_i}\cong\fg{\delta V_i^j}_{j=1}^{\mu_i}\cong R^{\mu_i}$ and $\zco{0}{r_i}\cong\fg{\delta r_i}\cong R$.
The homomorphism $\zco{0}{V_i}\to \zco{0}{V_{i+1}}$ sends $\delta V_i^j$ to the sum of the components of $V_{i+1}$ that $V_i^j$ contains.
Essentially, that homomorphism is a finite direct sum of diagonal homomorphisms $\Delta:R\to R^n$ sending $1\mapsto (1,\ldots,1)$.
By Lemmas~\ref{lemma:diagonal_extension_preserving} and~\ref{lemma:extension_preserving_direct_sum}, that homomorphism is extension preserving.
The unique natural splitting $\sigma_i:\zco{0}{r_i}\to\zho{0}{V_i}$ of the $i$th row sends $\delta r_i\mapsto \sum_{j=1}^{\mu_i} \delta V_i^j$.
That splitting is essentially the diagonal homomorphism $\Delta:R\to R^{\mu_i}$ sending $1\mapsto (1,\ldots,1)$.
Diagram~\eqref{eq:direct_system} commutes, including the splittings.

Apply the direct limit functor---which is exact---to~\eqref{eq:direct_system} to obtain the following exact sequence.
\begin{equation}\label{eq:sesrecosplitdirectrel}
\begin{tikzcd}
0 & \zeco{0}{r} \arrow[l]  & \zeco{0}{X} \arrow[l] & \zeco{0}{X,r} \arrow[l] & 0 \arrow[l]
\end{tikzcd}
\end{equation}
By our definition of reduced end cohomology, we have $\rzeco{0}{X}=\zeco{0}{X,r}$.
Therefore,~\eqref{eq:sesrecosplitdirectrel} is our desired exact sequence~\eqref{eq:sesrecosplitdirect}.
Lemma~\ref{lemma:basis-preserving-DL} implies that $\zeco{0}{X}$ is a free $R$-module of countable rank.
We know $\zeco{\ast}{r}=\zeco{0}{r}\cong R$.
So, the exact sequence~\eqref{eq:sesrecosplitdirect} implies that $\rzeco{0}{X}$ is a free $R$-module and
$\rank{\zeco{0}{X}}=\tn{rank}\rzeco{0}{X}+1$.
If $X$ has $k\in\Z_+$ ends, then the sequence $1\leq \mu_1 \leq \mu_2 \leq \mu_3 \leq\cdots$
eventually stabilizes at $k$ and $\zeco{0}{X}\cong R^k$.
If $X$ has infinitely many ends, then the sequence $1\leq \mu_1 \leq \mu_2 \leq \mu_3 \leq\cdots$ is unbounded.
As each $\zco{0}{V_i}$ injects into $\zeco{0}{X}$, we get that the rank of $\zeco{0}{X}$ is countably infinite.
The splittings $\sigma_i:\zco{0}{r_i}\to\zho{0}{V_i}$ yield a splitting $\zeco{0}{r}\to \zeco{0}{X}$ of~\eqref{eq:sesrecosplitdirectrel}.
For naturality of that splitting, let $f:(Y,s)\to(X,r)$ be a ray-based proper map of metrizable generalized continua.
Take the commutative diagram formed from~\eqref{eq:direct_system}, the analogous diagram for $(Y,s)$,
and the induced homomorphisms from the former to the latter.
Applying the direct limit functor yields the desired commutative diagram.
\end{proof}

Note that in Section~\ref{sec:cem} we introduced and proved the existence of $r$-efficient
compact exhaustions---and compact exhaustions of maps more generally---precisely
to make the algebra in the previous proof as simple as possible.
We leave the interested reader with an exercise:
construct an explicit basis for $\zeco{0}{X}\cong C(E(X),R)$ where $X$ is the infinite binary tree in Figure~\ref{fig:infbintree}.

\section{End cohomology of an end sum}
\label{sec:king}

End sum is an operation that combines two noncompact manifolds along an end of each manifold.
It was introduced by Gompf~\cite{gompf83}---then a graduate student---to construct various smooth manifolds homeomorphic but not diffeomorphic to $\R^4$.
End sum is now a major tool in $4$-manifolds and has applications in dimensions $2$, $3$, and
higher---see~\cite{gompf85} for $4$-manifolds, \cite{cg} and \cite{cgh} for background, manifolds of dimensions $\geq3$,
and further references, and~\cite{ac} for surfaces.
In this section, let $A \approx B$ mean that $A$ and $B$ are diffeomorphic---not necessarily preserving orientation---and
let $\Int{}{}$ denote manifold interior.

Colloquially, Gompf~\cite[p.~322]{gompf83} described end sum as gluing together two noncompact manifolds along an end of each manifold
using a \textit{piece of scotch tape}.
In more detail, let $M$ be a smooth, connected, oriented, noncompact manifold of dimension $n+1\geq 2$ with compact---possibly empty---boundary.
Let $r:\nnr\to M$ be a smoothly embedded proper ray in $\Int{M}{}$.
We call $(M,r)$ an \deffont{end sum pair}.
\begin{figure}[htbp!]
    \centerline{\includegraphics[scale=1.0]{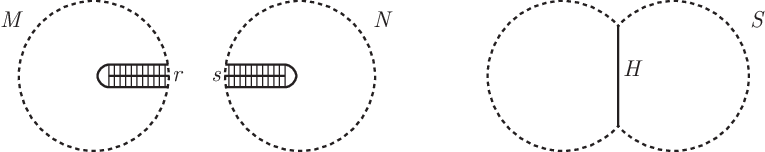}}
    \caption{End sum pairs $(M,r)$ and $(N,s)$ with regular neighborhoods $\nu r$ and $\nu s$ hatched (left) and
		end sum $S=(M,r)\es (N,s)$ (right).}
\label{fig:endsum}
\end{figure}
Let $(M,r)$ and $(N,s)$ be end sum pairs of the same dimension $n+1\geq 2$.
Let $\nu r\subset \Int{M}{}$ and $\nu s\subset \Int{N}{}$ be smooth, closed regular neighborhoods of $r$ and $s$ respectively.
Define $\widehat{M}=M-\Int{\nu r}{}$ and $\widehat{N}=N-\Int{\nu s}{}$.
Notice that $\widehat{M}$ contains $\partial \nu r \approx \R^n$ as a boundary component,
and similarly $\widehat{N}$ contains $\partial \nu s \approx \R^n$ as a boundary component.
Orient boundaries by the outward normal first convention~\cite[Ch.~3]{gp}.
As indicated in Figure~\ref{fig:endsum}, the \deffont{end sum} $S=(M,r)\es (N,s)$ of $(M,r)$ and $(N,s)$ is defined
to be the oriented manifold obtained by gluing together $\widehat{M}$ and $\widehat{N}$
along $\partial \nu r$ and $\partial \nu s$ by an orientation reversing diffeomorphism\footnote{For
details and variations on the end sum operation, see~\cite{cks}, \cite{cg}, and \cite[$\S$3]{ac}.
In particular, the latter includes a proof that proper data ensures the end sum is Hausdorff.}.
Let $H\subset S$ denote the properly embedded copy of $\R^n$ that is the common image of both $\partial\nu r$ and $\partial\nu s$.

The main goal of this section is to compute the end cohomology algebra of $S$
in terms of the algebras of $M$ and $N$.
Such a theorem was originally conceived by King (unpublished) during the collaboration~\cite{cks}.
Guilbault, Haggerty, and the second author~\cite[$\S$5]{cgh} carefully fixed orientation conventions and proved King's theorem.
At that time, a precise definition of reduced end cohomology was lacking---nevertheless, all applications in~\cite{cgh}
are valid by the proofs therein since only end cohomology elements of positive dimension were utilized.
We reuse all orientation conventions from~\cite[$\S$5]{cgh} and now employ
our precise definition of reduced end cohomology.
Throughout this section, $R$ is a PID.

The end sum $S$ is an adjunction space of $\widehat{M}$ and $\widehat{N}$.
We have smooth, proper embeddings
$i_M:\pa{\widehat{M},\partial \nu r} \to \pa{S,H}$ and
$i_N:\pa{\widehat{N},\partial \nu s} \to \pa{S,H}$.
Choose the orientation of $H$ so that the restriction
$\rest{i_M}\partial \nu r:\partial \nu r \to H$ is an orientation preserving diffeomorphism.
So, $\rest{i_N}\partial \nu s:\partial \nu s \to H$ is an orientation reversing diffeomorphism.
Choose $u\subset H\approx \R^n$ an unknotted smooth, proper ray.
Let $r'\subset \partial \nu r$ and $s'\subset \partial \nu s$ be the smooth, proper rays carried to $u$ by $i_M$ and $i_N$ respectively.
Thus, we have smooth, proper embeddings of closed triples
\begin{equation}\label{eq:triplesembed}
\begin{tikzcd}
\pa{\widehat{M},\partial \nu r, r'} \arrow[r,"i_M"] & \pa{S,H,u} & \pa{\widehat{N},\partial \nu s, s'} \arrow[l,"i_N"']
\end{tikzcd}
\end{equation}
as depicted in Figure~\ref{fig:endsum2}
\begin{figure}[htbp!]
    \centerline{\includegraphics[scale=1.0]{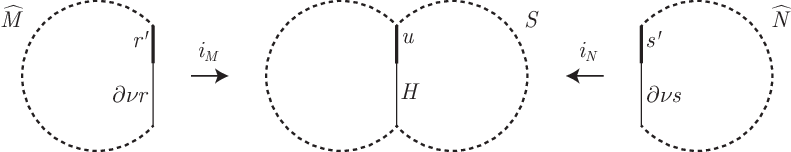}}
    \caption{Smooth, proper embeddings of closed triples.}
\label{fig:endsum2}
\end{figure}

The maps $i_M$ and $i_N$ induce the commutative diagram
\begin{equation}\label{eq:triplesexact}
\begin{tikzcd}
\phantom{} & \zeco{k}{\partial\nu r,r'} \arrow[l] & \zeco{k}{\widehat{M},r'} \arrow[l] & \zeco{k}{\widehat{M},\partial\nu r} \arrow[l] & \zeco{k-1}{\partial\nu r,r'} \arrow[l,"\delta_M"'] & \phantom{} \arrow[l]\\
\phantom{} & \zeco{k}{H,u} \arrow[l] \arrow[u] \arrow[d] & \zeco{k}{S,u} \arrow[l] \arrow[u] \arrow[d] & \zeco{k}{S,H} \arrow[l,"j^{\ast}_{\tn{e}}"'] \arrow[u,"i_M^{\ast}"] \arrow[d,"i_N^{\ast}"'] & \zeco{k-1}{H,u} \arrow[l,"\delta_S"'] \arrow[u,"\cong"'] \arrow[d,"\cong"] & \phantom{} \arrow[l]\\
\phantom{} & \zeco{k}{\partial\nu s,s'} \arrow[l] & \zeco{k}{\widehat{N},s'} \arrow[l] & \zeco{k}{\widehat{N},\partial\nu s} \arrow[l] & \zeco{k-1}{\partial\nu s,s'} \arrow[l,"\delta_N"'] & \phantom{} \arrow[l]
\end{tikzcd}
\end{equation}
where the rows are the long exact sequences of the closed triples~\eqref{eq:triplesembed}.
As $H\approx\R^n$, $n\geq1$, and $u$ is unknotted in $H$, the long exact sequence of the closed pair $(H,u)$
implies that $\zeco{n-1}{H,u}\cong R$ and $\zeco{k}{H,u}=\cpa{0}$ otherwise.
So, $j^{\ast}_{\tn{e}}: \zeco{k}{S,H}\to\zeco{k}{S,u}$ is an isomorphism in each dimension $k\ne n$ and $k\ne n-1$,
is surjective in dimension $n$, and is injective in dimension $n-1$.
It will turn out that $j^{\ast}_{\tn{e}}$ is an isomorphism in dimension $n-1$ as well.

The algebra $\zeco{\ast}{S,H}$ is related to algebras for $M$ and $N$ 
by a well-known commutative diagram that yields Mayer-Vietoris sequences---see \cite[p.~32]{eilenbergsteenrod},
\cite[pp.~110 \& 145]{maycc}, and \cite[pp.~470--471]{cgh}.
In particular, there is an isomorphism
\begin{equation}\label{eq:MV}
\begin{tikzcd}[row sep=0]
h: \zeco{\ast}{S,H} \arrow[r,"\cong"] & \zeco{\ast}{\widehat{M},\partial\nu r} \eds \zeco{\ast}{\widehat{N},\partial\nu s} \\
\alpha \arrow[r, mapsto] & \pa{i_M^{\ast}(\alpha),i_N^{\ast}(\alpha)}
\end{tikzcd}
\end{equation}
where the cup product is coordinatewise in the direct sum.
We also have isomorphisms
\begin{equation}\label{eq:seriesisos}
\begin{tikzcd}
\zeco{\ast}{\widehat{M},\partial\nu r} & \zeco{\ast}{M,\nu r} \arrow[l,"\cong"'] \arrow[r,"\cong"] & \zeco{\ast}{M,r} = \rzeco{\ast}{M}
\end{tikzcd}
\end{equation}
where the first follows from excision and the second is induced by inclusion---a proper homotopy equivalence since $\nu r$ nicely collapses to $r$.
Similarly, we have the analogous isomorphisms for $N$.
Hence, we have an isomorphism
\begin{equation}\label{eq:combineisos}
\begin{tikzcd}
\zeco{\ast}{S,H} \cong \zeco{\ast}{M,r} \eds \zeco{\ast}{N,s}
\end{tikzcd}
\end{equation}
We have shown that
\begin{equation}\label{eq:combineisosk}
\begin{tikzcd}
\zeco{k}{S,u} \cong \zeco{k}{M,r} \eds \zeco{k}{N,s}
\end{tikzcd}
\end{equation}
for each dimension $k\ne n$ and $k\ne n-1$.
To understand those remaining dimensions, we must look more closely at the coboundary homomorphism $\delta_S$ in~\eqref{eq:triplesexact}.
For that purpose, we recall notation and orientation conventions from~\cite[$\S$5]{cgh}.

By~\cite[Lemma~5.1]{cgh}, there is a smooth, proper Morse function $h:M\to \R$ that is bounded below and is well-behaved on $\nu r$.
Namely,
$\rest{h}r$ is projection,
$\rest{h}\nu r$ has a unique critical point that is a global minimum occurring in $\partial \nu r$,
$h^{-1}\pa{[t,\infty)}\cap \pa{\nu r,\partial \nu r}\approx [t,\infty)\times \pa{D^n,S^{n-1}}$ for each $t\in\nnr$,
each nonnegative integer is a regular value of $h$,
and the boundary of $M$---compact by hypothesis---is contained in $h^{-1}\pa{(-\infty,0)}$.
Such a Morse function is constructed using Whitney's embedding theorem and is depicted as height in Figure~\ref{fig:morse}.
\begin{figure}[htbp!]
    \centerline{\includegraphics[scale=1.0]{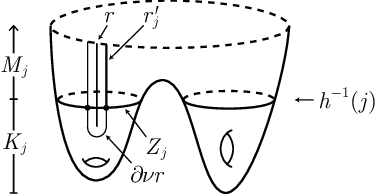}}
    \caption{Manifold $M$ with a Morse function $h$ depicted as height.}
\label{fig:morse}
\end{figure}
Figure~\ref{fig:triplesnotation} contains further submanifolds of $M$ that we now define.
\begin{figure}[htbp!]
    \centerline{\includegraphics[scale=1.0]{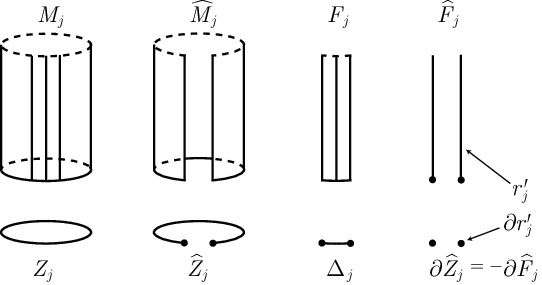}}
    \caption{Closed submanifold $M_j$ of $M$ and relevant closed submanifolds of $M_j$.}
\label{fig:triplesnotation}
\end{figure}

For each nonnegative integer $j$, we define:
\begin{align*}
K_j &= h^{-1}\pa{(-\infty,j]}\\
M_j &= h^{-1}\pa{[j,\infty)}\\
F_j &= \nu r \cap M_j \approx [j,\infty)\times D^n\\
Z_j &= \tn{the unique component of $M_j \cap K_j = h^{-1}(j)$ that meets $\nu r$}\\
\widehat{M}_j &= M_j-\Int{\nu r}{}\\
\widehat{F}_j &= \partial \nu r \cap M_j \approx [j,\infty)\times S^{n-1}\\
\widehat{Z}_j &= Z_j-\Int{\nu r}{}\\
r_j &= r \cap M_j\\
r'_j &= r' \cap M_j\\
\Delta_j &= \nu r \cap Z_j \approx D^n
\end{align*}
As an aid in assimilating that notation, recall the mnemonic from~\cite[p.~481]{cgh}: $\widehat{X}$ denotes a ``nicely punctured'' copy of $X$.
Orient $M_j$ and $K_j$ as codimension-0 submanifolds of $M$.
Note that $\cpa{K_j}_{j=0}^{\infty}$ is a compact exhaustion of $M$,
and $M_j \cap K_j = h^{-1}(j)$ is a nonempty, finite topological disjoint union of closed $n$-manifolds.
Equip $\partial K_j$ with the boundary orientation.
Orient $Z_j$ and $\widehat{Z}_j$ as codimension-0 submanifolds of $\partial K_j$.
Equip $\partial \widehat{Z}_j$ with the boundary orientation.
Note that $\partial\widehat{Z}_j\approx S^{n-1}$.

With integer coefficients, let $\br{\partial\widehat{Z}_j} \in \zho{n-1}{\partial\widehat{Z}_j}$ denote the fundamental class,
and let $\br{\widehat{Z}_j,\partial\widehat{Z}_j} \in \zho{n}{\widehat{Z}_j,\partial\widehat{Z}_j}$
denote the relative fundamental class---see Kreck~\cite[pp.~5--6]{kreck}.
If $n>1$, then both of those groups are infinite cyclic with preferred generators given by the fundamental classes,
and the long exact sequence of the pair $\pa{\widehat{Z}_j,\partial \widehat{Z}_j}$ contains
\begin{equation}\label{lepngt1}
\begin{tikzcd}[row sep=0]
\Z \cong \zho{n}{\widehat{Z}_j,\partial\widehat{Z}_j} \arrow[r,"\partial_{\ast}","\cong"'] & \zho{n-1}{\partial\widehat{Z}_j} \cong\Z \\
\br{\widehat{Z}_j,\partial\widehat{Z}_j} \arrow[r, mapsto] & \br{\partial\widehat{Z}_j}
\end{tikzcd}
\end{equation}
If $n=1$, then~\eqref{lepngt1} becomes
\begin{equation}\label{lepneq1}
\begin{tikzcd}[row sep=0]
\Z \cong \zho{1}{\widehat{Z}_j,\partial\widehat{Z}_j} \arrow[r,"\partial_{\ast}"] & \zho{0}{\partial\widehat{Z}_j} \cong \Z^2\\
\br{\widehat{Z}_j,\partial\widehat{Z}_j} \arrow[r, mapsto] & \br{\partial\widehat{Z}_j}
\end{tikzcd}
\end{equation}
where $\partial_{\ast}$ sends $1\mapsto(1,-1)$.
In any case $n\geq 1$, the long exact sequence of the triple $\pa{\widehat{Z}_j,\partial \widehat{Z}_j,\partial r'_j}$---see May~\cite[p.~110]{maycc}---contains
$\partial_{\ast}:\zho{n}{\widehat{Z}_j,\partial\widehat{Z}_j} \to \zho{n-1}{\partial\widehat{Z}_j,\partial r'_j}$.
By definition, that boundary homomorphism is the composition
\begin{equation}\label{compforb}
\begin{tikzcd}[row sep=0]
\zho{n}{\widehat{Z}_j,\partial\widehat{Z}_j} \arrow[r,"\partial_{\ast}"] & \zho{n-1}{\partial\widehat{Z}_j} \arrow[r] &
\zho{n-1}{\partial\widehat{Z}_j,\partial r'_j}
\end{tikzcd}
\end{equation}
of homomorphisms from the long exact sequences of the pairs $\pa{\widehat{Z}_j,\partial \widehat{Z}_j}$ and
$\pa{\partial \widehat{Z}_j,\partial r'_j}$.
It is straightforward to verify that overall composition is an isomorphism of copies of $\Z$.
So, we have
\begin{equation}\label{letnge1}
\begin{tikzcd}[row sep=0]
\Z \cong \zho{n}{\widehat{Z}_j,\partial\widehat{Z}_j} \arrow[r,"\partial_{\ast}","\cong"'] & \zho{n-1}{\partial\widehat{Z}_j,\partial r'_j} \cong \Z \\
\br{\widehat{Z}_j,\partial\widehat{Z}_j} \arrow[r, mapsto] & \br{\partial\widehat{Z}_j,\partial r'_j}
\end{tikzcd}
\end{equation}
where $\br{\partial\widehat{Z}_j,\partial r'_j}$ is defined to be the image of $\br{\widehat{Z}_j,\partial\widehat{Z}_j}$ under $\partial_{\ast}$.
By the universal coefficients theorem and taking the obvious duals,
the long exact sequence of the triple $\pa{\widehat{Z}_j,\partial \widehat{Z}_j,\partial r'_j}$ contains
\begin{equation}\label{lectnge1}
\begin{tikzcd}[row sep=0]
R \cong \zco{n}{\widehat{Z}_j,\partial\widehat{Z}_j} & \zco{n-1}{\partial\widehat{Z}_j,\partial r'_j} \cong R \arrow[l,"\delta"',"\cong"]\\
\br{\widehat{Z}_j,\partial\widehat{Z}_j}^{\ast} & \br{\partial\widehat{Z}_j,\partial r'_j}^{\ast} \arrow[l, mapsto]
\end{tikzcd}
\end{equation}

The inclusion maps of closed triples $\pa{\widehat{Z_j},\partial\widehat{Z_j},\partial r'_j}\hookrightarrow \pa{\widehat{M_j},\widehat{F_j},r'_j}$ yield the commutative diagram
\begin{equation}\label{eq:triplesexact2}
\begin{tikzcd}
\phantom{} & \zco{k}{\partial\widehat{Z}_j,\partial r'_j} \arrow[l] & \zco{k}{\widehat{Z}_j,\partial r'_j} \arrow[l] & \zco{k}{\widehat{Z}_j,\partial\widehat{Z}_j} \arrow[l] & \zco{k-1}{\partial\widehat{Z}_j,\partial r'_j} \arrow[l,"\delta"'] & \phantom{} \arrow[l]\\
\phantom{} & \zco{k}{\widehat{F}_j,r'_j} \arrow[l] \arrow[u,"\cong"'] & \zco{k}{\widehat{M}_j,r'_j} \arrow[l] \arrow[u] & \zco{k}{\widehat{M}_j,\widehat{F}_j} \arrow[l] \arrow[u] & \zco{k-1}{\widehat{F}_j,r'_j} \arrow[l,"\delta"'] \arrow[u,"\cong"'] & \phantom{} \arrow[l]
\end{tikzcd}
\end{equation}
where the rows are the long exact sequences of the closed triples.
The indicated vertical isomorphisms hold in all dimensions since
$\pa{\widehat{F}_j,r'_j}$ strong deformation retracts to $\pa{\partial\widehat{Z}_j,\partial r'_j}$---see Figure~\ref{fig:triplesnotation}.
In dimension $k=n$, the coboundary homomorphism~\eqref{lectnge1} is an isomorphism.
Thus, commutativity of~\eqref{eq:triplesexact2} implies that the coboundary homomorphism
$\delta:\zco{n-1}{\widehat{F}_j,r'_j} \to \zco{n}{\widehat{M}_j,\widehat{F}_j}$
is injective.
Taking the direct limit, we get that the coboundary homomorphism
$\delta_M: \zeco{n-1}{\partial\nu r,r'} \to \zeco{n}{\widehat{M},\partial\nu r}$ is injective.
Similarly, $\delta_N: \zeco{n-1}{\partial\nu s,s'} \to \zeco{n}{\widehat{N},\partial\nu s}$ is injective.
Commutativity of~\eqref{eq:triplesexact} and injectivity of $\delta_M$---or injectivity of $\delta_N$---imply that
$\delta_S: \zeco{n-1}{H,u} \to \zeco{n}{S,H}$ is injective.
Exactness of~\eqref{eq:triplesexact} and injectivity of $\delta_S$ imply that
$j^{\ast}_{\tn{e}}: \zeco{n-1}{S,H} \to \zeco{n-1}{S,u}$ is surjective.
So, $j^{\ast}_{\tn{e}}: \zeco{k}{S,H} \to \zeco{k}{S,u}$ is an isomorphism in each dimension $k\ne n$ and is surjective in dimension $n$.
Hence, $\zeco{\ast}{S,u}\cong \zeco{\ast}{S,H}/ K$ where $K$ is the kernel of $j^{\ast}_{\tn{e}}: \zeco{n}{S,H} \to \zeco{n}{S,u}$.
The kernel $K$ is an ideal of $\zeco{\ast}{S,H}$ and is principal
since---by exactness in~\eqref{eq:triplesexact}---$K$ equals the image of $\delta_S: \zeco{n-1}{H,u} \to \zeco{n}{S,H}$
and $\zeco{n-1}{H,u}\cong R$.
Recalling~\eqref{eq:MV} and~\eqref{eq:combineisos}, we wish to identify $K$ in $\zeco{\ast}{M,r} \eds \zeco{\ast}{N,s}$.

By~\eqref{lectnge1} and~\eqref{eq:triplesexact2}, the classes $\br{\partial\widehat{Z}_j,\partial r'_j}^{\ast}$
determine a preferred generator
\[
\br{\partial \nu r, r'}_{\tn{e}}^{\ast}\in \zeco{n-1}{\partial \nu r, r'}\cong \dlim_j \zco{n-1}{\widehat{F}_j,r'_j}  \cong R
\]
Let $\omega\in \zeco{n-1}{H,u}\cong R$ be the generator that is sent to
$\br{\partial \nu r, r'}_{\tn{e}}^{\ast}$ by $\rest{i_M}$.
So, $\rest{i_N}$ sends $\omega$ to $-\br{\partial \nu s, s'}_{\tn{e}}^{\ast}$.
Define $\br{r}_{\tn{e}}^{\ast}\in\zeco{n}{M,r}$---which we call a \deffont{ray-fundamental class}---to be
the image of $\br{\partial \nu r, r'}_{\tn{e}}^{\ast}$ under
the isomorphisms~\eqref{eq:seriesisos} and similarly for $N$.
Thus, we have a portion of~\eqref{eq:triplesexact} augmented with~\eqref{eq:seriesisos}
\begin{equation}\label{eq:keydiag}
\begin{tikzcd}
\zeco{n}{M,r} & & & \br{r}_{\tn{e}}^{\ast}\\
\zeco{n}{\widehat{M},\partial\nu r} \arrow[u, "\cong"', "\eqref{eq:seriesisos}"] & \zeco{n-1}{\partial\nu r,r'}\cong R \arrow[l, tail, "\delta_M"']
& & 
\delta_M\pa{\br{\partial \nu r, r'}_{\tn{e}}^{\ast}} \arrow[u, mapsto] &
\br{\partial \nu r, r'}_{\tn{e}}^{\ast} \arrow[l, mapsto]\\
\zeco{n}{S,H} \arrow[u,"i_M^{\ast}"] \arrow[d,"i_N^{\ast}"'] & \zeco{n-1}{H,u}\cong R \arrow[l, tail, "\delta_S"'] \arrow[u,"\cong"'] \arrow[d,"\cong"]
& &
\delta_S(\omega) \arrow[u, mapsto] \arrow[d, mapsto] & \omega \arrow[l, mapsto] \arrow[u,mapsto] \arrow[d,mapsto]\\
\zeco{n}{\widehat{N},\partial\nu s} \arrow[d, "\cong", "\eqref{eq:seriesisos}"'] & \zeco{n-1}{\partial\nu s,s'}\cong R \arrow[l, tail, "\delta_N"']
& &
\delta_N\pa{-\br{\partial \nu s, s'}_{\tn{e}}^{\ast}} \arrow[d, mapsto] & -\br{\partial \nu s, s'}_{\tn{e}}^{\ast} \arrow[l, mapsto]\\
\zeco{n}{N,s} & & & -\br{s}_{\tn{e}}^{\ast}
\end{tikzcd}
\end{equation}
By exactness $K=\im{\delta_S}$, and under the isomorphism~\eqref{eq:combineisos} $K$ corresponds to the principal ideal
$\fg{\pa{\br{r}^{\ast}_{\tn{e}}, -\br{s}^{\ast}_{\tn{e}}}}$ of dimension $n$.
Thus, we have proved the following version of King's theorem.

\begin{theorem}\label{thm:King}
Let $(M,r)$ and $(N,s)$ be end sum pairs of dimension $n+1\geq 2$.
Let $S=(M,r)\es (N,s)$ be the end sum of $(M,r)$ and $(N,s)$.
There exists an isomorphism of graded $R$-algebras
\begin{equation}\label{eq:redmainiso}
\rzeco{\ast}{S} \cong \pa{\rzeco{\ast}{M} \eds \rzeco{\ast}{N}}/\fg{\pa{\br{r}^{\ast}_{\tn{e}}, -\br{s}^{\ast}_{\tn{e}}}}
\end{equation}
where, by our definition, $\rzeco{\ast}{S}=\zeco{\ast}{S,u}$, $\rzeco{\ast}{M}=\zeco{\ast}{M,r}$, and $\rzeco{\ast}{N}=\zeco{\ast}{N,s}$.
\end{theorem}

Theorems~\ref{recdirect} and~\ref{thm:King} immediately imply the following.

\begin{corollary}\label{cor:numends}
Furthermore in Theorem~\ref{thm:King}, $\card{E(S)}=\card{E(M)}+\card{E(N)}-1$ (infinities not distinguished).
In particular, if $M$ and $N$ are one-ended, then $S$ is one-ended.
\end{corollary}

\begin{remarks}\label{kingrem}\mbox{ \\ } 
\begin{enumerate}[label=(\alph*)]
\item Regarding Corollary~\ref{cor:numends}, Axon and the second author~\cite[Lemma~5.9]{ac}
proved that $E(S)$ is homeomorphic to the quotient space of the topological disjoint union of $E(M)$ and $E(N)$
where the ends $\rho$ and $\sigma$---pointed to by $r$ and $s$ respectively---are identified.
\item The proof of our version of King's theorem is structured similarly to the proof in~\cite[$\S$5]{cgh} with the following distinctions:
we employed our precise definition of reduced end cohomology, used additional baserays $r'$ and $s'$, and correspondingly used triples of spaces.
Further, the seemingly two different conclusions in King's theorem~\cite[Thm.~5.4]{cgh} dissolved into a single conclusion
in our version.
\item In the construction of the end sum $S$, an unknotted ray $u$ was chosen in $H\approx \R^n$. 
When $n>1$, all such rays are ambiently isotopic in $H$.
When $n=1$, $H$ contains two such rays up to ambient isotopy.
Those rays need not be ambiently isotopic in $S$ either, although they do point to the same end of $S$
and they are related by a self-homeomorphism of $S$---see the proof of Lemma~5.9 and Theorem~6.10 in~\cite{ac}.
\item Ray-fundamental classes may be computed in some interesting cases---see~\cite[Ex.~5.3 \& $\S$6]{cgh}.
\end{enumerate}
\end{remarks}

\section{Future directions}
\label{sec:fd}

We conjecture that the main results of Section~\ref{sec:br}---existence of a baseray and a proper retract to each baseray---extend to trees.
A \deffont{tree} is a locally finite, acyclic simplicial $1$-complex.
Each tree is a generalized continuum.
A \deffont{leaf} is a vertex of degree $1$.
A tree is \deffont{rooted} provided it contains a distinguished vertex called the \deffont{root}.

\begin{conjecture}[Existence of nice embedded tree]
\label{conjextree}
Let $X$ be a noncompact, metrizable generalized continuum.
Then, there exists an infinite rooted tree $T$ and a proper (topological) embedding $\tau:T \rightarrowtail X$
such that $E(\tau):E(T)\to E(X)$ is a homeomorphism.
Further, $T$ may be chosen to contain no leaves or such that the root is the unique leaf.
\end{conjecture}

\begin{conjecture}[Existence of retract to tree]
\label{conjretree}
Let $X$ be a noncompact generalized continuum.
If $\tau:T \rightarrowtail X$ is a proper embedding of an infinite tree $T$ and $E(\tau)$ is injective,
then there exists a proper map $\rho:X\to T$ such that $\rho\circ \tau=\tn{id}:T \to T$.
In particular, $\tau\circ\rho:X\to\im{\tau}$ is a proper retraction of $X$ onto the image of $\tau$.
\end{conjecture}

Conjecture~\ref{conjextree} may require some additional niceness hypothesis on the space $X$.
Halin's tree lemma~\cite[pp.~63 \& 78]{bq} is a version of Conjecture~\ref{conjextree} for CW complexes.
A natural approach to Conjecture~\ref{conjretree} is to argue as in the ray version
but using the Tietze extension theorem with target a finite tree---which holds since each finite tree is a retract of the $2$-disk.
Proofs of versions of those two conjectures would immediately yield a splitting result more general than Theorem~\ref{mainrecotheorem}.
Guilbault and the second author will utilize such results for manifolds in a forthcoming paper to
prove that for each $k$, the $k$-dimensional end cohomology $R$-module of an end sum is independent up to isomorphism
of the chosen rays.
Thus, the ring structure is crucial for using end cohomology to distinguish various end sums of given manifolds,
as has been done in~\cite{ch,cgh}.

The dependence of end sum on ray choice has been well-studied---see
\cite{ch,cg,cgh}---yet interesting questions remain.
Guilbault, Haggerty, and the second author~\cite[Ques.~1.1]{cgh} have asked: \textit{for contractible, open $n$-manfiolds $M$ and $N$
of dimension $n\geq 4$, is $M\es N$ well-defined up to diffeomorphism or up to homeomorphism?}
Poincar\'{e} duality at the end---see Geoghegan~\cite[pp.~361--362]{geoghegan}---implies that the end cohomology algebra of a
contractible, open $n$-manifold is isomorphic to the ordinary cohomology algebra of $S^{n-1}$.
Thus, if one hopes to construct examples answering that question in the negative, then invariants other than end cohomology appear to be necessary.

It is interesting to ask whether end sums of two smooth, open, one-ended $n$-manifolds
can be homeomorphic but not diffeomorphic.
Examples for some $n\geq8$ were constructed by Gompf and the second author~\cite[Ex.~3.4(a)]{cg}.
It is unknown whether examples exist in dimension $n=4$.
A natural question is:
\textit{Let $X$ be a smooth, one-ended, oriented 4–manifold.
Can summing $X$ with a fixed exotic $\R^4$, preserving orientation,
yield different diffeomorphism types depending on the choice of ray in $X$?}
For more on that question, see~\cite[pp.~1303, 1325--1326]{cg}.

For a variety of additional open questions on ends---including interesting aspects
outside of the scope of the present paper---see Guilbault~\cite{guilbault}.

\end{document}